\documentclass[11pt,a4paper]{article}
\usepackage[T1]{fontenc}
\usepackage[francais,english]{babel}
\usepackage[applemac]{inputenc}
\usepackage{pb-diagram}
\usepackage{amscd}
\usepackage{enumerate}
\usepackage[vmargin=3.5cm]{geometry}%
\usepackage{color}%
\usepackage{mdwlist}%
\usepackage{graphicx}%
\usepackage{float}%
\usepackage{makeidx}%
\usepackage{amsmath}%
\usepackage{amssymb}%
\usepackage{amsthm}%
\usepackage{amsfonts}%
\usepackage{mathrsfs}%
\usepackage{bbm}%
\usepackage[all,cmtip]{xy}%
\usepackage{rotate}%
\usepackage{yfonts}%
\usepackage{array}%
\newtheorem{Prop}{Proposition}[section]%
\newtheorem{Theo}[Prop]{Theorem}
\newtheorem{Def}[Prop]{Definition}%
\newtheorem{Cor}[Prop]{Corollary}%
\newtheorem{Lem}[Prop]{Lemma}%
\newtheorem{Conj}[Prop]{Conjecture}%
\newtheorem{Hyp}[Prop]{Assumption}%
\newtheorem{DefEnglish}[Prop]{Definition}%[section]

\newtheorem{LemEnglish}[Prop]{Lemma}%[section]
\newtheorem{HypEnglish}[Prop]{Assumption}%
\newcommand{\A}{\mathbb A}%

\newcommand{\Bbf}{\mathbf B}
\newcommand{\C}{\mathbb C}%
\newcommand{\fp}{\mathbb F_{p}}%
\newcommand{\Fp}{\mathbb F}%

\newcommand{\G}{\GL_{2}}

\newcommand{\N}{\mathbb N}%
\newcommand{\Q}{\mathbb Q}%
\newcommand{\qp}{\mathbb Q_{p}}%
\newcommand{\ql}{\mathbb Q_{\ell}}%
\newcommand{\R}{\mathbb R}%
\newcommand{\Z}{\mathbb Z}%
\newcommand{\zp}{\mathbb Z_{p}}%
\newcommand{\Lcal}{\mathcal L}%
\newcommand{\Ncal}{\mathcal N}%
\newcommand{\Ocal}{\mathcal O}%
\newcommand{\Kcal}{\mathcal K}%
\newcommand{\Wcal}{\mathcal W}%
\newcommand{\Fcali}{\mathscr F}%
\newcommand{\Xcali}{\mathscr X}%
\newcommand{\aid}{\mathfrak a}%
\newcommand{\bid}{\mathfrak b}%
\newcommand{\cid}{\mathfrak c}%
\newcommand{\pid}{\mathfrak p}%
\newcommand{\mgot}{\mathfrak m}%
\newcommand{\GL}{\operatorname{GL}}%

\newcommand{\SL}{\operatorname{SL}}

\newcommand{\et}{\operatorname{et}}

\newcommand{\diamant}[1]{\langle#1\rangle}%
\newcommand{\somme}[2]{\underset{#1}{\overset{#2}\sum}}%
\newcommand{\produit}[2]{\underset{#1}{\overset{#2}\prod}}%
\newcommand{\produittenseur}[2]{\underset{#1}{\overset{#2}\bigotimes}}%
\newcommand{\sommedirecte}[2]{\underset{#1}{\overset{#2}\bigoplus}}%
\newcommand{\suiteexacte}[5]{0\fleche#3\overset{#1}{\fleche}#4\overset{#2}{\fleche}#5\fleche0}
\newcommand{\cardinal}[1]{|#1|}

\newcommand{\limproj}[1]{\underset{\underset{#1}\longleftarrow}\lim}

\newcommand{\isom}{\overset{\sim}{\longrightarrow}}
\newcommand{\image}{\operatorname{Im}}

\newcommand{\plonge}{\hookrightarrow}

\newcommand{\tors}{\operatorname{tors}}%
\newcommand{\ord}{\operatorname{ord}}
\newcommand{\mot}{\operatorname{mot}}

\newcommand{\tenseur}{\otimes}

\newcommand{\Ltenseur}{\overset{\operatorname{L}}{\tenseur}}
\newcommand{\modulo}{\operatorname{ mod }}
\newcommand{\Spec}{\operatorname{Spec}}

\newcommand{\Id}{\operatorname{Id}}%
\newcommand{\Frac}{\operatorname{Frac}}%
\newcommand{\Tate}{\operatorname{Ta}}%
\newcommand{\Cone}{\operatorname{Cone}}%
\newcommand{\fleche}{\longrightarrow}%
\newcommand{\partiediv}{\operatorname{div}}%
\newcommand{\croix}{^{\times}}%
\newcommand{\surjection}{\twoheadrightarrow}%
\newcommand{\rhobar}{\bar{\rho}}%
\newcommand{\vide}{\varnothing}%

\newcommand{\sm}{\operatorname{sm}}%
\newcommand{\dR}{\operatorname{dR}}%
\newcommand{\Fil}{\operatorname{Fil}}

\newcommand{\Hun}{H^{1}}

\newcommand{\Htilde}{\tilde{H}}

\newcommand{\Htildeun}{\tilde{H}^{1}}
\newcommand{\Htildetildeun}{\tilde{\tilde{H}}^{1}}

\newcommand{\RGamma}{\operatorname{R}\Gamma}%
\newcommand{\Det}{\operatorname{{D}et}}%
\newcommand{\Tam}{\operatorname{Tam}}%
\newcommand{\Sel}{\operatorname{Sel}}%

\newcommand{\tr}{\operatorname{tr}}
\newcommand{\Classe}{\operatorname{Cl}}%
\newcommand{\Fr}{\operatorname{Fr}}%
\newcommand{\Gal}{\operatorname{Gal}}

\newcommand{\Qbar}{\bar{\Q}}%
\newcommand{\qpbar}{\Qbar_{p}}%
\newcommand{\Fpbar}{\bar{\mathbb F}}%
\newcommand{\Dbar}{\bar{D}}
\newcommand{\kg}{\kappa}%
\newcommand{\la}{\lambda}
\newcommand{\La}{\Lambda}%
\newcommand{\s}{\sigma}%
\newcommand{\Si}{\Sigma}
\newcommand{\hgot}{\mathfrak h}%
\newcommand{\Hecke}{\mathbf{T}}%
\newcommand{\Hs}{\Hecke^{\Sigma}}
\newcommand{\Eul}{\operatorname{Eul}}
\newcommand{\eqdef}{\overset{\operatorname{def}}{=}}

\DeclareFontEncoding{OT2}{}{} % to enable usage of cyrillic fonts
  \newcommand{\textcyr}[1]{%
    {\fontencoding{OT2}\fontfamily{wncyr}\fontseries{m}\fontshape{n}%
     \selectfont #1}}%
\newcommand{\Sha}{{\mbox{\textcyr{Sh}}}}%
\newcommand{\Nekovar}{Nekov\'a\v{r}}%
\newcommand{\Iwagood}{Iwasawa-suitable\ }%
\newcommand{\TW}{\mathbf{TW}}
\newcommand{\PD}{\mathbf{PD}}

\numberwithin{equation}{subsection}%
\date{}
\begin{document}%
\title{The Equivariant Tamagawa Number Conjectures for modular motives with coefficients in Hecke algebras}
\author{Olivier Fouquet}%
\maketitle

%
%\selectlanguage{francais}%
\selectlanguage{english}%

\newcommand{\hord}{\mathfrak h^{\ord}}%
\newcommand{\hdual}{\mathfrak h^{dual}}%
\newcommand{\matricetype}{\begin{pmatrix}\ a&b\\ c&d\end{pmatrix}}%
\newcommand{\Iw}{\operatorname{Iw}}%
\newcommand{\psiwa}{\psi_{\operatorname{Iw}}}%
\newcommand{\Lambdaf}{\mathbf{\Lambda}}
\newcommand{\Hi}{\operatorname{Hi}}
\newcommand{\cyc}{\operatorname{cyc}}
\newcommand{\ab}{\operatorname{ab}}
\newcommand{\can}{\operatorname{can}}%
\newcommand{\Fitt}{\operatorname{Fitt}}%
\newcommand{\Tfiwa}{T_{f,\operatorname{Iw}}}%
\newcommand{\Vfiwa}{V_{f,\operatorname{Iw}}}%
\newcommand{\Afiwa}{A^{*}_{f,\operatorname{Iw}}(1)}%
%SectionTour
\newcommand{\Af}{\operatorname{A}}%
\newcommand{\Dunzero}{D_{1,0}}%
\newcommand{\Uun}{U_{1}}%
\newcommand{\Uzero}{U_{0}}%
\newcommand{\Uundual}{U^{1}}%
\newcommand{\Uunun}{U^{1}_{1}}%
\newcommand{\Wdual}{\Wcal^{dual}}
%SectionCM
\newcommand{\JunNps}{J_{1,0}(\Ncal, P^{s})}%
\newcommand{\Tatepord}{\Tate_{\pid}^{ord}}%
\newcommand{\Kum}{\operatorname{Kum}}%
\newcommand{\zcid}{z(\cid)}%
\newcommand{\kgtilde}{\tilde{\kappa}}%
\newcommand{\kiwa}{\varkappa}%
\newcommand{\kiwatilde}{\tilde{\varkappa}}%
\newcommand{\Hbar}{\bar{H}}%
\newcommand{\Tred}{T/\mgot T}%
\newcommand{\Riwa}{R_{\operatorname{Iw}}}%
\newcommand{\Oiwa}{\Ocal_{\operatorname{Iw}}}%
\newcommand{\Kiwa}{\Kcal_{\operatorname{Iw}}}%
\newcommand{\Sp}{\Sigma^{p}}%
\newcommand{\Ap}{A_{\pid}}%
\newcommand{\Liwa}{\Lambda_{\operatorname{Iw}}}%
\newcommand{\liwa}{\lambda_{\operatorname{Iw}}}%
\newcommand{\liwaf}{\lambda(f)_{\operatorname{Iw}}}%
\newcommand{\Viwa}{\mathcal V_{\operatorname{Iw}}}%
\newcommand{\pseudiso}{\overset{\centerdot}{\isom}}%
\newcommand{\pseudisom}{\overset{\approx}{\fleche}}%
\newcommand{\carac}{\operatorname{char}}%
\newcommand{\length}{\operatorname{length}}
\newcommand{\eord}{e^{\ord}}%
\newcommand{\eordm}{e^{\ord}_{\mgot}}%
\newcommand{\hordinfini}{\hord_{\infty}}%
\newcommand{\Mordinfini}{M^{\ord}_{\infty}}%
\newcommand{\hordm}{\hord_{\mgot}}%
\newcommand{\hminm}{\hgot^{min}_{\mgot}}%
\newcommand{\Mordm}{M^{\ord}_{\mgot}}%
\newcommand{\Mtwist}{M^{tw}_{\mgot}}
\newcommand{\Xun}{X_{1}}%
\newcommand{\Xundual}{X^{1}}%
\newcommand{\Xunun}{X^{1}_{1}}%
\newcommand{\Xtw}{X^{tw}}
\newcommand{\Inert}{\mathfrak{In}}%
\newcommand{\Ts}{T_{\Sigma}}
\newcommand{\Tsprime}{T_{\Sigma'}}
\newcommand{\Tla}{T_{\la}}
\newcommand{\Vlaiwa}{V_{\la,\Iw}}
\newcommand{\Tpsi}{T_{\psi}}
\newcommand{\Tpsiwa}{T_{\psi,\Iw}}
\newcommand{\Vpsiwa}{V_{\psi,\Iw}}
\newcommand{\Vla}{V_{\la}}
\newcommand{\rla}{\rho_{\la}}
\newcommand{\Tsx}{T_{\Sigma}^{x_{1}}}
\newcommand{\Vs}{V_{\Sigma}}
\newcommand{\Rs}{R_{\Sigma}}
\newcommand{\Rsr}{R_{\Sigma}(\rhobar)}
\newcommand{\Xsr}{\Xcali_{\Sigma}(\rhobar)}
\newcommand{\Xsm}{\Xcali^{\sm}_{\Sigma}(\rhobar)}
\newcommand{\Rsrchi}{R_{\Sigma}^{\chi}(\rhobar)}
\newcommand{\Raid}{R(\aid)}
\newcommand{\Faid}{\Frac(\aid)}%
\newcommand{\Taid}{T(\aid)}
\newcommand{\Vaid}{V(\aid)}
\newcommand{\zaid}{\z(\aid)}
\newcommand{\Hsm}{\mathbf{T}_{\mgot_{\rhobar}}^{\Sigma}}
\newcommand{\Hsaid}{\mathbf{T}(\aid)}
\newcommand{\Hsmprime}{\mathbf{T}_{\mgot_{\rhobar}}^{\Sigma'}}
\newcommand{\Hsmprimeprime}{\mathbf{T}_{\mgot_{\rhobar}}^{\Sigma''}}
\newcommand{\mr}{\mgot_{\rhobar}}
\newcommand{\Hsmr}{\mathbf{T}_{\mgot_{\rhobar}}^{\Sigma}}
\newcommand{\Hsmrord}{\mathbf{T}_{\mgot_{\rhobar}}^{\Sigma,\ord}}
\newcommand{\Hsmx}{\mathbf{T}_{\mgot_{\rhobar}}^{\Sigma,x_{1}}}
\newcommand{\Hsx}{\mathbf{T}_{\mgot_{\rhobar}}^{\Sigma,x_{1}}}
\newcommand{\Hr}{\mathbf{T}_{\mgot_{\rhobar}}^{\Sigma}}
\newcommand{\Ds}{\Delta_{\Sigma}}
\newcommand{\Dsprime}{\Delta_{\Sigma'}}
\newcommand{\Tsp}{T(\psi)}%
\newcommand{\Asp}{A_{\Sp}}%
\newcommand{\Vsp}{V_{\Sp}}%
\newcommand{\SK}{\mathscr{S}}%
\newcommand{\Rord}{R^{\ord}}%
\newcommand{\per}{\operatorname{per}}
\newcommand{\z}{\mathbf{z}}
\newcommand{\vbf}{\mathbf{v}}
\newcommand{\vbfbar}{\bar{\mathbf{v}}}
\newcommand{\vbfd}{{\mathbf{v}^{\vee}}}
\newcommand{\zs}{\mathbf{z}_{\Sigma}}
\newcommand{\zsprime}{\mathbf{z}_{\Sigma'}}
\newcommand{\zsl}{\mathbf{z}_{\Sigma,\Lambdaf}}
\newcommand{\zsf}{\mathbf{z}(f)_{\Sigma}}
\newcommand{\ztilde}{\tilde{\z}}
\newcommand{\ziwa}{\z(f)_{\Iw}}
\newcommand{\zsiwa}{\mathbf{z}_{\Sigma,\Iw}}
\newcommand{\zsfiwa}{\mathbf{z}(f)_{\Sigma,\Iw}}
\newcommand{\zka}{{\mathbf{z}_{\operatorname{Ka}}}}
\newcommand{\Ebarbar}{\bar{\bar{E}}}
\newcommand{\Grsym}{\mathfrak S}
\newcommand{\Gr}{\operatorname{Gr}}
\newcommand{\epsi}{\varepsilon}
\newcommand{\Fun}[2]{F^{{\mathbf{#1}}}_{#2}}
\newcommand{\triv}{\operatorname{triv}}
\newtheorem*{TheoA}{Theorem A}%
\newtheorem*{TheoB}{Theorem B}%
\newtheorem{ConjIMC}{Iwasawa Main Conjecture}[section]%
\newtheorem*{ConjETNC}{Theorem B}%%\newtheorem{ConjArticle}{Conjecture}%
\newcommand{\isocan}{\overset{\can}{\simeq}}
\newcommand{\cusps}{\operatorname{cusps}}
\newcommand{\ad}{\operatorname{ad}}
\newcommand{\Lie}{\operatorname{Lie}}
\makeatletter
\newcommand\@biprod[1]{%
  \vcenter{\hbox{\ooalign{$#1\prod$\cr$#1\coprod$\cr}}}}
\newcommand\biprod{\mathop{\mathpalette\@biprod\relax}\displaylimits}
\makeatother
\newcommand{\Etilde}{\tilde{E}}
\newcommand{\Atilde}{\tilde{A}}
\newcommand{\Btilde}{\tilde{B}}
\newcommand{\Var}{\operatorname{Var}}
\newcommand{\Rep}{\operatorname{Rep}}
\newcommand{\Xbar}{\bar{X}}
\newcommand{\Spa}{\operatorname{Spa}}
\newcommand{\Is}{\operatorname{Is}}
\newcommand{\der}{\operatorname{der}}
\newcommand{\CompCoho}{\Htildeun\left(U^{p},\Ocal\right)_{\mgot_{\rhobar}}}
\newcommand{\CompCoh}{\Htildetildeun\left(U^{p},\Ocal\right)_{\mgot_{\rhobar}}}
\newcommand{\CompCohoC}{\Htildeun_{c}\left(U^{p},\Ocal\right)_{\mgot_{\rhobar}}}
\newcommand{\Uaid}{U_{1}(\aid)}
\newcommand{\Uaidl}{U_{1}(\aid)_{\ell}}
\newcommand{\zt}{\tilde{\z}}
\newcommand{\zaidt}{\tilde{\z}(\aid)}
\abstract{Modular motives have coefficients in Hecke algebras. According to the equivariant philosophy, special values of $L$-functions of eigencuspforms should therefore exhibit equivariant properties with respect to various Hecke actions. This manuscript shows that this is indeed the case at least under broad conditions on ramification and deduce from them new properties of the Iwasawa Main Conjecture for modular forms.}

\tableofcontents
%
%
%\begin{center}
%\textbf{Résumé}
%\end{center}
%\paragraph{} Au début des années 2000, Ralph Greenberg a demandé si l'on pouvait démontrer la Conjecture Principale de la théorie d'Iwasawa pour une famille de Hida de formes modulaires paraboliques propres quasi-ordinaires en la propageant par congruences à partir d'un cas connu. Emerton-Pollack-Weston ont montré que c'était effectivement possible lorsque l'invariant $\mu$ de cette famille est trivial. Dans cet article, nous montrons que c'est le cas sans hypothèse supplémentaire, et que ce résultat demeure en fait vrai sur l'anneau de déformation universelle d'une représentation résiduelle modulaire irréductible.
\section{Introduction}
When $p=157$, the Iwasawa cohomology module
\begin{equation}\nonumber
H^{2}_{\Iw}\left(\Z[1/p],\zp(1)\right)\eqdef\limproj{n}\Classe(\Z[\zeta_{p^{n}},1/p])[p^{\infty}]
\end{equation}
is a non-trivial, finitely generated, torsion $\zp[[\Gal(\Q(\zeta_{p^{\infty}})/\Q)]]$-module whose structure is easily described: it is a direct sum of two eigenspaces for the action of $\Gal(\Q(\zeta_{p})/\Q)$ which are both free $\zp$-modules of rank 1 with an explicit action of $\Gal(\Q(\zeta_{p^{\infty}})/\Q(\zeta_{p}))$. These statements belong to classical Iwasawa theory, and they are easily obtained from the fact that $157$ divides $B_{62}$ and $B_{110}$ but no other even Bernoulli number $B_{i}$ for $2\leq i\leq 156$ using the existence of the Kubota-Leopoldt $p$-adic zeta function $\zeta_{p}$ (\cite{KubotaLeopoldt}) and the Iwasawa Main Conjecture (\cite{IwasawaMainConjecture,MazurWiles}).

Suppose now that $p=5$ and that we are asking similar questions for the Iwasawa cohomology module
\begin{equation}\nonumber
H^{2}_{\Iw}\left(\Z[1/p],T_{p}E\right)\eqdef\limproj{n}\ H^{2}_{\et}\left(\Z[\zeta_{p^{n}},1/p],j_{*}T_{p}E\right)
\end{equation}
where $T_{p}E$ is the $p$-adic Tate module of the rational elliptic curve
\begin{equation}\nonumber
E:y^{2}=x^{3}-1825x+306625
\end{equation}
and where $j:\Spec\Q(\zeta_{p^{n}})\fleche\Spec\Z[\zeta_{p^{n}},1/p]$ is the natural morphism. Note that $E$ plainly has bad additive reduction at $p$. Hence, previous general results on the Iwasawa Main Conjecture for elliptic curves do not apply to $E$ (moreover, experts will notice that $E$ has no auxiliary prime with residual Steinberg reduction whereas this is a necessary hypothesis to apply the main current theorems on the Iwasawa Main Conjecture such as \cite{SkinnerUrban,FouquetWan}). Besides, computational exploration of Iwasawa-theoretic properties of $E$ through its $p$-adic $L$-function is impossible, since such a $p$-adic $L$-function is presently lacking at primes of additive reduction. Consequently, questions about the structure of $H^{2}_{\Iw}\left(\Z[1/p],T_{p}E\right)$ are considered hard.

This manuscript formulates and proves versions of the Iwasawa Main Conjecture which are compatible with congruences in $p$-adic families of eigencuspforms (conjecture \ref{ConjUniv} and theorem \ref{TheoIntro} below). A concrete illustration of its main results is as follows: the Iwasawa Main Conjecture (\cite{KatoHodgeIwasawa,KatoEuler}) holds for $E$ and $p$ and $H^{2}_{\Iw}\left(\Z[1/p],T_{p}E\right)$ is a free $\zp$-module of rank 1 with prescribed $\Gal(\Q(\zeta_{p^{\infty}})/\Q(\zeta_{p}))$-action (see section \ref{SubExample}). Unlike the setting of classical Iwasawa theory of the first paragraph, these statement may not be deduced from divisibility properties of special values of the $p$-adic $L$-function, since no such object is known to exist. To bypass the appeal to the $p$-adic $L$-function, level-lowering and level-raising congruences between $E$ and various eigencuspforms of different levels are used, explaining the need for a conjectural framework which is compatible with congruences. 

I learned the theory of $p$-adic families of automorphic representations and the associated deformation theory of Galois representations from \cite{BellaicheChenevier}. The careful study of the Galois-module structure of cohomology groups in analytic families of automorphic representations is the crux of \cite{BellaicheRank} and the idea of exploiting congruences in $p$-adic families of eigencuspforms to go beyond known properties of $p$-adic $L$-functions is the key idea of \cite{BellaicheEigen}. Above all, I found in myself the strength to learn the general theory of special values of $L$-functions of motives in great part thanks to warm personal encouragements provided at a critical point during my studies. I hope it is fitting that this manuscript be dedicated to the memory of the author of these works, the person who so kindly advised me then and thereby encouraged me to pursue research in the arithmetic of special values of $L$-functions: Joël Bellaïche. 

\subsection{Equivariant conjectures for modular motives}
Let $r\in\Z$, $k\geq2$ and $N\geq1$ be three integers. Let $f\in S_{k}(\Gamma_{1}(N))$ be an eigencuspform which we assume to have rational $q$-expansion in this introduction for simplicity of exposition and let $L(f,s)$ be the $L$-function of $f$. In this manuscript, it is consistently assumed that $L(f,r)$ is non-zero. In \cite{SchollMotivesModular}, T. Scholl constructed a Grothendieck motive $\Wcal(f)(r)$ attached to $f$. Let $V(f)_{r,p}$ be the $p$-adic étale realization $\Wcal_{\et,p}(f)(r)$ of $\Wcal(f)(r)$. The Tamagawa Number Conjecture (\cite{BlochKato,FontaineValeursSpeciales}) for the motive $\Wcal(f)(r)$ is the following statement.
\begin{Conj}\label{ConjTNCIntro}
There exists a specified \emph{fundamental line} $\left(\Delta_{\Q}(f),\per_{\C},\{\per_{p}\}_{p}\right)$ where $\Delta_{\Q}(f)$ is a one-dimensional $\Q$-vector space and $\per_{\C}$ and $\per_{p}$ for $p$ a prime are canonical isomorphisms
\begin{equation}\nonumber
\per_{\C}:\Delta_{\Q}(f)\tenseur_{\Q}\C\isocan\C,\ \per_{p}:\Delta_{\Q}(f)\tenseur_{\Q}\qp\isocan\Det^{-1}_{\qp}\RGamma_{c}\left(\Z[1/Np],V(f)_{r,p}\right)
\end{equation}
as well as a \emph{zeta element} $\z(f)\in\Delta_{\Q}(f)$ verifying the following properties.
\begin{enumerate}
\item
\begin{equation}\nonumber
\per_{\C}\left(\z(f)\tenseur1\right)=L(f,r).
\end{equation}
\item For all prime $p$, the $\zp$-module
\begin{equation}\nonumber
\per_{p}\left(\zp\cdot\z(f)\tenseur1\right)\subset\Det^{-1}_{\qp}\RGamma_{c}\left(\Z[1/Np],V(f)_{r,p}\right)
\end{equation}
is equal to $\Det^{-1}_{\zp}\RGamma_{c}\left(\Z[1/Np],T(f)\right)$ where $T(f)$ is any $\zp$-lattice inside $V(f)_{r,p}$ which is stable under the action of $\Gal(\Qbar/\Q)$.
\end{enumerate}
\end{Conj}
Conjecture \ref{ConjTNCIntro} is said to hold at a prime  $p$ if its first statement holds and the second holds for this prime $p$. As reviewed below, many cases of conjecture \ref{ConjTNCIntro} at $p$ are known.

By its very construction, the motive $\Wcal(f)(r)$ is a submotive of a motive $\Wcal(r)$ on which the Hecke algebra $\Hecke$ generated by Hecke operators outside $N$ acts. The equivariant philosophy of K. Kato (\cite{KatoViaBdR,FukayaKato}) suggests that this supplementary action of $\Hecke$ should be reflected in the behavior of special values of $L$-functions as predicted by conjecture \ref{ConjTNCIntro}. More specifically, the following statement should hold.
\begin{Conj}\label{ConjETNCIntro}
Denote by $Q(\Hecke)$ the total ring of quotients of $\Hecke$. There exists a free rank one $\Hecke$-module $\Delta(\Wcal(r))$ and a \emph{Hecke-equivariant zeta element} $\z\in\Delta_{Q(\Hecke)}(\Wcal(r))$ where
\begin{equation}\nonumber
\Delta_{Q(\Hecke)}(\Wcal(r))\eqdef\Delta(\Wcal(r))\tenseur Q(\Hecke)
\end{equation} satisfying the following properties.
\begin{enumerate}
\item The Hecke-equivariant zeta element $\z$ is a basis of $\Delta(\Wcal(r))$.
\item For all system of eigenvalues $\lambda(f):\Hecke\fleche\Qbar$ attached to an eigencuspform $f$, the natural map from $\Wcal(r)$ onto $\Wcal(f)(r)$ canonically induces an isomorphism
\begin{equation}\nonumber
\Delta(\Wcal(r))\tenseur_{\Hecke,\lambda(f)}\Qbar\isocan\Delta_{\Q}(f)\tenseur\Qbar
\end{equation}
which sends $\z\tenseur1$ to $\z(f)\tenseur1$.
\item For all prime $p$, there is a canonical isomorphism 
\begin{equation}\nonumber
\Delta(\Wcal(r))\tenseur_{\Hecke}\Hecke_{p}\isocan\Det^{-1}_{\Hecke_{p}}\RGamma_{c}\left(\Z[1/Np],T\right)
\end{equation}
where $\Hecke_{p}$ is $\Hecke\tenseur_{\Z}\zp$ and $T$ is any $\Gal(\Qbar/\Q)$-stable $\Hecke_{p}$-lattice inside the $p$-adic étale realization of $\Wcal(r)$ seen as $\Hecke\tenseur_{\Q}\qp$-module.
\end{enumerate}
\end{Conj}
The following equivalent reformulation of conjecture \ref{ConjETNCIntro} in terms of morphisms rather than element is convenient.
\begin{Conj}\label{ConjETNCIntroBis}
There exists a \emph{zeta morphism}
\begin{equation}\nonumber
\z:\Delta_{Q(\Hecke)}(\Wcal(r))\isom Q(\Hecke)
\end{equation}
satisfying the following properties.
\begin{enumerate}
\item The zeta morphism induces an isomorphism $\z:\Delta(\Wcal(r))\isom\Hecke$.
\item Let $\lambda(f):\Hecke\fleche\Qbar$ be system of eigenvalues. Let $\Delta_{\Qbar}(f)$ be $\Delta_{\Q}(f)\tenseur\Qbar$ and let $\z(f):\Delta_{\Qbar}(f)\isom\Qbar$ be the isomorphism sending the basis $\z(f)$ to 1. The  the diagram
\begin{equation}\nonumber
\xymatrix{
\Delta_{\Hecke}\left(\Wcal(r)\right)\ar[d]_{-\tenseur_{\lambda}\Qbar}\ar[r]^(0.6){\z}&\Hecke\ar[d]^{-\tenseur_{\lambda}\Qbar}\\
\Delta_{\Qbar}(f)\ar[r]_(0.6){\z(f)}&\Qbar
}
\end{equation}
is commutative.
\item For $T$ as in conjecture \ref{ConjETNCIntro}, put $\Delta_{\Hecke_{p}}=\Det^{-1}_{\Hecke_{p}}\RGamma_{c}\left(\Z[1/Np],T\right)$. Then the $p$-adic étale realization map induces an isomorphism $\Delta(\Wcal(r))\tenseur_{\Hecke}\Hecke_{p}\isocan\Delta_{\Hecke_{p}}$.
\end{enumerate}
\end{Conj} 
It is important to remark that conjectures \ref{ConjETNCIntro} and \ref{ConjETNCIntroBis} are logically \emph{much stronger} (and correlatively much more delicate to prove) than conjecture \ref{ConjTNCIntro} or even that the collection of conjectures \ref{ConjTNCIntro} for all the submotives of $\Wcal(r)$.

Conjecture \ref{ConjETNCIntroBis} generalizes to $p$-adic families of modular motives. The set-up is as follows. Let $\Si$ be a finite set of finite primes containing $\{\ell|Np\}$. Let $G_{\Q,\Si}$ be the Galois group of the maximal extension of $\Q$ unramified outside $\Si$. Let $\Rs$ be a complete local noetherian ring of mixed characteristic $(0,p)$ and $\Ts$ be a free $R_{\Si}$-module of rank 2 endowed with a continuous $G_{\Q,\Si}$-action. To a point
\begin{equation}\nonumber
\la:\Rs\fleche A
\end{equation}
is attached the $A[G_{\Q,\Si}]$-module $T_{\la}\eqdef\Ts\tenseur_{\Rs,\la}A$. We assume that there exists a set $\Spec^{\mot}R_{\Si}$ of motivic points of $\Rs$ satisfying the following property. If $\Ocal$ is the ring of integers of a finite extension of $\qp$ and if $\la:\Rs\fleche\Ocal$ is in $\Spec^{\mot}R_{\Si}$, there is an eigencuspform $f\in S_{k}(\Gamma_{1}(N))$ and $r\in\Z$ such that $T_{\la}$ is an $\Ocal$-lattice inside $V(f)_{r,p}\tenseur\qpbar$. We then say that $\la$ is attached to the motive $\Wcal(\la)=\Wcal(f)(r)$. The following is then a formulation of the generalized Iwasawa Main Conjecture with coefficients in $\Rs$ for $\Ts$ (\cite[Conjecture 3.2.2]{KatoViaBdR}).
\begin{Conj}\label{ConjIMCIntro}
Define
\begin{equation}\nonumber
\Ds\eqdef\Det^{-1}_{\Rs}\RGamma_{c}\left(\Z[1/\Si],\Ts\right).
\end{equation}
For all $\lambda:\Rs\fleche A$ in $\Spec\Rs$, put
\begin{equation}\nonumber
\Delta_{\la}\eqdef\Det^{-1}_{A}\RGamma_{c}\left(\Z[1/\Si],T_{\la}\right).
\end{equation}
If $\la:\Rs\fleche\Ocal$ is in $\Spec^{\mot}\Rs$, we denote by $\z(f):\Delta_{\la}\isom\Ocal$ the isomorphism given by the zeta element $\z(f)$ of conjecture \ref{ConjETNCIntroBis}.

Then, there exists a \emph{zeta morphism}
\begin{equation}\nonumber
\zs:\Ds\tenseur Q(\Rs)\isom Q(\Rs)
\end{equation}
satisfying the following properties.
\begin{enumerate}
\item The zeta morphism induces an isomorphism $\zs:\Ds\isom\Rs$.
\suspend{enumerate}
\resume{enumerate}
%\item Then the natural map
%\begin{equation}\nonumber
%\Ds\tenseur_{\Rs,\la}A\fleche\Delta_{\la}
%\end{equation}
%is an isomorphism.
\item Let $\la:\Rs\fleche\Ocal$ be in $\Spec^{\mot}\Rs$. The following diagram of natural maps
\begin{equation}\nonumber
\xymatrix{
\Ds\ar[d]_{-\tenseur_{\lambda}\Ocal}\ar[r]^{\zs}&\Ts\ar[d]^{-\tenseur_{\lambda}\Ocal}\\
\Delta_{\la}\ar[r]_{\z(f)}&\Ocal
}
\end{equation}
is commutative.
\end{enumerate}
\end{Conj}
Roughly speaking, conjecture \ref{ConjIMCIntro} says that the family of zeta morphisms $\zs$ defined over $\Rs$ interpolates the zeta morphisms $\z(f)$ at all motivic point $f$ of $\Rs$.

In the second statement of conjecture \ref{ConjIMCIntro} is implicit the fact that the natural map
\begin{equation}\label{EqIsomDeltaIntro}
\Ds\tenseur_{\Rs,\la}A\fleche\Delta_{\la}
\end{equation}
is an isomorphism for all $\la\in\Spec^{\mot}\Rs$. Under our hypothesis that $\Si$ contains all primes dividing $Np$, this indeed holds even without the hypothesis that $\la$ is motivic (however, the map \eqref{EqIsomDeltaIntro} need not be an isomorphism in general without this hypothesis on $\Si$, so that \cite[Conjecture 3.2.2]{KatoViaBdR}, which seemingly omits it, appears to be not correct as stated - see for instance the counterexamples in \cite[section 4.2]{FouquetX}).

The reader is advised to keep the following fundamental examples in mind.
\begin{enumerate}[(i)]
\item Suppose $\Rs=\zp$, $\Ts=T$ is a lattice inside $V(f)_{r,p}$ and $\Spec^{\mot}\Rs=\{\Id\}$. Then $\Ts$ is a punctual family. Conjecture \ref{ConjIMCIntro} is then a restatement of conjecture \ref{ConjTNCIntro} at the prime $p$. 
\item Suppose $\Rs$ is the localization of $\Hecke_{p}$ at a maximal prime ideal $\mgot$, $\Ts$ is an $\Rs$-lattice $T$ inside $\Wcal(r)\tenseur\qp$ and $\Spec^{\mot}\Rs$ is the set of system of eigenvalues which factor trough $\Hecke_{\mgot}$. Then $\Ts$ is a finite collection of points. Conjecture \ref{ConjIMCIntro} is then a restatement of the $\mgot$-local part of conjecture \ref{ConjETNCIntro}. 
\item Suppose that $\Rs$ is the Iwasawa algebra $\La_{\Iw}\eqdef\zp[[\Gal(\Q_{\infty}/\Q)]]$ where $\Q_{\infty}/\Q$ is the cyclotomic $\zp$-extension, that $\Ts$ is $T\tenseur_{\zp}\La_{\Iw}$ for $T\subset V(f)_{r,p}$ as above and that $\la$ belongs to $\Spec^{\mot}\Rs$ if and only if there exist $r\in\Z$ and $\chi:\Gal(\Q_{\infty}/\Q)\fleche\C$ such that $L(f,\chi,r)\neq0$ and such that $T_{\la}$ is a lattice inside $V(f)(r)\tenseur\chi$. Then $\Ts$ is the $p$-adic family of Galois representations attached by classical Iwasawa theory. Conjecture \ref{ConjIMCIntro} is then a statement of the Iwasawa Main Conjecture for modular motives (\cite[Conjecture 12.10]{KatoEuler}). In this setting, conjecture \ref{ConjIMCIntro} is frequently known to hold.
\item A refined variation of the previous examples is where $\Rs$ is $\Hecke_{\mgot}[[\Gal(\Q_{\infty}/\Q)]]$ for $\mgot$ a maximal ideal of $\Hecke_{p}$. Then $\Ts$ is a finite collection of classical Iwasawa-theoretic families attached to finitely many eigencuspforms (this is to example (iii) what example (ii) is to example (i)). Despite the superficial similarity, conjecture \ref{ConjIMCIntro} in this setting is much more delicate than in the previous one.
\item Suppose that $\Rs$ is the local Hida-Hecke algebra $\Hecke_{\Si}^{\ord}$ attached to a maximal ideal of the Hida-Hecke algebra, that $\Ts$ is the Galois representation attached to a Hida family of nearly-ordinary eigencuspform and that $\Spec^{\mot}\Rs$ is the set of system of eigenvalues attached to nearly-ordinary eigencuspforms $f$ of weight $k\geq2$ possibly twisted by a character $\chi_{\cyc}^{r}\chi$ with $r\in\Z$ and $\chi:\Gal(\Q_{\infty}/\Q)\fleche\C$ such that $L(f,\chi,r)\neq0$. In that setting, conjecture \ref{ConjIMCIntro} is a formulation of the Iwasawa Main Conjecture in Hida families. When $\Rs$ is a in addition a normal ring, conjecture \ref{ConjIMCIntro} is equivalent to the statement
\begin{equation}\nonumber
\left(\Lcal_{\Si,p}\right)=\carac_{\Rs}\Sel_{\Gr}\left(\Ts\right)
\end{equation}
where $\Lcal_{\Si,p}$ is the $p$-adic $L$-function of \cite{EmertonPollackWeston} and $\Sel_{\Gr}\left(\Ts\right)$ is a suitably defined $\Rs$-torsion Greenberg Selmer module (see for instance \cite{OchiaiMainConjecture,FouquetOchiai}).
\item Finally, the maximally general example is the case where $\Rs$ is the universal deformation ring $\Hsmr$ of a residually modular absolutely irreducible Galois representation $\rhobar$, where $\Ts$ is the universal deformation of $\rhobar$ and where $\Spec^{\mot}\Hsmr$ is the set of system of eigenvalues attached to eigencuspforms $f$ of weight $k\geq2$ possibly twisted by a character $\chi_{\cyc}^{r}\chi$ with $r\in\Z$ and $\chi:\Gal(\Q_{\infty}/\Q)\fleche\C$. It logically follows from statement 3 of conjecture \ref{ConjIMCIntro} for $\Ts$, that this conjecture implies conjecture \ref{ConjIMCIntro} at $p$ for $T_{\la}$ at all motivic points $\la$ of $\Rs$ such that $L(f,\chi,r)\neq0$ (so example (i) for these modular motives). It is true but not obvious that conjecture \ref{ConjIMCIntro} for $\Ts$ implies conjecture \ref{ConjIMCIntro} for $T_{\la}\tenseur\La_{\Iw}$ at all motivic points $\la$ (so example (iii)), conjecture \ref{ConjIMCIntro} with coefficients in $\Hecke_{\mgot}[[\Gal(\Q_{\infty}/\Q)]]$ whenever $\Hecke_{\mgot}$ is a quotient of $\Rs$ as in example (iv) and conjecture \ref{ConjIMCIntro} for a Hida family parametrized by $\Hecke_{\mr}^{\Si,\ord}$ as in example (v) if $\Rs$ admits $\Hecke_{\mr}^{\Si,\ord}$ as a quotient (the difficulty in establishing this statement being that the compatibility at motivic points does not formally entail the compatibility in families). More generally, theorem \ref{TheoCompatibility} below shows that this compatibility holds at the so-called \Iwagood specalization of $\Hsmr$.
\end{enumerate}
\subsection{Statement of results}
Assume that $p\geq3$ is an odd prime and let $k$ be a finite field of characteristic $p$. As above, let $\Sigma$ be a finite set of primes containing $p$ and let $G_{\Q,\Sigma}$ be the Galois group of the maximal extension of $\Q$ unramified outside $\Sigma$. Consider
\begin{equation}\nonumber
\rhobar:G_{\Q,\Sigma}\fleche\GL_{2}(k)
\end{equation}
an odd Galois representation which satisfies the following properties.
\begin{Hyp}\label{HypIntro}
\begin{enumerate}
\item The image of $\rhobar$ contains a subgroup conjugated to $\SL_{2}(\fp)$.
\item\label{ItemrhobarIntroEnP} If $\rhobar|_{G_{\qp}}$ is an extension of characters
\begin{equation}\nonumber
\suiteexacte{}{}{\chi}{\rhobar|_{G_{\qp}}}{\psi},
\end{equation}
then $\chi\psi^{-1}\neq1$ and $\chi\psi^{-1}\neq\bar{\chi}_{\cyc}^{-1}$ (here $\chi_{\cyc}$ is the $p$-adic cyclotomic character).
\end{enumerate}
\end{Hyp}
To $\rhobar$ is attached a universal deformation ring $\Rsr$ and a universal deformation $\Ts$. Define $\Rsr$ to be \emph{minimal} if for all motivic point $\la:\Rsr\fleche\Ocal$ the $G_{\Q,\Sigma}$-representation $\rho_{\la}$ is a \emph{minimal lift} in the sense of definition \cite[Definitions 3.3 and 3.32]{FujiwaraDeformation} (in particular, $\rhobar$ is then ramified at all $\ell\nmid p$ in $\Sigma$).
\begin{Theo}\label{TheoMinIntro}
Assume $\Rsr$ to be minimal. Then the following assertions are equivalent.
\begin{enumerate}
\item There exists a motivic point $\la\in\Spec^{\mot}\Rsr$ such that conjecture \ref{ConjIMCIntro} holds for $T_{\la}\tenseur\La_{\Iw}$.
\item For all motivic point $\la\in\Spec^{\mot}\Rsr$, conjecture \ref{ConjIMCIntro} holds for $T_{\la}\tenseur\La_{\Iw}$ (the classical Iwasawa Main Conjecture is true for the motive $\Wcal(f)(r)$). If in addition $L(f,\chi,r)\neq0$, then conjecture \ref{ConjTNCIntro} for $\Wcal(\la)$ at $p$ holds (the Tamagawa Number Conjecture is true for the motive $\Wcal(f)(r)$ at $p$).
\item The Iwasawa Main Conjecture (conjecture \ref{ConjIMCIntro}) holds for $\Ts$.
\end{enumerate}
\end{Theo}
When $\Rsr$ is not minimal, the statement of the main theorem is more involved, as it requires a technical condition on the supplementary ramification allowed. Here follows a summary (see assumption \ref{HypLocalDef} in section \ref{SubLevels} below for details).

Denote by $\Si(\rhobar)$ the finite set of finite primes at which $\rhobar$ is ramified and let $U^{(p)}\subset\GL_{2}(\A_{\Q}^{(p\infty)})$ be a compact open subgroup maximal outside $\Sigma\supset\Sigma(\rhobar)\cup\{p\}$ a finite set of primes. To this data is attached a local, reduced, $p$-adic Hecke algebra $\Hsmr$, which is the local factor corresponding to $\rhobar$ in the inverse limit of Hecke algebra of tame level $U^{(p)}$ as the level at $p$ goes to infinity generated by operators outside $\Sigma$. In general, $\Hsmr$ depends on $U^{(p)}$ even though this dependence is suppressed in the notation. Let $\Ts$ be the Galois representation with coefficients in $\Hsmr$ deforming $\rhobar$ (under our assumptions, $\Hsmr$ is the universal deformation ring and $\Ts$ is the universal deformation of $\rhobar$ satisfying certain supplementary conditions depending on $U^{(p)}$).

\begin{Theo}\label{TheoIntro}
In addition to assumption \ref{HypIntro}, assume that the compact open subgroup $U^{(p)}$ satisfies the following property: if $\ell\nmid p$ belongs to $\Sigma\backslash\Si(\rhobar)$ then it is odd and the deformation type at $\ell$ satisfies one of the conditions of assupmption \ref{HypLocalDef}. Then the following assertions are equivalent.
\begin{enumerate}
\item There exists a motivic point $\la\in\Spec^{\mot}\Rsr$ such that conjecture \ref{ConjIMCIntro} holds for $T_{\la}\tenseur\La_{\Iw}$.
\item For all motivic point $\la\in\Spec^{\mot}\Rsr$, conjecture \ref{ConjIMCIntro} holds for $T_{\la}\tenseur\La_{\Iw}$ (the classical Iwasawa Main Conjecture is true for the motive $\Wcal(f)(r)$). If in addition $L(f,\chi,r)\neq0$, then conjecture \ref{ConjTNCIntro} for $\Wcal(\la)$ at $p$ holds (the Tamagawa Number Conjecture is true for the motive $\Wcal(f)(r)$ at $p$).
\item The Iwasawa Main Conjecture (conjecture \ref{ConjIMCIntro}) holds for $\Ts$.
\end{enumerate}
\end{Theo}
%
%
%\begin{Theo}\label{TheoIntro}
%Assume that $\rhobar$ satisfies assumptions \ref{HypMain}. In addition, assume that the compact open subgroup $U^{(p)}$ satisfies the following property: if $\ell\nmid p$ belongs to $\Sigma\backslash\Si(\rhobar)$ then one of the following holds.
%\begin{enumerate}
%\item Either $\ell\not\equiv\pm1\modulo p$.
%\item Or $\ell\equiv-1\modulo p$ and for all motivic point $\la\in\Spec^{\mot}\Hsmr$, the restriction of $\rho_{\la}$ to $G_{\Q_{\ell}}$ is reducible.
%\item Or $\ell\equiv1\modulo p$ and for all motivic point $\la\in\Spec^{\mot}\Hsmr$, the restriction of $\rho_{\la}$ to $I_{\ell}$ is scalar.
%\end{enumerate}
%Then the same equivalences as in theorem \ref{TheoMinIntro} are true, \textit{i.e} the following assertions are equivalent.
%\begin{enumerate}
%\item There exists a motivic point $\la\in\Spec^{\mot}\Rsr$ such that conjecture \ref{ConjIMCIntro} holds for $T_{\la}\tenseur\La_{\Iw}$.
%\item For all motivic point $\la\in\Spec^{\mot}\Rsr$, conjecture \ref{ConjIMCIntro} holds for $T_{\la}\tenseur\La_{\Iw}$ (the classical Iwasawa Main Conjecture is true for the motive $\Wcal(f)(r)$). If in addition $L(f,\chi,r)\neq0$, then conjecture \ref{ConjTNCIntro} for $\Wcal(\la)$ at $p$ holds (the Tamagawa Number Conjecture is true for the motive $\Wcal(f)(r)$ at $p$).
%\item The Iwasawa Main Conjecture (conjecture \ref{ConjIMCIntro}) holds for $\Ts$.
%\end{enumerate}
%\end{Theo}
As mentioned above, the classical Iwasawa Main Conjecture (conjecture \ref{ConjIMCIntro}) is frequently known to hold for $T_{\la}\tenseur\La_{\Iw}$. Combining these results with theorem \ref{TheoIntro} yields the following theorem.
\begin{Cor}\label{CorIntro}
Assume that all the assumptions of theorem \ref{TheoIntro} hold. Assume in addition that there exists an odd prime $\ell\nmid p$ satisfying the the following properties.
\begin{enumerate}
\item Let $\mu:G_{\ql}\fleche\{\pm1\}$ be the non-trivial unramified quadratic character. Then $\rhobar|G_{\ql}$ is up to twist by a power of the cyclotomic character a ramified extension 
\begin{equation}\nonumber
\suiteexacte{}{}{\mu\chi_{\cyc}^{}}{\rhobar|G_{\ql}}{\mu}.
\end{equation}
\item There exists $\lambda\in\Spec^{\mot}\Rs$ attached to an eigencuspform $f\in S_{k}(\Gamma_{0}(N))$ with $\ell||N$. 
\end{enumerate}
Then the Iwasawa Main Conjecture (conjecture \ref{ConjIMCIntro}) holds for $\Ts$.
\end{Cor}
\subsection{Outline of the proof}The proofs of the previous statements rely on the combination of the following two general principles. First, conjectures on special values of $L$-functions should be formulated in terms of zeta morphisms between fundamental lines and compatibilities between these morphisms should be systematically explored. This principle suggests two strategies which are unfortunately in tension: on the one hand, zeta morphisms may be promoted to Euler systems and thus yield strong divisibility results over regular rings; on the other hand, we want to study zeta morphisms with coefficients in the largest possible generality, so with coefficients in large $p$-adic Hecke algebras, but Hecke algebras are not regular rings. This brings to our second main idea. As was first observed by K. Rubin, the method of Euler systems shows by descent that when a motive varies within the smooth locus of a certain deformation space of Galois representations, the deviation from the predictions of the conjectures on special values of the $L$-function of that motive must remain bounded. When the smooth locus is not the whole deformation space, this provides very little information as counterexamples to the conjectures, which are concentrated on exceptional points of the deformation space for elementary reasons, could very well lurk in the non-smooth locus. However, this observation does suggest trying to resolve singularities in a process compatible with zeta morphisms. In this manuscript, this is achieved by using Taylor-Wiles-Kisin systems, and what is precisely shown is that any point in the non-smooth locus of global deformation ring is sufficiently close $\mgot$-adically to a point on a regular ring appearing in the theory of Taylor-Wiles-Kisin systems for an ersatz of the descent procedure to be carried over (see lemma \ref{LemTW} and proposition \ref{PropTWgeneral} for the construction of the regular ring which appears and lemma \ref{LemDescendBis} for the approximate descent procedure). Along the way, strengthening of the outcome of the method of Euler systems and of results on congruences in Iwasawa theory are established (corollaries \ref{CorKatoStronger} and \ref{CorEPW}).

As mentioned after the statement of conjecture \ref{ConjETNCIntro}, the Equivariant Tamagawa Number Conjectures with coefficients in the Hecke algebra $\Hecke_{\mgot}$ (conjectures \ref{ConjETNCIntro} and \ref{ConjETNCIntroBis}) and the Iwasawa Main Conjecture with coefficients in the Hecke algebra $\Hsmr$  (conjecture \ref{ConjIMCIntro}) are logically much stronger statements even than the collection of all Tamagawa Number Conjectures for each motive $\Wcal(\la)$ attached to a motivic point of $\Hecke_{\mgot}$ or $\Hsmr$. The ETNC with coefficients in $\Hsmr$ expresses the special values of the $L$-functions of all eigencuspforms attached to classical primes of $\Hsm$ in terms of a single zeta element $\zs$ and of a single fundamental line $\Delta_{\Si}(\Ts)$. Because $\Hsmr$ is under the hypotheses of theorem \ref{TheoIntro} the universal deformation ring of $\rhobar$, these are exactly the eigencuspforms with residual representation isomorphic to $\rhobar$ satisfying certain ramification conditions. The Iwasawa Main Conjecture with coefficients in the Hecke algebra thus captures congruences between special values of $L$-functions of congruent modular forms. In particular, theorem \ref{TheoIntro} and corollary \ref{CorIntro} settle in the affirmative the question asked at the end of the introduction of \cite{MazurValues}, whereas the truth of the classical Iwasawa main conjecture for a pair of congruent eigencuspforms or even the truth of ETNC with coefficients in Hecke algebras for such a pair would not \textit{a priori} have any bearing on this question as far as this author can tell. Reversing the perspective as in \cite{KatoICM}, theorem \ref{TheoIntro} also implies that the structure as $\Hsm$-module of the Galois cohomology of $M(f)_{\et,p}$ is encoded in the special values of the $L$-functions of forms congruent to $f$ (see section \ref{SubExample}).

To illustrate with a concrete example, consider $f$ the eigencuspform attached to the elliptic curve with good supersingular reduction at $p=3$ with $a_{p}=0$ of section \ref{SubExample}. It is easy to check that the Iwasawa Main Conjecture at $p=3$ is true for $T_{p}E$. This can be done for instance by computing the Iwasawa invariants of the $p$-adic $L$-function of \cite{PollackSupersingular} or (as is actually done below) by computing the Kolyvagin class attached to a suitable choice of Kolyvagin primes. In fact, the complex $\RGamma_{\et}(\Z[1/p],T(f)_{\Iw})$ has the simplest cohomology possible: the cohomology of this complex is concentrated in degree 1 and $\Hun_{\et}(\Z[1/p],T(f)_{\Iw})$ is a free $\La_{\Iw}$-module of rank 1 generated by Kato's class $\z(f)_{\Iw}$. There is a unique eigencuspform $g\in S_{4}(\Gamma_{0}(1640))$ with finite non-zero slope at the single prime above $3$ for which it is congruent to $f$. As a consequence, the $p$-adic $L$-function of $g$  (\cite{ManinPadic,VisikPadic,MazurTateTeitelbaum}) is not a power-series so it is computationally hard to extract from it refined information about the Galois cohomology of $M(g)_{\et,3}$. Nevertheless, using the minimal information that $g$ is congruent to $f$, theorem \ref{TheoIntro} and corollary \ref{CorIntro} imply first that the Iwasawa Main Conjecture is true for $T(g)_{\Iw}$ then yields an explicit description of $\RGamma_{\et}(\Z[1/p],T(g)_{\Iw})$: $\Hun_{\et}(\Z[1/p],T(g)_{\Iw})$ is a free $\Ocal_{\Iw}$-module of rank 1 which is \emph{not} generated by Kato's class $\z(g)_{\Iw}$ and $H^{2}_{\et}(\Z[1/p],T(g)_{\Iw})$ is a cyclic torsion $\Oiwa$-module isomorphic to $\Oiwa/(\omega)$ for $\omega$ a certain explicit Euler factor. As far as this author can see, the knowledge of the classical Iwasawa Main Conjecture (in any form) for one element of the pair or even for both would allow no such prediction.

\paragraph{Notations}All rings are assumed to be commutative (and unital). The total quotient ring of a reduced ring $R$ is denoted by $Q(R)$ and the fraction field of a domain $A$ is denoted by $\Frac(A)$. If $R$ is a domain with field of fraction $K$, if $j$ denotes the morphism $j:\Spec K\fleche\Spec R$ and if $\Fcali$ is a sheaf on the étale site of $\Spec K$, we denote by $\RGamma_{\et}(R,\Fcali)$ the étale cohomology complex $\RGamma_{\et}(\Spec R,j_{*}\Fcali)$.

If $F$ is a field, we denote by $G_{F}$ the Galois group of a separable closure of $F$. If $F$ is a number field and $\Sigma$ is a finite set of places of $F$, we denote by $F_{\Sigma}$ the maximal Galois extension of $F$ unramified outside $\Sigma\cup\{v|\infty\}$ and by $G_{F,\Sigma}$ the Galois group $\Gal(F_{\Sigma}/F)$. The ring of integer of $F$ is written $\Ocal_{F}$. If $v$ is a finite place of $F$, then $\Ocal_{F,v}$ is the unit ball of $F_{v}$. The reciprocity law of local class field theory is normalized so that the uniformizing parameter is sent to (a choice of lift of) the geometric Frobenius morphism $\Fr(v)$. For all rational primes $\ell$, we fix an algebraic closure $\Qbar_{\ell}$ of $\Q_{\ell}$, an embedding of $\Qbar$ into $\Qbar_{\ell}$ and an identification $\iota_{\infty,\ell}:\C\simeq\Qbar_{\ell}$ extending $\Qbar\plonge\Qbar_{\ell}$.

The non-trivial element of $\Gal(\C/\R)$ is denoted by $\tau$. If $R$ is a ring in which $2$ is a unit and $M$ is an $R[\Gal(\C/\R)]$-module (resp. $m$ is an element of $M$), then $M^{\pm}$ (resp. $m^{\pm}$) denotes the eigenspace on which $\tau$ acts as $\pm 1$ (resp. the projection of $m$ to the $\pm$-eigenspace).

The $p$-adic cyclotomic character of $G_{\Q}$ is denoted by $\chi_{\cyc}$. The field $\Q_{\infty}/\Q$ is the cyclotomic $\zp$-extension of $\Q$, that is to say the only Galois extension of $\Q$ with Galois group $\Gamma$ isomorphic to $\zp$. For $n\in\N$, the number field $\Q_{n}$ is the sub-extension of $\Q_{\infty}$ with Galois group $\Z/p^{n}\Z$. If $S$ is a reduced $\zp$-algebra, we write $S_{\Iw}$ for the completed group-algebra $S[[\Gamma]]$. To stick to usual notations, the 2-dimensional regular local ring $\Z_{p,\Iw}=\zp[[\Gamma]]\simeq\zp[[X]]$ is denoted by $\Lambda_{\Iw}$.

\section{The ETNC for modular motives}
\subsection{Modular motives}
As in the introduction, we consider an eigencsupform
\begin{equation}\nonumber
f(z)=\somme{n=1}{\infty}a_{n}(f)q^{n}\in S_{k}(\Gamma_{1}(N),\epsi)
\end{equation}
of weight $k\geq2$, level $\Gamma_{1}(N)$ and nebentypus $\epsi$ with coefficients in a number field $F\subset\C$.

To $f$ and $r\in\Z$ is attached a Grothendieck motive $\Wcal(f)(r)$ pure of weight $k-1$ over $\Q$ and with coefficients in $F$ (\cite{SchollMotivesModular}). We fix $\pid|p$ a prime ideal of $\Ocal_{F}$ over $p$ and write $E$ for $F_{\pid}$ and $\Ocal$ for the unit ball of $E$. For $*\in\{\C,\dR,\pid\}$, We write $V(r)_{*}$ for the corresponding Betti, de Rham or $\pid$-adic étale realization of $\Wcal(f)(r)$. We recall that $V(r)_{\C}$ and $V(r)_{\pid}$ are endowed with an action of $\Gal(\C/\R)$ and that Artin comparison theorem 
\begin{equation}\nonumber
V(r)_{\C}\tenseur_{F}E\isocan V(r)_{\pid}
\end{equation}
is $\Gal(\C/\R)$-equivariant for this action. We denote by $(-)^{+}$ the functor
\begin{equation}\nonumber
 H^{0}\left(\Gal(\C/\R),-\right).
\end{equation}
Let $\zeta_{p^{n+1}}$ be a primitive root of unity of order $p^{n+1}$ and let $\Q_{n}/\Q$ be the subfield of $\Q(\zeta_{p^{n+1}})$ satisfying
\begin{equation}\nonumber
G_{n}\eqdef\Gal(\Q_{n}/\Q)\simeq\Z/p^{n}\Z.
\end{equation}
If $\chi\in\hat{G}_{n}$ is a character of $G_{n}$, we denote by $F_{\chi}$ the sub-extension of $\C$ generated by $F$ and the image of $\chi$. We write $E_{\chi}$ and $\Ocal_{\chi}$ for a corresponding choice of a finite extension of $E$ and, in a slight abuse of notation, we still denote by $\pid\in\Spec\Ocal_{\chi}$ a prime ideal above $p$. We consider $h^{0}(\Spec\Q_{n})$ the Artin motive attached to the regular representation of $G_{n}$ viewed as a pure motive over $\Q$ with coefficients in $F_{\chi}$. Let $h^{0}(\Spec\Q_{n})_{\chi}$ be the direct summand of $h^{0}(\Spec\Q_{n})$ on which $G_{n}$ acts through $\chi$. Let $\Wcal(f\tenseur\chi)(r)$ be the pure Grothendieck motive $\Wcal(f)(r)\times h^{0}(\Spec\Q_{n})_{\chi}$ with coefficients in $F_{\chi}$. We write $V(r)_{\chi,*}$ for the corresponding $*$-realization.

For $S$ a set of finite primes containing\footnote{Contrary to what is generally the case in this manuscript, it is not necessary here to assume that $S$ also contains the primes dividing the level $N$.} $\{p\}$, the $S$-partial $L$-function $L_{S}(f^{*},\chi,s)$ is the holomorphic complex function satisfying 
\begin{equation}\nonumber
L_{S}(f^{*},\chi,s)\eqdef\produit{\ell\notin S}{}\frac{1}{1-\overline{a_{\ell}}\chi(\Fr(\ell))\ell^{-s}+\overline{\epsi(\ell)}\chi(\Fr(\ell))\ell^{k-1-2s}}
\end{equation}
for all $s\in\C$ with $\Re s>>0$ (here $\overline{\cdot}$ denotes complex conjugation).

The Betti-de Rham comparison isomorphism induces a complex period map
\begin{equation}\label{EqPerC}
\per_{\C}:\Fil^{0}V_{\chi,\dR}(r)\tenseur_{F_{\chi}}\C\fleche V_{\chi,\C}(r-1)^{+}\tenseur_{F_{\chi}}\C.
\end{equation}
which is an isomorphism if $1\leq r\leq k-1$ (\cite{DeligneFonctionsL}). Suppose $1\leq r\leq k-1$. The composition of localization at $p$ with the dual exponential map $\exp^{*}$ of \cite{BlochKato}
\begin{equation}\nonumber
\exp^{*}:\Hun(G_{\qp(\zeta_{p^{n}})},V(r)_{\pid})\fleche D_{\dR}^{0}(V(r)_{\pid})
\end{equation}
induces an inverse $p$-adic period map of $E_{\chi}$-vector spaces
\begin{equation}\label{EqPerP}
\per^{-1}_{p}:\Hun_{\et}(\Z[1/S],V(r)_{\chi,\pid})\fleche\Fil^{0}V(r)_{\chi,\dR}\tenseur_{F}E_{\chi}
\end{equation}
which is equivariant under the action of $G_{n}$ on both sides.

According to \cite{Rohrlich} and \cite[Theorems 12.4 and 12.5]{KatoEuler}, for all $n\in\N$ and all $\chi\in\hat{G}_{n}$ except possibly finitely many, the $S$-partial $L$-function $L_{S}(f,\chi,s)$ does not vanish at $r$, $\Hun_{\et}(\Z[1/S],V(r)_{\chi,\pid})$ is an $E_{\chi}$-vector space of dimension 1 and \eqref{EqPerP} is an isomorphism. When this holds, the motive $\Wcal(f\tenseur{\chi})(r)$ is said to be \emph{strictly critical} (implicitly, at $p$) (\cite[Section 3.2.6]{KatoViaBdR}).

\subsection{The Tamagawa Number Conjectures}
\subsubsection{The motivic fundamental line}
Suppose that $\Wcal(f\tenseur{\chi})(r)$ is strictly critical. The determinant functor applied to the $p$-adic period map \eqref{EqPerP} then yields an isomorphism of free $E_{\chi}$-vector spaces of rank 1
\begin{equation}\nonumber%\label{EqPerDetP}
\xymatrix{
\Det_{E_{\chi}}\Hun_{\et}(\Z[1/S],V(r)_{\chi,\pid})\ar[d]_{\isocan}^{\per^{-1}_{p}}\\
\Det_{E_{\chi}}\Fil^{0}V(r)_{\chi,\dR}\tenseur_{F}E_{\chi}.
}
\end{equation}
Recall that ${(-)}^{+}$ sends a module to its invariant submodule under $\Gal(\C/\R)$. Taking tensor product with the determinant of $V(r-1)_{\chi,\pid}^{+}$ yields an identification  
\begin{equation}\label{EqPerDetP}
\xymatrix{
\Det_{E_{\chi}}\Hun_{\et}(\Z[1/S],V(r)_{\chi,\pid})\tenseur_{E_{\chi}}\Det^{-1}_{E_{\chi}}V(r-1)_{\chi,\pid}^{+}\ar[d]_{\isocan}^{\per^{-1}_{p}}\\
\Det_{E_{\chi}}\left(\Fil^{0}V(r)_{\chi,\dR}\tenseur_{F}E_{\chi}\right)\tenseur_{E_{\chi}}\Det^{-1}_{E_{\chi}}V(r-1)_{\chi,\pid}^{+}.
}
\end{equation}
Similarly, the determinant functor applied to the complex period map \eqref{EqPerC} induces an identification
\begin{equation}\label{EqPerDetC}
\xymatrix{
\left[\Det^{}_{\C}\Fil^{0}V(r)_{\chi,\dR}\tenseur_{F_{\chi}}\C\right]\tenseur_{\C}\left[\Det^{-1}_{\C}V(r-1)_{\chi,\C}^{+}\tenseur_{F_{\chi}}\C\right]\ar[d]_{\isocan}^{\per_{\C}}\\
\C.
}
\end{equation}  
\begin{Def}
The \emph{motivic fundamental line} $\left(\Delta_{F}(\Wcal(f\tenseur{\chi})(r)),\per_{p},\per_{\C}\right)$ of the strictly critical motive $\Wcal(f\tenseur{\chi})(r)$ is the one-dimensional $F_{\chi}$-vector space
\begin{equation}\label{EqQsub}
\Delta_{F}(\Wcal(f\tenseur{\chi})(r))\eqdef\Det_{F_{\chi}}\Fil^{0}V(r)_{\chi,\dR}\tenseur_{F_{\chi}}\Det^{-1}_{F_{\chi}}V(r-1)_{\chi,\C}^{+}
\end{equation}
together with the two isomorphisms 
\begin{equation}\nonumber
\per_{p}:\Delta_{F}(\Wcal(f\tenseur{\chi})(r))\tenseur_{F_{\chi}}E_{\chi}\isom\Det_{E_{\chi}}\Hun_{\et}(\Z[1/S],V(r)_{\chi,\pid})\tenseur_{E_{\chi}}\Det^{-1}_{E_{\chi}}V(r-1)_{\chi,\pid}^{+}
\end{equation}
\begin{equation}\nonumber
\per_{\C}:\Delta_{F}(\Wcal(f\tenseur{\chi})(r))\tenseur_{F_{\chi}}\C\isom\C.
\end{equation}
\end{Def}
Though this is not apparent in the notations, the isomorphism $\per_{p}$ and hence the motivic fundamental line do depend on the choice of $S$. The motivic fundamental line of $\Wcal(f\tenseur\chi)(r)$ is an $F_{\chi}$-rational subspace both of the target of \eqref{EqPerDetP} and of the source of \eqref{EqPerDetC}. Suppose 
\begin{equation}\label{EqExampleZeta}
\z:V(r-1)_{\chi,\pid}^{+}\fleche\Hun_{\et}(\Z[1/S],V(r)_{\chi,\pid})
\end{equation}
is a morphism. Taking determinants on both sides, $\z$ can be viewed as an element of the source of \eqref{EqPerDetP}. If furthermore its image through $\per^{-1}_{p}$ lands in the motivic fundamental line $\Delta_{F}(\Wcal(f\tenseur{\chi})(r))$, then a complex number $\per_{\C}(\per^{-1}_{p}(\z)\tenseur1)$ is attached. Let $\per_{\C}(\per^{-1}_{p}(-)\tenseur1)$ be the map defined by this composition on the sub-$F_{\chi}$-subspace $\Delta_{F}(\Wcal(f\tenseur{\chi})(r))_{\pid}$ equal to the image of $\Delta_{F}(\Wcal(f\tenseur{\chi})(r))$ through $\per_{p}$. 
\subsubsection{Zeta morphisms}
The motive $\Wcal(f\tenseur\chi)(r)$ is still assumed to be strictly critical.

As in \eqref{EqExampleZeta} above, $\z$ is an isomorphism
\begin{equation}\nonumber
\z:V(r-1)_{\chi,\pid}^{+}\isom\Hun_{\et}(\Z[1/S],V(r)_{\chi,\pid})
\end{equation}
between the 1-dimensional $E_{\chi}$-vector spaces $V(r-1)_{\chi,\pid}^{+}$ and $\Hun_{\et}(\Z[1/S],V(r)_{\chi,\pid})$. Applying the functor $\Det$ then yields an identification
\begin{equation}\nonumber
\Det_{E_{\chi}}(\z):\Det_{E_{\chi}}\Hun_{\et}(\Z[1/S],V(r)_{\chi,\pid})\tenseur_{E_{\chi}}\Det^{-1}_{E_{\chi}}V(r-1)_{\chi,\pid}^{+}\simeq E_{\chi}.
\end{equation}
Let $\Delta_{F,\z}(\Wcal(f\tenseur{\chi})(r))$ be the $F_{\chi}$-vector space pre-image of $F_{\chi}\subset E_{\chi}$ through this isomorphism and let $\Delta_{F}(\Wcal(f\tenseur{\chi})(r))_{\pid}$ be the image of $\Delta_{F}(\Wcal(f\tenseur{\chi})(r))$ through $\per_{p}$
\begin{Def}
A morphism 
\begin{equation}\nonumber
\z:V(r-1)_{\chi,\pid}^{+}\isom\Hun_{\et}(\Z[1/S],V(r)_{\chi,\pid})
\end{equation}
is \emph{motivic} if $\Delta_{F,\z}(\Wcal(f\tenseur{\chi})(r))$ is equal to $\Delta_{F}(\Wcal(f\tenseur{\chi})(r))_{\pid}$. If a morphism $\z$ is motivic, we say it is the \emph{$S$-partial zeta morphism} of $\Wcal(f\tenseur{\chi})(r)$ if
\begin{equation}\label{EqZetaPeriodComplex}
\per_{\C}\left(\per_{p}^{-1}\left(\Det_{E_{\chi}}(\z)^{-1}(1)\right)\tenseur1\right)=L_{S}(f^{*},\chi^{-1},1-r)\in\C.
\end{equation}
\end{Def}
Note that the composition $\per_{\C}\left(\per_{p}^{-1}\left(\Det_{E_{\chi}}(\z)^{-1}(1)\right)\tenseur1\right)$ makes sense as the pre-image of $\Det_{E_{\chi}}(Z)^{-1}(1)$ through $\per_{p}$ belongs to $\Delta_{F}(\Wcal(f\tenseur{\chi})(r))$ if (and only if) $\z$ is motivic. An $S$-partial zeta morphism is uniquely determined by \eqref{EqZetaPeriodComplex}, and hence unique if it exists.

Taking the determinant functor of a zeta morphism induces a trivialization
\begin{equation}\nonumber
\z:\Det^{-1}_{E_{\chi}}\RGamma_{\et}\left(\Z[1/S],V(r)_{\chi,\pid}\right)\tenseur_{E_{\chi}}\Det^{-1}_{E_{\chi}}V(r-1)_{\chi,\pid}^{+}\isom E_{\chi}
\end{equation}
\begin{DefEnglish}
The \emph{$p$-adic fundamental line} $\Delta_{S,\pid}(f\tenseur\chi)(r)$ is the one-dimensional $E_{\chi}$-vector space
\begin{equation}\nonumber
\Delta_{S,\pid}(f\tenseur\chi)(r)\eqdef\Det^{-1}_{E_{\chi}}\RGamma_{\et}\left(\Z[1/S],V(r)_{\chi,\pid}\right)\tenseur_{E_{\chi}}\Det^{-1}_{E_{\chi}}V(r-1)_{\chi,\pid}^{+}
\end{equation}
together with the free of rank one sub-$\Ocal_{\chi}$-module
\begin{equation}\nonumber
\Delta_{S,\pid}(T)\eqdef\Det^{-1}_{\Ocal_{\chi}}\RGamma_{\et}\left(\Z[1/S],T\right)\tenseur_{\Ocal_{\chi}}\Det^{-1}_{\Ocal_{\chi}}T(-1)^{+}
\end{equation}
obtained by choosing a $G_{\Q,S}$-stable $\Ocal_{\chi}$-lattice $T$ inside $V(r)_{\chi,\pid}$.
\end{DefEnglish}
It is well-known and easy to prove that $\Delta_{S,\pid}(T)$ does not depend on the choice of the $G_{\Q,S}$-stable $\Ocal_{\chi}$-lattice $T$ (see \cite{KatoHodgeIwasawa,BurnsFlachMotivic}).
\begin{Conj}[The Tamagawa Number Conjectures, \cite{BlochKato,KatoHodgeIwasawa}]\label{ConjTNC}
The $S$-partial zeta morphism $\z_{S}(f\tenseur\chi)(r)$ of $\Wcal(f\tenseur{\chi})(r)$ exists and
\begin{equation}\nonumber
\z_{S}(f\tenseur\chi)(r):\Delta_{S,\pid}(f\tenseur\chi)(r)\isom E_{\chi}
\end{equation}
induces an isomorphism
\begin{equation}\nonumber
\z_{S}(f\tenseur\chi)(r):\Delta_{S,\pid}(T)\isom \Ocal_{\chi}.
\end{equation}
\end{Conj}
Conjecture \ref{ConjTNC} is equivalent to the Tamagawa Number Conjecture for the motive $\Wcal(f\tenseur{\chi})(r)$ of \cite{BlochKato} expressing the value at zero of the $L$-function of $\Wcal(f\tenseur{\chi})(r)$ in terms of Tamagawa measures (see \cite[section 11.6]{FontaineValeursSpeciales}). Suppose in addition that $\chi$ is the trivial character and that $f$ is attached to a rational elliptic curve $E/\Q$ of conductor $N$ ($f$ belongs to $S_{2}(\Gamma_{0}(N))$ and has coefficients in $\Q$). Then we may compare \ref{ConjTNC} with the $p$-part of the Birch and Swinnerton-Dyer Conjecture predicting the equality
\begin{equation}\nonumber
v_{p}\left(\frac{L(E,1)}{\Omega_{E}}\right)=v_{p}\left(\frac{\cardinal{\Sha(E/\Q)[p^{\infty}]}\produit{\ell|N}{}\Tam_{\ell}(E/\Q)}{\cardinal{E(\Q)_{\tors}}^{2}}\right).
\end{equation}
A well-known translation (\cite{BlochKato,BurnsFlachMotivic,VenjakobETNC}) then shows that conjecture \ref{ConjTNC} predicts
\begin{equation}\nonumber
v_{p}\left(\frac{L(E,1)}{\Omega_{E}}\right)=v_{p}\left(\frac{\cardinal{\Sha(E/\Q)/\Sha(E/\Q)_{\partiediv}[p^{\infty}]}\produit{\ell|N}{}\Tam_{\ell}(E/\Q)}{\cardinal{E(\Q)_{\tors}}^{2}}\right).
\end{equation}
Under our ongoing hypothesis $L(E,1)\neq0$, it is known that $\Sha(E/\Q)$ is a finite group and hence that its divisible part $\Sha(E/\Q)_{\partiediv}$ vanishes (\cite{GrossZagier,KolyvaginEuler,KatoEuler}). Hence, conjecture \ref{ConjTNC} is equivalent to the $p$-part of the Birch and Swinnerton-Dyer Conjecture in our setting. This author believes that the relation between conjecture \ref{ConjTNC} and the $p$-part of the Birch and Swinnerton-Dyer Conjecture is not well understood even for a single prime $p$ and a single elliptic curve if $\ord_{s=1}L(E,s)\geq2$.
\subsection{The classical Iwasawa Main Conjecture}
Fix $T(f)$ a $G_{\Q}$-stable $\Ocal$-lattice inside $\Wcal(f)_{\et,\pid}$. Recall that $\Oiwa$ is the completed power series-ring $\Ocal[[\Gal(\Q_{\infty}/\Q)]]$. Seen as module over itself, it is endowed with a Galois action by the composition of the quotient $G_{\Q,S}\surjection\Gal(\Q_{\infty}/\Q)\plonge\Oiwa\croix$. Write $T(f)_{\Iw}$ for $T(f)\tenseur_{\Ocal}\Oiwa$ with $G_{\Q,S}$-action on both sides of the tensor product. 

%Note that $T(f)_{\Iw}$ is a Galois representation with coefficients in $\Oiwa$ which is unramified outside of $S$ and such that $T(f)_{\Iw}\tenseur_{\Oiwa}k$ is equal to $\rho_{f}$. Hence, $T(f)_{\Iw}$ is a deformation of $\rhobar_{f}$ and hence a $\Oiwa$-valued point of $\Xsr$. More generally, if $\Pcal\in\Spec\Oiwa$ is a prime ideal of height 1 prime to $(p)$, then $\Oiwa/\Pcal$ is a discrete valuation ring flat over $\zp$ and $T(f)_{\Iw}\tenseur_{\Oiwa}\Oiwa/\Pcal$ is a Galois representation deforming $\rhobar_{f}$ and hence an $\Oiwa/\Pcal$-valued point of $\Xsr$. Hence, $\Spec\Oiwa[1/p]$ is a subset of $\Xsr$. As we already observed at the end of the second paragraph of the proof of theorem \ref{TheoModular}, $\Oiwa$ is the universal deformation of the trivial character with values in $k\croix$ and so the subset $\Spec\Oiwa$ is closed in $\Xsr$.

%
Consider a morphism
\begin{equation}\label{EqMorIw}
\z_{\Iw}:T(f)_{\Iw}(-1)^{+}\fleche\Hun_{\et}\left(\Z[1/S],T(f)_{\Iw}\right).
\end{equation}
After descent to a finite level $n$, a choice of character $\chi$, a choice of a twist $r$ such that $\Wcal(f\tenseur{\chi})(r)$ is strictly critical and inversion of $p$, the morphism \eqref{EqMorIw} induces a morphism
\begin{equation}\nonumber
\z_{\chi,r}:V(r-1)_{\chi,\pid}^{+}\isom\Hun_{\et}(\Z[1/S],V(r)_{\chi,\pid}).
\end{equation}
\begin{DefEnglish}
A morphism
\begin{equation}\label{EqMorIwBis}
\z_{\Iw}:T(f)_{\Iw}(-1)^{+}\fleche\Hun_{\et}\left(\Z[1/S],T(f)_{\Iw}\right).
\end{equation}
is the \emph{$S$-partial zeta morphism} of $T(f)_{\Iw}$ if the morphism $\z_{\chi,r}$ it induces coincides with the $S$-partial zeta morphism $\z_{S}(f\tenseur\chi)(r)$ for all choice of $\chi$ and $r$ such that $\Wcal(f\tenseur{\chi})(r)$ is strictly critical.
\end{DefEnglish}
The following deep theorem of K. Kato establishes that $T(f)_{\Iw}$ has an {$S$-partial zeta morphism}.
\begin{Theo}[K. Kato]\label{TheoZetaMorphism}
For all $S\supset\{p\}$, there exists an $S$-partial zeta morphism
\begin{equation}\nonumber
\z(f)_{\Iw}:T(f)_{\Iw}(-1)^{+}\fleche\Hun_{\et}\left(\Z[1/S],T(f)_{\Iw}\right)
\end{equation}
and by descent an $S$-partial zeta morphism
\begin{equation}\nonumber
\z_{S}(f\tenseur\chi)(r):V(r-1)_{\chi,\pid}^{+}\isom\Hun_{\et}(\Z[1/S],V(r)_{\chi,\pid})
\end{equation}
provided $\Wcal(f\tenseur{\chi})(r)$ is strictly critical.
\end{Theo}
\begin{proof}
See \cite[Theorems 12.4, 12.5]{KatoEuler}.
\end{proof}
The complexes $\RGamma_{\et}\left(\Z[1/S],T(f)_{\Iw}\right)$ and $T(f)_{\Iw}(-1)^{+}$ are bounded complexes $\Oiwa$-modules. As $\Oiwa$ is a regular local ring, they are perfect complexes by  the Auslander-Buchsbaum and Serre theorem. Hence, they have well-defined determinants over $\Oiwa$. Put
\begin{equation}\nonumber
\Delta_{\Oiwa}(T(f)_{\Iw})\eqdef\Det^{-1}_{\Oiwa}\RGamma_{\et}\left(\Z[1/S],T(f)_{\Iw}\right)\tenseur_{\Oiwa}\Det^{-1}_{\Oiwa}T(f)_{\Iw}(-1)^{+}.
\end{equation}
Exactly as in the case of $\Ocal_{\chi}$, the morphism \eqref{EqMorIw} induces an isomorphism
\begin{equation}\nonumber
\z(f)_{\Iw}:\Delta_{\Oiwa}(T(f)_{\Iw})\tenseur_{\Oiwa}\Frac(\Oiwa)\isom\Frac(\Oiwa).
\end{equation}
The following is the statement of the Iwasawa Main Conjecture for the motive $\Wcal(f)$.
\begin{Conj}[The Iwasawa Main Conjecture \cite{KatoEuler}]\label{ConjIMC}
The $S$-partial zeta morphism $\z(f)_{\Iw}$ of $T(f)_{\Iw}$ induces an isomorphism
\begin{equation}\nonumber
\z(f)_{\Iw}:\Delta_{\Oiwa}(T(f)_{\Iw})\isom\Oiwa.
\end{equation}
\end{Conj}
When the eigencuspform $f$ is nearly ordinary at $p$, conjecture \ref{ConjIMC} implies the usual formulation of the Iwasawa Main Conjecture in terms of $p$-adic $L$-function and Selmer modules. For instance, the following proposition holds.
\begin{Prop}\label{PropIMCIMC}
Assume that  $f$ is nearly ordinary at $p$ and that conjecture \ref{ConjIMC} holds. Then there is an equality
\begin{equation}\label{EqEgalOrdLpadic}
\left(L^{\cyc}_{p}(f)\right)\Oiwa=\carac_{\Oiwa}\Htilde^{2}_{f}(G_{\Q,\Sigma},T(f)_{\Iw}).
\end{equation}
Here, $L^{\cyc}_{p}(f)\in\Oiwa$ is the Mazur-Swinnerton-Dyer $p$-adic $L$-function (\cite{MazurSwinnertonDyer}) and $\Htilde^{2}_{f}(G_{\Q,\Sigma},T(f)_{\Iw})$ is the second cohomology group of the \Nekovar-Selmer complex with Greenberg's local condition at $p$ (\cite{SelmerComplexes}). 
\end{Prop}
Note that the proposition holds even if $L_{p}^{\cyc}(f)$ has an exceptional zero.
\begin{proof}
When in addition to being nearly ordinary, the $\GL_{2}(\qp)$-representation $\pi(f)_{p}$ is in the principal series, this follows from the arguments of \cite[Section 17.13]{KatoEuler}. The remaining case, that is to say if $\pi(f)_{p}$ is Steinberg, is treated in \cite[Section 4.4]{ColmezBSD}.
\end{proof}
When $f$ is of weight 2 or when $a_{p}(f)=0$, there are conjectures of R. Pollack, S. Kobayashi and I. Sprung predicting equalities between a pair of ideals generated in $\Oiwa$ by a pair of $p$-adic $L$-functions on one hand and the characteristic ideals of a pair of $\Oiwa$-adic Selmer groups on the other hand. When they are known to make sense, these conjectures also follow from conjecture \ref{ConjIMC} (see \cite{KobayashiIMC,PollackSupersingular,SprungIMC} for formulations and details).

%The following proposition is easy but it does not appear in the standard literature on the topic and is the prototype of the results we prove in this text.
%\begin{Prop}\label{PropComp}
%The Iwasawa Main Conjecture (conjecture \ref{ConjIMC}) implies the Tamagawa Number Conjecture (conjecture \ref{ConjTNC}) for all integers $1\leq r\leq k-1$ and all finite order character $\chi$ of $p$-power order.
%\end{Prop}
%\begin{proof}
%This boils down to the commutativity of the following diagram 
%\begin{equation}\nonumber
%\xymatrix{
%\ar[d]\Delta_{\Oiwa}(T(f)_{\Iw})\ar[rr]^(0.6){\z(f)_{\Iw}}&&\Oiwa\ar[d]^{\chi}\\
%\Delta_{S,\pid}(T(f)\tenseur\chi)(r)\ar[rr]^(0.65){\z_{S}(f\tenseur\chi)(r)}&&\Ocal_{\chi}
%}
%\end{equation}
%where the left vertical arrows is the isomorphism
%\begin{equation}\nonumber
%\Delta_{\Oiwa}(T(f)_{\Iw})\tenseur_{\Oiwa,\chi}\Ocal_{\chi}\isocan\Delta_{S,\pid}(T(f)\tenseur\chi)(r)
%\end{equation} 
%induced by the canonical isomorphism of complexes
%\begin{equation}\nonumber
%\RGamma_{\et}\left(\Z[1/S],T(f)_{\Iw}\right)\Ltenseur_{\Oiwa,\chi}\Ocal_{\chi}\simeq\RGamma_{\et}\left(\Z[1/S],T(f\tenseur\chi)(r)\right)
%\end{equation}
%coming from the universal coefficient property of étale cohomology (note that this uses crucially the fact that $T(f)_{\Iw}$ is a perfect complex of étale sheaves of $\Oiwa$-modules on $\Spec\Z[1/S]$).
%\end{proof}
\subsection{The universal Iwasawa Main Conjecture}
\subsubsection{Deformations}\label{SubLevels}
Let $k$ be the residual field of $\Ocal$ (the letter $k$ also denotes the weight of the eigencuspform $f$, this should cause no confusion). Let
\begin{equation}\nonumber
\rhobar:G_{\Q}\fleche\GL_{2}(k)
\end{equation}
be the residual representation attached to $f$. Let $N(\rhobar)$ be its tame Artin conductor. Let $\Si$ be a finite set of finite primes containing $\{\ell|N(\rhobar)p\}$ and let $\Si^{(p)}$ be $\Si\backslash\{p\}$.

Let $U\subset\GL_{2}(\A_{\Q}^{(\infty)})$ be a compact open subgroup. Write $U_{p}=U\cap\GL_{2}(\qp)$ and $U=U_{p}U^{(p)}$ with $U^{(p)}\subset\G(\A_{\Q}^{(p\infty)})$. Let $\Sigma(U)$ be the finite set of finite places $\ell$ such that $U_{\ell}$ is not compact open maximal. To $U$ is attached the modular curve $Y(U)$ whose set of complex points is the double quotient
\begin{equation}\nonumber
Y(U)(\C)\simeq\GL_{2}(\Q)\backslash\left(\C-\R\times\GL_{2}(\A_{\Q}^{(\infty)})/U\right).
\end{equation}
The classical Hecke algebra acts on the $\Z$-module $S_{k}(U)$ of eigenforms of level $U$ and weight $k$. The reduced Hecke algebra $\Hecke(U)$ is the sub-$\zp$-algebra generated inside the classical Hecke algebra by the diamond operators $\diamant{\ell}$ and the Hecke operators $T(\ell)$ for $\ell\notin\Sigma(U)^{(p)}$. It is a finite, flat, reduced, semi-local $\zp$-algebras.

A compact open subgroup $U^{(p)}\subset\GL_{2}(\A_{\Q}^{(p\infty)})$ is said to be allowable (with respect to $\rhobar$) if there exists a maximal ideal $\mgot_{\rhobar}$ of the Hecke algebra generated $\Hecke(U^{(p)})$ such that
\begin{equation}\nonumber
\begin{cases}
\tr\rhobar(\Fr(\ell))=T(\ell)\modulo\mgot_{\rhobar}\\
\det\rhobar(\Fr(\ell))=\ell\diamant{\ell}\modulo\mgot_{\rhobar}
\end{cases}
\end{equation}
for all $\ell\notin\Sigma(U^{(p)})\cup\{p\}$ (we recall that $\Sigma(U^{(p)})$ is the set of primes out of which $U^{(p)}\subset\G(\A_{\Q}^{(p\infty)})$ is maximal). A finite set of primes $\Sigma\supset\{\ell|N(\rhobar)p\}$ is said to be allowable (with respect to $\rhobar$) if there exists an allowable compact open subgroup $U^{(p)}\subset\G(\A_{\Q}^{(p\infty)})$ such that $\Sigma\supset\Sigma(U^{(p)})$.

Fix $U'\subset U$ two allowable compact open subgroups. Denote respectively by $\mgot'_{\rhobar}$ and $\mgot_{\rhobar}$ the maximal ideals of $\Hecke^{\Sigma'}(U')$ and $\Hecke^{\Sigma}(U)$ attached to $U'$ and $U$. Then the inclusion $S_{k}(U)\subset S_{k}(U')$ induces by restriction a map $\Hecke(U')_{\mgot'_{\rhobar}}\fleche\Hecke(U)_{\mgot_{\rhobar}}$ which is always a surjection and which is an isomorphism if $U$ is sufficiently small. Define
\begin{equation}\nonumber
\Hs_{\mr}\eqdef\limproj{U_{p}}\ \Hecke(U_{p}U^{(p)})_{\mr}.
\end{equation}
The ring $\Hsmr$ does not depend on the weight $k$ but does depend in general on $U^{(p)}$ though this dependance is suppressed for simplicity of notation. If $U^{(p)}$ is sufficiently small, then $\Hsmr$ depends only on $\Si$. However, in the rest of the text, the group $U^{(p)}$ is not necessarily be sufficiently small in that sense at primes $\ell\equiv\pm1\modulo p$. In particular, the ring $\Hsm$ does depend on $U^{(p)}$ in general even though this is suppressed from the notation. The $\zp$-algebra $\Hsmr$ is flat and its relative dimension is at least 3 (see \cite{GouveaMazur}).

Assume that $\rhobar:G_{\Q,\Si}\fleche\GL_{2}(k)$ satisfies the assumptions of the introduction, repeated here.
\begin{Hyp}\label{HypMain}
\begin{enumerate}
\item The image of $\rhobar$ contains a subgroup conjugated to $\SL_{2}(\fp)$.
\item\label{ItemrhobarEnP} If $\rhobar|_{G_{\qp}}$ is an extension of characters
\begin{equation}\nonumber
\suiteexacte{}{}{\chi}{\rhobar|_{G_{\qp}}}{\psi},
\end{equation}
then $\chi\psi^{-1}\neq1$ and $\chi\psi^{-1}\neq\bar{\chi}^{-1}_{\cyc}$ (here $\chi_{\cyc}$ is the $p$-adic cyclotomic character).
\end{enumerate}
\end{Hyp}
Under assumption \ref{HypMain}, there exists a $G_{\Q,\Sigma}$-representation $\Ts$ with coefficients in $\Hsmr$ with residual representation isomorphic to $\rhobar$. A local morphism $\la:\Hsmr\fleche\qpbar$ is \emph{motivic} if it is equal to the trace of the $G_{\Q,\Si}$-representation $\rho_{g}$ attached to an eigencuspform $g\in S_{k'}(U'_{p}U^{(p)})$ for some compact open subgroup $U'_{p}\subset\G(\qp)$ and some integer $k'\geq2$ up to twist by an integer power $r$ of the cyclotomic character and a finite order character $\chi$ of $\Gal(\Q_{\infty}/\Q)$ . Equivalently, $\la$ is a motivic point of $\Spec\Hsmr$ if it is up to twist the system of eigenvalues of an eigencuspform of weight greater than 2. Motivic points of $\Hsmr$ form a Zariski-dense set. The ring $\Hsm$ is the universal deformation ring representing deformations of $\rhobar$ unramified outside $\Si$ and which are attached to eigencuspforms for $U^{(p)}$ at de Rham points. 

If $\la$ is a motivic point with values in a discrete valuation ring $S$, let $T(\la)$ be the $G_{\Q,\Si}$-representation $\Ts\tenseur_{\Hsmr,\la}S$, let $T(\la)_{\Iw}$ be the $G_{\Q,\Si}$-representation $\Ts\tenseur_{\Hsmr,\la}S_{\Iw}$ and let $\zs(\la)_{\Iw}$ be the $\Si$-partial zeta morphism of theorem \ref{TheoZetaMorphism}.  %A motivic point is often denoted by $\la$ even though it strictly speaking also depends on $r$ and $\chi$. 

\subsubsection{The Iwasawa Main Conjecture}
The following theorem, which is due to K. Nakamura and independently P. Colmez and S. Wang, is essential in the rest of the text. 
\begin{Theo}[\cite{NakamuraUniversal}, Theorem 1.1]\label{TheoNakamura}
There exists a zeta morphism
\begin{equation}\nonumber
\zs:\Ts(-1)^{+}\fleche\Hun_{\et}(\Z[1/\Sigma],\Ts)
\end{equation}
such that for all motivic point $\la$ the following diagram commutes 
\begin{equation}\label{DiagNaka}
\xymatrix{
\Ts(-1)^{+}\ar[d]\ar[rr]^{\zs}&&\Hun_{\et}(\Z[1/\Sigma],\Ts)\ar[d]\\
T(\la)_{\Iw}(-1)^{+}\ar[rr]^{\zs(\la)_{\Iw}}&&\Hun_{\et}(\Z[1/\Si],T(\la)_{\Iw}).
}
\end{equation}
\end{Theo}
We recall that this deep theorem requires the full strength of the $p$-adic Langlands Correspondence of \cite{ColmezFoncteur,PaskunasMontreal}.
\begin{DefEnglish}
The universal fundamental line is the rank-one, free $\Hsmr$-module $\Ds$ defined by
\begin{equation}\nonumber
\Ds\eqdef\Det^{-1}_{\Hsmr}\RGamma_{\et}\left(\Z[1/\Si],\Ts\right)\tenseur_{\Hsmr}\Det_{\Hsmr}^{-1}\Ts(-1)^{+}.
\end{equation}
\end{DefEnglish}
The $\Hsmr$-module $\Ts$ is a perfect complex of étale sheaves of $\Hsmr$-module. The functor $\RGamma_{\et}\left(\Z[1/\Si],-\right)$ sends perfect complex of étale sheaves to perfect complex of étale sheaves (see e.g \cite[Proposition (4.2.9)]{SelmerComplexes}) and commutes with arbitrary base change of ring. Hence the universal fundamental line $\Ds$ is well-defined and commutes with arbitrary change of coefficients.
\begin{Prop}\label{PropTorsion}
Let $Q(\Hsmr)$ be the total ring of quotients of $\Hsmr$. The morphism
\begin{equation}\nonumber
\zs:\Ts(-1)^{+}\fleche\Hun_{\et}(\Z[1/\Sigma],\Ts)
\end{equation}
induces
\begin{equation}\label{EqZetaElementSigma}
\zs:\Ds\tenseur_{\Hsmr}Q(\Hsmr)\isom Q(\Hsmr).
\end{equation}
\end{Prop}
\begin{proof}
We show that $H^{2}_{\et}\left(\Z[1/\Si],\Ts\right)$ and $H^{1}_{\et}\left(\Z[1/\Si],\Ts\right)/\image\zs$ are torsion $\Hsmr$-modules and that $H^{i}_{\et}\left(\Z[1/\Si],\Ts\right)$ vanishes if $i\neq1,2$. 

The cohomology of the complex $\RGamma_{\et}\left(\Z[1/\Si],\Ts\right)$ is concentrated in degree $[0,3]$. Let $\la:\Hsmr\fleche\Ocal$ be a motivic point. Under assumptions \ref{HypMain}, the cohomology of the complex $\RGamma_{\et}\left(\Z[1/\Si],T(\la)_{\Iw}\right)$ is concentrated in degree $[1,2]$,  $H^{2}_{\et}(\Z[1/\Si],T(\la)_{\Iw})$ is a torsion $\Oiwa$-module and $H^{1}_{\et}(\Z[1/\Si],T(\la)_{\Iw})$ is a free $\Oiwa$-module of rank 1. The canonical isomorphism
\begin{equation}\nonumber
\RGamma_{\et}\left(\Z[1/\Si],\Ts\right)\Ltenseur_{\Hsmr,\la}\Oiwa\isocan\RGamma_{\et}\left(\Z[1/\Si],T(\la)_{\Iw}\right)
\end{equation}
and Nakayama's lemma show first that the cohomology of $\RGamma_{\et}\left(\Z[1/p],\Ts\right)$ is concentrated in degree $[1,2]$, then that $H^{2}_{\et}\left(\Z[1/\Si],\Ts\right)$ is torsion and finally that $H^{1}_{\et}\left(\Z[1/\Si],\Ts\right)$ is generated by at most one element over $Q(\Hsmr)$. In order to show that the $\Hsmr$-module $H^{1}_{\et}\left(\Z[1/\Si],\Ts\right)/\image\zs$ is torsion, it is consequently sufficient to do so after base-changing to each irreducible component of $\Hsmr$. Another application of Nakayama's lemma shows that $H^{1}_{\et}\left(\Z[1/\Si],\Ts\right)/\image\zs$ is torsion after base-change to an irreducible component of $\Hsmr$ if there exists a motivic point $\psi$ of $\Hsmr$ factoring through this irreducible component and such that
\begin{equation}\nonumber
H^{1}_{\et}\left(\Z[1/\Si],T(\psi)_{\Iw}\right)/\image\z(\psi)_{\Iw}
\end{equation}
is torsion. This holds for any motivic point by \cite[Theorem 12.4]{KatoEuler}.

After extension of scalars to $Q(\Hsmr)$, the fundamental line $\Ds$ is thus canonically isomorphic to the determinant functor of the zero complex, and hence canonically isomorphic to $Q(\Hsmr)$.
\end{proof}
Proposition \ref{PropTorsion} shows that the following conjecture is at least meaningful.
\begin{Conj}[Universal Iwasawa Main Conjecture]\label{ConjUniv}
The universal zeta morphism induces an isomorphism
\begin{equation}\nonumber
\zs:\Ds\isom\Hsmr.
\end{equation}
\end{Conj}
A variant of conjecture \ref{ConjUniv} is formulated in \cite[Conjecture 3.2.2]{KatoViaBdR}, but in the absence of a specific candidate for $\zs$.
\subsection{Compatibility of the ETNC at \Iwagood points}
Let $\la:\Hsmr\fleche A$ be a local ring morphism with values in a reduced ring $A$ and let us write $T_{\la}$ for $\Ts\tenseur_{\Hsmr,\la}A$. We say that $T_{\la}$ is \Iwagood if the complex
\begin{equation}\nonumber
\Cone\left(T_{\la}(-1)^{+}[-1]\fleche\RGamma_{\et}(\Z[1/\Si],T_{\la})\right)
\end{equation}
built out of the morphism 
\begin{equation}\nonumber
\zs(\la):T_{\la}(-1)^{+}\fleche\Hun_{\et}(\Z[1/\Si],T_{\la})
\end{equation}
induced from $\zs$ is acyclic after extension of scalars to the total ring of quotient $Q(A)$ of $A$. Proposition \ref{PropTorsion} shows that the identity is \Iwagood and the main results of \cite{KatoEuler} imply that $\la(f)_{\Iw}:\Hsmr\fleche\Oiwa$ given by the system of eigenvalues of a classical eigencuspform is Iwasawa-suitable.
\begin{Lem}\label{LemIwaGood}
Let $\la:\Hsmr\fleche A$ be a quotient map. Suppose that for each minimal prime $\aid\in\Spec A$, there is a domain $B$ and a commutative diagram 
\begin{equation}\nonumber
\xymatrix{
\Hsmr\ar[r]\ar[d]_{\la\modulo\aid}& B\\
A/\aid\ar[ru]_{\psi(\aid)}
}
\end{equation}
with $\psi(\aid)$ an \Iwagood morphism, then $\la$ is also Iwasawa-suitable.
\end{Lem}
\begin{proof}
Let us first assume that $A$ is a domain and write $\psi$ for $\psi(\aid)$. The canonical isomorphism
\begin{equation}\nonumber
\RGamma_{\et}\left(\Z[1/\Si],\Tla\right)\Ltenseur_{A,\psi}B\isocan\RGamma_{\et}\left(\Z[1/\Si],\Tpsi\right)
\end{equation}
and Nakayama's lemma show first that the cohomology of $\RGamma_{\et}\left(\Z[1/\Si],\Tla\right)$ is concentrated in degree $[1,2]$, then that  $H^{2}_{\et}\left(\Z[1/\Si],\Tla\right)$ is torsion and finally that $H^{1}_{\et}\left(\Z[1/\Si],\Tla\right)$ has $A$-rank at most one. As $\image\z(\psi)$ is a non-zero $B$-module, another application of Nakayama's lemma shows that $H^{1}_{\et}\left(\Z[1/\Si],\Tla\right)/\image\z(\la)$ is $A$-torsion. Hence, $\la$ is Iwasawa-suitable.

Returning to the general case, the first part of the proof shows that for all minimal prime $\aid\in\Spec A$, $\la\modulo\aid$ is Iwasawa-suitable. As $Q(A)$ is a semi-local ring whose maximal primes are in bijections with the minimal primes of $A$, $\la$ is Iwasawa-suitable.
%We show that the $A$-modules $H^{2}_{\et}\left(\Z[1/\Si],\Tla\right)$ and $H^{1}_{\et}\left(\Z[1/\Si],\Tla\right)/\image\zs(\la)$ are torsion modules. Let $\psi$ be one the $\psi(\aid)$. The canonical isomorphism
%\begin{equation}\nonumber
%\RGamma_{\et}\left(\Z[1/\Si],\Tla\right)\Ltenseur_{A,\psi}B\isocan\RGamma_{\et}\left(\Z[1/\Si],\Tpsi\right)
%\end{equation}
%and Nakayama's lemma show first that the cohomology of $\RGamma_{\et}\left(\Z[1/\Si],\Tla\right)$ is concentrated in degree $[1,2]$, then that  $H^{2}_{\et}\left(\Z[1/\Si],\Tla\right)$ is torsion and finally that $H^{1}_{\et}\left(\Z[1/\Si],\Tla\right)$ is generated by at most one element. In order to show that $H^{1}_{\et}\left(\Z[1/\Si],\Tla\right)/\image\zs$ is $A$-torsion, it is consequently sufficient to do so after base-changing to each irreducible component of $A$, in which case it is sufficient to show that $H^{1}_{\et}\left(\Z[1/\Si],T_{\psi(\aid)}\right)/\image\zs(\psi(\aid))$ is torsion. This is our assumption that $\psi(\aid)$ is an \Iwagood morphism.% is  at a motivic point (which we know exist on the irreducible component under examination by theorem \ref{TheoModular}). The result then follows by \cite[Theorem 12.4]{KatoEuler}. %As the morphism $\z(f)_{\Iw}$ is non-zero, the commutative diagram \eqref{DiagNaka} shows that $\image\zs$ is non-
\end{proof}
The following seemingly formal and easy proposition is the key to the proof of our theorem. As explained in the introduction, it can be interpreted as an optimal version of Greenberg's Control Theorem (\cite{GreenbergControl,OchiaiMainConjecture,FouquetOchiai,FouquetX}). It also appears to be false in general and very hard to prove even with supplementary assumptions if fundamental lines are replaced by characteristic ideals of Selmer groups (in particular, our crucial reliance on this proposition justifies the use of fundamental lines rather than the more concrete and usual Selmer modules).
\begin{Prop}\label{PropDetComSigma}
Let
\begin{equation}\nonumber
\xymatrix{
\Hsmr\ar[r]\ar[d]_{\la}& B\\
A\ar[ru]_{\psi}
}
\end{equation}
be a commutative diagram of Iwasawa-suitable ring-morphisms with $B$ a domain. We define
\begin{equation}\nonumber
\Delta_{\Si}(T_{\la})=\Det^{-1}_{A}\RGamma_{\et}\left(\Z[1/\Si],T_{\la}\right)\tenseur_{A}\Det^{-1}_{A}T_{\la}(-1)^{+}
\end{equation}
and 
\begin{equation}\nonumber
\Delta_{\Si}(\Tpsi)=\Det^{-1}_{B}\RGamma_{\et}\left(\Z[1/\Si],\Tpsi\right)\tenseur_{B}\Det^{-1}_{B}\Tpsi(-1)^{+}.
\end{equation}
Then it is possible to choose $(x,y)\in A^{2}$ such that the image of the zeta morphism $\zs(\la)$ is equal to $\frac{x}{y}A$ and $(x',y')\in B^{2}$ such that the image of the zeta morphism $\zs(\psi)$ is equal to $\frac{x'}{y'}B$ so that the natural map
\begin{equation}\nonumber
\Delta_{\Si}(T_{\la})\tenseur_{A,\psi}B{\fleche}\Delta_{\Si}(\Tpsi)
\end{equation}
is an isomorphism which fits into a commutative diagram
\begin{equation}\nonumber
\xymatrix{
\Delta_{\Si}(T_{\la})\ar[d]_{-\tenseur_{A,\psi}B}\ar[rr]^{{\zs(\la)}}&&\frac{x}{y}A\ar[d]^{\psi}\\
\Delta_{\Si}(\Tpsi)\ar[rr]^(0.55){{\zs(\psi)}}&&\frac{x'}{y'}B
}.
\end{equation}
In particular, the morphism $\psi$ extends to a morphism $\frac{x}{y}A\fleche\frac{x'}{y'}B$.
\end{Prop}
In the following, proposition \ref{PropDetComSigma} is almost always applied to $\la$ equal to the identity or the projection to an irreducible component and $\psi$ is a motivic point or a point with values in the ring of integers of a finite extension of $\qp$ with good properties.
\begin{proof}
Let $\pid\in\Spec A$ be the kernel of $\psi:A\fleche B$. As localization at $\pid$ is flat, the complex
\begin{equation}\nonumber
\RGamma_{\et}(\Z[1/\Si],T_{\la}\tenseur_{A}A_{\pid})\simeq\RGamma_{\et}(\Z[1/\Si],T_{\la}\tenseur_{A})\Ltenseur_{A}A_{\pid}
\end{equation}
is a perfect complex of $A_{\pid}$ which commutes with $-\Ltenseur_{A_{\pid}}\kg(\pid)$. As the complex 
\begin{equation}\nonumber
\Cone\left((T_{\la}\tenseur_{A}\kg(\pid))(-1)^{+}\overset{\z(\la)}{\fleche}\RGamma_{\et}(\Z[1/\Si],T_{\la}\tenseur_{A}\kg(\pid))\right)
\end{equation}
is acyclic by definition of $\la$ being Iwasawa-suitable, the complex
\begin{equation}\nonumber
\Cone\left((T_{\la}\tenseur_{A}A_{\pid})(-1)^{+}\overset{\z(\la)}{\fleche}\RGamma_{\et}(\Z[1/\Si],T_{\la}\tenseur_{A}A_{\pid})\right)
\end{equation}
is acyclic by Nakayama's lemma. By functoriality of $\Det$, there is a commutative diagram
\begin{equation}\nonumber
\xymatrix{
\Delta_{\Si}(T_{\la})\ar@{^{(}->}[r]\ar[d]_{-\tenseur_{A}B}&\Delta_{A}(T_{\la}\tenseur_{A}A_{\pid})\ar[r]^(0.7){\sim}\ar[d]_{-\tenseur_{A_{\pid}}\kg(\pid)}&A_{\pid}\ar[d]^{-\tenseur_{A_{\pid}}\kg(\pid)}\\
\Delta_{\Si}(\Tpsi)\ar@{^{(}->}[r]&\Delta_{\Si}(\Tpsi\tenseur\kg(\pid))\ar[r]^(0.6){\sim}&\kg(\pid)
}
\end{equation}
whose vertical arrows are induced by $\psi$. In particular, it is possible to choose $(x,y)\in (A\backslash\pid)^{2}$ such that the image $\zs(\la)\left(\Delta_{\Si}(T_{\la})\right)$ of $\Ds(\Tla)$ through $\zs$ is generated inside $Q(A)$ by $x/y$. Then $\psi(x)/\psi(y)$ is well-defined and generates $\zs(\psi)\left(\Delta_{\Si}(\Tpsi)\right)$.
\end{proof}
The following theorem was discussed in the introduction.
\begin{Theo}\label{TheoCompatibility}
Suppose that $\la:\Hsmr\fleche A$ is a quotient map. Assume that each irreducible components of $A$ admits a motivic point and that the Universal Iwasawa Main Conjecture (conjecture \ref{ConjUniv}) holds. Then the isomorphism
\begin{equation}\nonumber
\zs(\la):\Ds\tenseur_{\Hsmr,\la}Q(A)\isom Q(A)
\end{equation}
induces an isomorphism
\begin{equation}\nonumber
\zs(\la):\Ds\tenseur_{\Hsmr,\la}A\isom A.
\end{equation}
In particular, the Universal Iwasawa Main Conjecture (conjecture \ref{ConjUniv}) implies that the Iwasawa Main Conjecture (conjecture \ref{ConjIMC}) holds at all motivic points
 and that there is a commutative diagram
 \begin{equation}\nonumber
\xymatrix{
\Ds\ar[d]_{-\tenseur_{\Hecke_{\mr}^{\Si}}\Oiwa}\ar[rr]^{{\zs}}&&\Hecke_{\mr}^{\Si}\ar[d]^{\lambda(f)_{\Iw}}\\
\Delta_{\Sigma}\left(T(f)_{\Iw}\right)\ar[rr]^(0.55){{\zs(f)_{\Iw}}}&&\Oiwa.
}
\end{equation}
If there is a quotient map $\pi(\ord):\Hsm\fleche\Hecke_{\mr}^{\Si,\ord}$ onto the nearly Hida-Hecke algebra $\Hecke_{\mr}^{\Si,\ord}$ (that is to say if $\rhobar|G_{\qp}$ is reducible), then there is an isomorphism
\begin{equation}\nonumber
\zs(\ord):\Ds^{\ord}\isom\Hecke_{\mr}^{\Si,\ord}
\end{equation}
which fits in the commutative diagram
\begin{equation}\nonumber
\xymatrix{
\Ds\ar[d]_{-\tenseur_{\Hecke_{\mr}^{\Si}}\Hecke_{\mr}^{\Si,\ord}}\ar[rr]^{{\zs}}&&\Hecke_{\mr}^{\Si}\ar[d]^{\pi(\ord)}\\
\Ds^{\ord}\ar[d]_{-\tenseur_{\Hecke_{\mr}^{\Si,\ord}}\Oiwa}\ar[rr]^{{\zs(\ord)}}&&\Hecke_{\mr}^{\Si,\ord}\ar[d]^{\lambda(f)_{\Iw}}\\
\Delta_{\Sigma}\left(T(f)_{\Iw}\right)\ar[rr]^(0.55){{\zs(f)_{\Iw}}}&&\Oiwa.
}
\end{equation}
at all motivic points of the nearly Hida-Hecke algebra.% $\Hecke_{\mr}^{\Si,\ord}$ if $\Hecke_{\mr}^{\Si,\ord}$ is a quotient of $\Hsmr$ (that is to say if $\rhobar|G_{\qp}$ is reducible).
\end{Theo}
\begin{proof}
According to lemma \ref{LemIwaGood}, the specialization $\la$ is Iwasawa-suitable. The proof is then similar to the proof of proposition \ref{PropDetComSigma} for $\la$ equal to the identity and $\psi$ equal to $\la$ only much simpler as we may choose $x=y=1$ by our assumption that conjecture \ref{ConjUniv} holds (in particular, it is not necessary to assume that the kernel of $\psi$ is a prime ideal).

If $f$ is an eigencuspform, the morphism $\la(f)_{\Iw}:\Hsmr\fleche\Oiwa$ is \Iwagood with values in a domain. According to Hida theory (\cite{HidaInventionesOrdinary,HidaNearlyOrdinary}), each irreducible component of $\Hecke_{\mr}^{\Si,\ord}$ contains a point $\la(f)_{\Iw}:\Hecke_{\mr}^{\Si,\ord}\fleche\Oiwa$ attached to the classical Iwasawa theory of a motivic point. If we assume that conjecture \ref{ConjUniv} holds, then the formalism of the first part of the proof applies to $\la(f)_{\Iw}$ and to the projection onto $\Hecke_{\mr}^{\Si,\ord}$.
\end{proof}

\section{The Taylor-Wiles-Kisin machinery}
The aim of this section is to construct a Taylor-Wiles-Kisin system with coefficients in the completed tensor product $B$ of local universal deformation rings (see proposition \ref{PropTWgeneral} below) and to give a sufficient condition for all the irreducible components of $B$ to be regular rings. When this condition holds, it is possible to use the Taylor-Wiles-Kisin system to approximate a resolution of singularities of the Hecke algebra in a way that is compatible with trivialization of fundamental lines. This is used in section \ref{SubGeneral} below to establish the main theorems of the manuscript.
\begin{DefEnglish}[Axioms of Taylor-Wiles-Kisin systems]\label{DefTaylorWilesKisin}
A Taylor-Wiles-Kisin system is a set $\{B,R,M,(R_{n},M_{n})_{n\in\N}\}$ whose elements satisfy the following properties.
\begin{enumerate}
\item\label{ItemBregular} The ring $B$ is a complete, local, flat $\Ocal$-algebra of dimension $d+1$.
\item The ring $R$ is a $B$-algebra and $M$ is a non-zero $R$-module.
\item There exists $(h,j)\in\N^{2}$ such that for all $n\in\N$, there are maps of $\Ocal$-algebras
\begin{equation}\nonumber
\Ocal[[y_{1},\cdots,y_{h+j}]]\fleche R_{n}\fleche R.
\end{equation}
\suspend{enumerate}
Denote by $\Bbf$ the $h+j+1$-dimensional $B$-algebra $B[[x_{1},\cdots,x_{h+j-d}]]$.
\resume{enumerate}
\item\label{ItQuotient} For all $n\in\N$, $R_{n}$ is a quotient of $\Bbf$.
\item\label{ItSurjectionRn} For all $n\in\N$, there is a surjective map of $B$-algebras $R_{n}\fleche R$  whose kernel is the ideal $(y_{1}\cdots,y_{h})R_{n}$.
\item\label{ItemTW} For all $n\in\N$, there is a surjective map of $R_{n}$-modules $M_{n}\fleche M$ whose kernel is $(y_{1}\cdots,y_{h})M_{n}$.
\item\label{ItemFree} For all $n\in\N$, the annihilator $\bid_{n}$ of $M_{n}$ in $\Ocal[[y_{1},\cdots,y_{h+j}]]$ is included in the ideal
\begin{equation}\nonumber
\left((1+y_{1})^{p^{n}}-1,\cdots,(1+y_{h})^{p^{n}}-1\right)
\end{equation}
and $M_{n}$ is a finite $\Ocal[[y_{1},\cdots,y_{h+j}]]/\bid_{n}$-module free of rank $s$.
\end{enumerate}
\end{DefEnglish}
Let $\TW=\{B,R,M,(R_{n},M_{n})_{n\in\N}\}$ be a Taylor-Wiles-Kisin system and let $n$ be an integer. According to property \ref{ItemFree} of definition \ref{DefTaylorWilesKisin}, $M_{n}$ is a free $\Ocal[[y_{1},\cdots,y_{h+j}]]/\bid_{n}$-module. So it has depth at least $h+j$ as $\Ocal[[y_{1},\cdots,y_{h+j}]]/\bid_{n}$-module. According to property \ref{ItemTW}, $M$ is thus of depth at least $j$ as $\Ocal[[y_{h+1},\cdots,y_{h+j}]]$-module and so is a free $\Ocal[[y_{h+1}\cdots,y_{h+j}]]$-module, necessarily of rank $s$.

It follows axiomtically from definition \ref{DefTaylorWilesKisin} that if
\begin{equation}\nonumber
\{B,R,M,(R_{n},M_{n})_{n\in\N}\}
\end{equation}
is a Taylor-Wiles-Kisin with integers $(h,j)$, then for all $j'\geq j$, the system
\begin{equation}\nonumber
\{B',R',M',(R'_{n},M'_{n})_{n\in\N}\}
\end{equation}
with $B'=B[[x_{1},\cdots,x_{j'-j}]]$, $R'_{n}=R[[x_{1},\cdots,x_{j'-j}]]$, $M'_{n}=M_{n}\tenseur_{R_{n}}R'_{n}$ is a Taylor-Wiles-Kisin with integers $(h,j')$.

For $(m,n)\in\N^{2}$ a pair of integers satisfying $m\geq1$ and $n\geq m$, define $r_{m}$ to be the integer $smp^{m}(h+j)$ and $\cid_{m}$ to be the ideal
\begin{equation}\nonumber
\left(\varpi_{\Ocal}^{m},(1+y_{1})^{p^{m}}-1,\cdots,(1+y_{h})^{p^{m}}-1,y_{h+1}^{p^{m}},\cdots,y_{h+j}^{p^{m}}\right)\Ocal[[y_{1},\cdots,y_{h+j}]].
\end{equation}
Let $\mgot_{R_{n}}^{(r_{m})}$ be the ideal of $R_{n}$ generated by $r_{m}$-th powers of elements of $\mgot_{R_{n}}$. Define
\begin{equation}\nonumber
D_{m,n}\eqdef R_{n}/\left(\cid_{m}R_{n}+\mgot_{R_{n}}^{(r_{m})}\right),\ L_{m,n}\eqdef M_{n}/\cid_{m}M_{n}
\end{equation}
According to \cite[page 1158]{KisinFlat}, $\mgot_{R_{n}}^{(r_{m})}M_{n}$ is included in $\cid_{m}M_{n}$. Hence, the quotient $L_{m,n}$ is endowed with a structure of $D_{m,n}$-module.
\begin{DefEnglish}
Let $\TW=\{B,R,M,(R_{n},M_{n})_{n\in\N}\}$ be a Taylor-Wiles-Kisin system with $M$ of rank $s$ over $\Ocal[[y_{h+1},\cdots,y_{h+j}]]$. For $n\geq m\geq1$ two integers, a patching datum $\PD$ of level $(n,m)$ attached to $\TW$ is a multiplet
\begin{equation}\nonumber
\left(D_{m,n},L_{m,n},\psi_{m,n},\pi_{m,n}^{(1)},\pi_{m,n}^{(2)},\pi_{m,n}^{(3)}\right)
\end{equation}
satisfying the following properties.
\begin{enumerate}
\item The map
\begin{equation}\nonumber
\psi_{m}:\Ocal[[y_{1},\cdots,y_{h+j}]]\fleche D_{m,n}
\end{equation}
is a morphism of local $\Ocal$-algebras.
\item The map
\begin{equation}\nonumber
\pi_{m,n}^{(1)}:D_{m,n}\fleche R/\left(\cid_{m}R+\mgot_{R}^{(r_{m})}\right)
\end{equation}
is a surjective map of $B$-algebras.
\item The map
\begin{equation}\nonumber
\pi_{m,n}^{(2)}:\Bbf\fleche D_{m,n}
\end{equation}
is a surjective map of $B$-algebras.
\item The map
\begin{equation}\nonumber
\pi_{m,n}^{(3)}:L_{m,n}\fleche M/\cid_{m}M
\end{equation}
is a surjective map of $\Bbf$-modules.
\end{enumerate}
A patching datum $\PD$ attached to $\TW$ is a sequence of patching datum of level $(n,m)\in\N^{2}$. For simplicity of notation, the patching datum $(D,L,\psi,\pi^{(1)},\pi^{(2)},\pi^{(3)})$ is denoted by $(D,L)$  when $\psi,\pi^{(1)},\pi^{(2)},\pi^{(3)},m$ and $n$ are implicit from the context.
\end{DefEnglish}

Two patching data $(D_{1},L_{1},\psi_{1},\pi_{1}^{(1)},\pi_{1}^{(2)},\pi_{1}^{(3)})$ and $(D_{2},L_{2},\psi_{2},\pi_{2}^{(1)},\pi_{2}^{(2)},\pi_{2}^{(3)})$ both of level $m$ are isomorphic if there exists a pair $(\phi,\psi)$ of morphisms such that $\phi$ is a morphism of $\Bbf$-algebras making the diagram
\begin{equation}\nonumber
\xymatrix{
\Bbf\ar@{>>}[rr]^(0.5){\pi_{2}^{(2)}}\ar@{>>}[d]_{\pi_{1}^{(2)}}&&D_{2}\ar@{>>}[d]^{\pi_{2}^{(1)}}\\
D_{1}\ar_{\phi}[rru]\ar@{>>}[rr]_{\pi_{1}^{(1)}}&&\frac{R}{\left(\cid_{m}R+\mgot_{R}^{(r_{m})}\right)}
}
\end{equation} 
commutative and $\psi$ is a surjective morphism of $D_{1}$-modules making the diagram 
\begin{equation}\nonumber
\xymatrix{
L_{1}\ar@{>>}[r]^(0.4){\pi_{1}^{(3)}}\ar@{>>}[d]_(0.4){\psi}&M/\cid_{m}M\\
L_{2}\ar@{>>}[ru]_{\pi_{2}^{(3)}}&
}
\end{equation}
commutative. 

Let $\TW$ be a Taylor-Wiles-Kisin system with associated patching datum $\PD$. The cardinality of the ring $D_{m,n}$ is then bounded by the cardinality of $\Bbf/\mgot^{(r_{m})}\Bbf$, which is independent of $n$. Hence, there exists a subsequence $(Q(n))_{n\in\N}$ such that for all $n\geq m\geq1$, the patching datum $(D_{m,Q(n)},L_{m,Q(n)})$ is isomorphic to $(D_{m},L_{m})\eqdef(D_{m,Q(m)},L_{m,Q(m)})$. Let $(D_{m})_{m\geq1}$ and $(L_{m})_{m\geq1}$ be the inverse systems with transition maps induced by the isomorphisms of patching data
\begin{equation}\nonumber
\left(D_{m+1}/(\cid_{m}D_{m+1}+\mgot_{D_{m+1}}^{(r_{m})}),L_{m+1}/\cid_{m}L_{m+1}\right)\simeq(D_{m},L_{m}).
\end{equation}
Let $R_{\infty}$ and $L_{\infty}$ be their inverse limits. The patched system attached to the patching datum $\PD$ is written $(D_{m},L_{m})_{m\geq1}$. The limit object $(R_{\infty},L_{\infty})$ is the inverse limit of the patched system.
\begin{LemEnglish}\label{LemTW}
Let $\TW=\{B,R,M,(R_{n},M_{n})_{n\in\N}\}$ be a Taylor-Wiles-Kisin system whose patched system has limit objects $(R_{\infty},L_{\infty})$ and let $\aid_{\infty}$ be a minimal prime ideal of $R_{\infty}$. Then $R_{\infty}/\aid_{\infty}$ is isomorphic to $\Bbf/\aid=B/\aid[[x_{1},\cdots,x_{h+j-d}]]$ for some minimal prime ideal $\aid$ of $B$ and thus is a regular local ring if and only if $B/\aid$ is one. 
\end{LemEnglish}
\begin{proof}
Recall that according to property \ref{ItemBregular} of definition \ref{DefTaylorWilesKisin}, $B$ is a local ring of dimension $d$. Hence, the ring $\Bbf$ is local of dimension $h+j+1$ and for any minimal prime ideal $\aid$ of $B$, the quotient $\Bbf/\aid$ is a local domain of dimension $h+j+1$ which is a regular local ring if and only if $B/\aid$ is one. It is thus enough to prove that there exists a minimal prime ideal $\aid$ such that $R_{\infty}/\aid_{\infty}$ is isomorphic to $\Bbf/\aid$.

For all integer $m\geq1$, there exists by construction a surjection $\pi^{(2)}_{m}:\Bbf\surjection D_{m}$ and hence a surjection $\pi_{\infty}^{(2)}:\Bbf\surjection R_{\infty}$. Let $\aid_{\infty}$ be a minimal prime ideal of $R_{\infty}$ and let $\aid$ be a minimal prime ideal of $B$ such that $\pi_{\infty}^{(2)}$ induces a surjection from $\Bbf/\aid$ onto $R_{\infty}/\aid_{\infty}$. The maps
$$\psi_{m}:\Ocal[[y_{1},\cdots,y_{h+j}]]\fleche D_{m}$$
for $m\geq1$ induce a map $\psi_{\infty}:\Ocal[[y_{1},\cdots,y_{h+j}]]\fleche R_{\infty}$. According to the remark following definition \ref{DefTaylorWilesKisin}, $L_{\infty}$ is finite and free (of rank $s$) for this $\Ocal[[y_{1},\cdots,y_{h+j}]]$-module structure. Hence, the depth of $L_{\infty}$ as $\Ocal[[y_{1},\cdots,y_{h+j}]]$-module is $h+j+1$ by the Auslander-Buchsbaum and Serre theorem. The depth of a finitely generated module is invariant by change of local noetherian rings so
\begin{equation}\nonumber
\operatorname{depth}_{\Ocal[[y_{1},\cdots,y_{h+j}]]}L_{\infty}=\operatorname{depth}_{R_{\infty}}L_{\infty}.
\end{equation}
The depth of a finitely generated module over a local noetherian ring is less than the dimension of the ring so
\begin{equation}\nonumber
h+j+1=\operatorname{depth}_{R_{\infty}}L_{\infty}\leq\dim R_{\infty}=\dim R_{\infty}/\aid_{\infty} \leq\dim\Bbf/\aid=\dim\Bbf=h+j+1.
\end{equation}
All the inequalities above are consequently equalities. In particular
\begin{equation}\nonumber
\dim R_{\infty}/\aid_{\infty}=\dim\Bbf/\aid.
\end{equation}
The morphism $\Bbf/\aid\surjection R_{\infty}/\aid_{\infty}$ is thus a surjection from a complete, local, noetherian domain to a complete, local ring of the same dimension and hence an isomorphism.
\end{proof}
By Chebotarev's density theorem, there exist a set $X$ whose elements are the empty set and finite sets $Q$ of common cardinality $r>0$ of rational primes $q\notin\Sigma$ such that, for all $n\in\N$, the set
\begin{equation}\nonumber
X_{n}\overset{\operatorname{def}}{=}\{Q\in X|\forall q\in Q,\ q\equiv1\modulo p^{n}\textrm{ and }\rhobar(\Fr(q))\textrm{ has distinct eigenvalues}\}
\end{equation}
is infinite. For $Q\in X$, we put $R_{Q}=\Hecke^{\Si\cup Q}_{\mr}$ where it is understood that $U_{\ell}$ is sufficiently small at all $\ell\in Q$. Under assumption \ref{HypMain} and assumption \ref{HypLocalDef} below, a Taylor-Wiles-Kisin system with favorable commutative algebra properties exists. Since a universal deformation ring (local or global, usual or framed) of $\rhobar$ is naturally isomorphic to the same universal deformation ring for $\rhobar\tenseur\chi$ and since we are interested in the following only on commutative algebra properties of these rings, we may freely twist $\rhobar$ to check these assumptions (accordingly assumptions \ref{HypMain} and \ref{HypLocalDef} are insensitive to twist by a character). 
\begin{HypEnglish}\label{HypLocalDef}
The compact open subgroup $U^{(p)}$ satisfies the following property: if $\ell\nmid p$ belongs to $\Sigma\backslash\Si(\rhobar)$ then it is odd and one of the following holds.
\begin{enumerate}
\item\label{ItemWildInertia} Let $I_{\ell}/P_{\ell}$ be the maximal pro-$p$-quotient of $I_{\ell}$. Then $\rhobar|P_{\ell}$ is non-scalar.
\item\label{ItemDiffplusmoins} $\ell\not\equiv\pm1\modulo p$.
\suspend{enumerate}
In \ref{ItemMoins} below, it is assumed that $\ell\equiv-1\modulo p$.
\resume{enumerate}
\item\label{ItemMoins}Let $\gamma$ be the ratio of the eigenvalue of a lift of $\rhobar(\Fr(\ell))$. Then one of the following holds.
\begin{enumerate}
\item Either $\gamma\neq-1$.
\item\label{ItemMoinsB} Or $\gamma=-1$, and either $\rhobar|_{I_{\ell}}$ is non-scalar or for all motivic point $\la\in\Spec^{\mot}\Hsmr$, the restriction of $\rho_{\la}$ to $G_{\Q_{\ell}}$ is reducible.
\end{enumerate}
\suspend{enumerate}
In \ref{Itemplusramifie} and \ref{Itemplusnonramifie}, it is assumed that $\ell\equiv1\modulo p$. Moreover we assume one of the followings.
\resume{enumerate}
\item\label{Itemplusramifie} The $I_{\ell}$-representation $\rhobar|I_{\ell}$ is non-scalar and the deformation type at $\ell$ does not correspond to a ramified principal series.
\item\label{Itemplusnonramifie} The $I_{\ell}$-representation $\rhobar|I_{\ell}$ is scalar, the deformation type at $\ell$ does not correspond to a ramified principal series. Let $\s$ be a lift of $\Fr(\ell)$. Then
\begin{enumerate}
\item Either $\rhobar(\s)$ has distinct eigenvalues.
\item Or $\rhobar(\s)$ is not diagonalizable.
\item Or the deformation type at $\ell$ is a twist by character.
\end{enumerate}
\end{enumerate}
%
%\begin{enumerate}
%\item Either $\ell\not\equiv\pm1\modulo p$.
%\item Or $\ell\equiv-1\modulo p$ and one of the following holds.
%\begin{enumerate}
%\item If $\alpha,\beta$ are the eigenvalues of $\rhobar(\Fr(\ell))$, then $\alpha/\beta\neq-1$.
%\item If $\alpha,\beta$ are the eigenvalues of $\rhobar(\Fr(\ell))$, then $\alpha/\beta=-1$ and for all motivic point $\la\in\Spec^{\mot}\Hsmr$, the restriction of $\rho_{\la}$ to $G_{\Q_{\ell}}$ is reducible.
%
%\end{enumerate}
%\item Or $\ell\equiv1\modulo p$ and one of the following holds
%\begin{enumerate}
%\item The eigenvalues of $\rhobar(\Fr(\ell))$ are distinct.
%\item The morphism $\rhobar(\Fr(\ell))$ is not diagonalizable and for all motivic point $\la\in\Spec^{\mot}\Hsmr$, the image of $\rho_{\la}|I_{\ell}$ is infinite.
%\end{enumerate}
%\end{enumerate}
\end{HypEnglish}
In the proposition below, we denote by $R_{\ell}$ the universal framed deformation ring of $\rhobar|G_{\Q_{\ell}}$ for $\ell\in\Si$ ($\ell|p$ is possible).
\begin{Prop}\label{PropTWgeneral}
Suppose that assumption \ref{HypLocalDef} holds. Then there exists a Taylor-Wiles-Kisin system $\{B,R^{\square}_{\vide},\Ds^{\square},(R_{n}^{\square},\Delta_{n}^{\square})_{n\in\N}\}$ satisfying the following properties
\begin{enumerate}
\item The integer $j$ of definition \ref{DefTaylorWilesKisin} is equal to $3$.
\suspend{enumerate}
%For $\ell\in\Sigma$, let $R_{\ell}$ be the universal framed deformation ring of $\rhobar|G_{\Q_{\ell}}$.
\resume{enumerate}
\item  The ring $B$ is the completed tensor product over $\Ocal$
\begin{equation}\nonumber
B=\widehat{\underset{\ell\in\Si}{\bigotimes}}R_{\ell}.
\end{equation}
\item Each irreducible component of $B$ is a regular local ring of dimension $6+4\cardinal{\Sigma^{(p)}}$.
\item The ring $R^{\square}_{\vide}$ is the universal framed deformation ring $R^{\square}_{\Sigma}(\rhobar)$ representing framed deformation of $\rhobar$ unramified outside $\Sigma$ and such that any motivic point of $R^{\square}_{\Sigma}(\rhobar)$ is attached to an eigencuspform of level $U'_{p}U^{p}$. The isomorphism
\begin{equation}\nonumber
\Rsr\isom\Hsmr
\end{equation}
induces an isomorphism
\begin{equation}\nonumber
R_{\vide}^{\square}\isom\Hsmr[[Y_{1},Y_{2},Y_{3}]].
\end{equation}
\item For all $n\in\N$, the ring $R_{n}^{\square}$ is the universal framed deformation ring $R_{Q_{n}}^{\square}$ with $Q_{n}\in X_{n}$.
\item For all $n\in\N$, let $Q_{n}\in X_{n}$ be the set such that $R_{n}=R_{Q_{n}}$. Then the $R_{n}^{\square}$-module $\Delta^{\square}_{n}$ is $\Delta^{}_{Q_{n}}\tenseur_{R_{Q_{n}}}R_{Q_{n}}^{\square}$.
\end{enumerate}
\end{Prop}
\paragraph{Remark}The proof given here takes a slight detour as this author does not know how to construct directly a Taylor-Wiles-Kisin system $\{B,R^{\square}_{\vide},M^{\square},(M^{\square}_{n},R_{n}^{\square})\}$ where the modules $M$ and $M_{n}$ are \emph{free} modules over $R^{\square}_{\vide}$ and $R_{n}^{\square}$ respectively.
\begin{proof}
Let $B$ be as in the statement of the proposition. Following the proof of \cite[Proposition (3.2.5)]{KisinFlat}, fix the cardinality of $Q_{n}$ to
\begin{equation}\nonumber
\cardinal{Q_{n}}=\dim_{k}\Hun(G_{\Q,\Sigma},\ad^{0}\rhobar(1))
\end{equation}
and put $r=\cardinal{Q_{n}}+\cardinal{\Sigma}-2$. Then, for all $Q_{n}$, $R_{n}^{\square}$ is a quotient of $R_{\vide}^{\square}[[X_{1},\cdots,X_{r}]]$. Under our assumption \ref{HypMain}, it follows from  the proof of \cite[Theorem (3.4.11)]{KisinFlat} that there exists a Taylor-Wiles-Kisin system $\{B,R^{\square}_{\vide},M^{\square},(M^{\square}_{n},R_{n}^{\square})\}$ where $M^{\square}$ and $M_{n}^{\square}$ are the extension of scalars to $R^{\square}_{\vide}$ and $R^{\square}_{n}$ of suitable Hecke-modules $M$ and $M_{n}$ over $R_{\vide}$ and $R_{n}$ respectively. Again by our assumption \ref{HypMain}, the existence of Taylor-Wiles system for minimal deformations as in \cite{TaylorWiles} and the complete intersection and freeness criterion of \textit{loc. cit.} (see \cite[Section 2.2]{FujiwaraDeformation}), $M$ and $M_{n}$ are free modules over $R_{\vide}$ and $R_{n}$ respectively. Thus $M^{\square}$ and $M_{n}^{\square}$ have maximal depths over $R^{\square}$ and $R_{n}^{\square}$ so are free by the Aulander-Buchsbaum and Serre formula. The axioms defining a Taylor-Wiles-Kisin system remain true after replacing the modules $M$  (resp. $M_{n}$) by an $R$-submodule (resp. an $R_{n}$-module) free of rank 1. Since the property of being a Taylor-Wiles system depends only on the isomorphism class of the modules considered, we may further assume that this Taylor-Wiles system is $\{B,R^{\square}_{\vide},\Ds^{\square},(R_{n}^{\square},\Delta_{n}^{\square})_{n\in\N}\}$ as in the statement of the proposition.

Next, we prove that every irreducible component of $B$ is a regular, local ring. According to the last assertion of lemma \ref{LemTW}, it is enough to prove that  irreducible components of $R_{\ell}$ are regular, local rings. This follows from the explicit computations of these irreducible components in \cite{RamakrishnaFlat,BockleDemuskin,SkinnerWilesDur,DiamondFlachGuo,ShottonDeformation}.

First we assume $\ell|p$. If $\rhobar|G_{\qp}$ is irreducible, then $R_{p}$ is a power-series ring of relative dimension 5 over $\Ocal$ by \cite[Theorem 4.1]{RamakrishnaFlat} (in the flat case) and \cite[Proposition 2.2]{DiamondFlachGuo} (in the general irreducible case). If $\rhobar|G_{\qp}$ is reducible, then the universal framed nearly-ordinary deformation ring $R_{p}^{\ord}$ is a power-series ring of relative dimension 3 by \cite[Lemma 2.2]{SkinnerWilesDur}. By \cite[Corollary 7.4]{BockleDemuskin}, $R_{p}$ is of relative dimension 5 over $\Ocal$ and the kernel of the surjection from $R_{p}$ onto $R_{p}^{\ord}$ is generated by two elements. These two elements thus necessarily form a regular sequence and so $R_{p}$ is a regular local ring of relative dimension 5.

Now we assume that $\ell\nmid p$. In \cite[Section 5]{ShottonDeformation}, all irreducible components $X$ of $R_{\ell}$ are explicitly computed. The outcome of these calculations is that $R_{\ell}$ is a regular, local ring of dimension 5 under the hypotheses of assumptions \ref{HypLocalDef}.

Finally, each prime $\ell\in\Sigma^{(p)}$ contributes 4 to the relative dimension of $B$ over $\Ocal$ and $R_{p}$ contributes 5 to the relative dimension of $B$ over $\Ocal$ so $\dim B=6+4\cardinal{\Sigma^{(p)}}$.
\end{proof}
Denote by $(R^{\square}_{\infty},\Delta^{\square}_{\infty})$ the limit objects of the patching system of the Taylor-Wiles system of the previous proposition.

\section{Proof of the main theorems}
\subsection{Statements}
We repeat the statements of our main results.
%\begin{Theo}\label{TheoMinCorps}
%Assume $\Hsmr$ to be minimal. Then the following assertions are equivalent.
%\begin{enumerate}
%\item There exists a motivic point $\la\in\Spec^{\mot}\Hsm$ such that
%\begin{equation}\nonumber
%\z(\la)_{\Iw}:\Delta_{\Oiwa}\left(T_{\la,\Iw}\right)\isom\Oiwa,
%\end{equation}
%that is to say conjecture \ref{ConjIMC} holds for $T_{\la,\Iw}$.
%\item For all motivic point $\la\in\Spec^{\mot}\Hsm$
%\begin{equation}\nonumber
%\z(\la)_{\Iw}:\Delta_{\Oiwa}\left(T_{\la,\Iw}\right)\isom\Oiwa,
%\end{equation}
%that is to say conjecture \ref{ConjIMC} holds for $T_{\la,\Iw}$.%If in addition $L(f,\chi,r)\neq0$, then conjecture \ref{ConjTNCIntro} for $\Wcal(\la)$ at $p$ holds (the Tamagawa Number Conjecture is true for the motive $\Wcal(f)(r)$ at $p$).
%\item The zeta morphism
%\begin{equation}\nonumber
%\zs:\Ds(\Ts)\tenseur Q(\Hsm)\isom Q(\Hsm)
%\end{equation}
%induces an isomorphism
%\begin{equation}\nonumber
%\zs:\Ds(\Ts)\isom\Hsm.
%\end{equation}
%In other words, the Iwasawa Main Conjecture (conjecture \ref{ConjUniv}) holds for $\Ts$.
%\end{enumerate}
%\end{Theo}

\begin{Theo}\label{TheoCorps}
Assume that $\rhobar$ satisfies \ref{HypMain} and \ref{HypLocalDef}.
Then the isomorphism \eqref{EqZetaElementSigma}
\begin{equation}\nonumber
\zs:\Ds(\Ts)\tenseur Q(\Hsmr)\isom Q(\Hsmr)
\end{equation}
of proposition \ref{PropTorsion} induces an inclusion 
\begin{equation}\nonumber
\zs:\Ds(\Ts)^{-1}\plonge\Hsmr.
\end{equation}
In particular, the following assertions are equivalent.
\begin{enumerate}
\item There exists a motivic point $\la\in\Spec^{\mot}\Hsm$ such that
\begin{equation}\nonumber
\z(\la)_{\Iw}:\Delta_{\Oiwa}\left(T_{\la,\Iw}\right)\isom\Oiwa.
\end{equation}
\item For all motivic point $\la\in\Spec^{\mot}\Hsm$
\begin{equation}\nonumber
\z(\la)_{\Iw}:\Delta_{\Oiwa}\left(T_{\la,\Iw}\right)\isom\Oiwa.
\end{equation}
\item The zeta morphism
\begin{equation}\nonumber
\zs:\Ds(\Ts)\tenseur Q(\Hsm)\isom Q(\Hsm)
\end{equation}
induces an isomorphism
\begin{equation}\nonumber
\zs:\Ds(\Ts)\isom\Hsm.
\end{equation}
In other words, the Iwasawa Main Conjecture (conjecture \ref{ConjUniv}) holds for $\Ts$.
\end{enumerate}
\end{Theo}
We record the following immediate corollary.
\begin{Cor}\label{CorKatoStronger}
Under the hypotheses of theorem \ref{TheoCorps}, there is a divisibility
\begin{equation}\label{EqDivCorKato}
\carac_{\Oiwa}H^{2}_{\et}(\Z[1/p],T(f)_{\Iw})\mid\carac_{\Oiwa}H^{1}_{\et}(\Z[1/p],T(f)_{\Iw})/\Oiwa\cdot\z(f)_{\Iw}
\end{equation}
for all $f$ attached to a motivic point of $\la:\Hsmr\fleche\Ocal$
\end{Cor} 
In \cite[Theorem 12.5]{KatoEuler}, divisibility \eqref{EqDivCorKato} is proved except when $\rho_{f}|G_{\qp}$ has ordinary reduction but not potential good reduction.
\begin{proof}
Theorem \ref{TheoCorps} shows that the image of $\Ds(\Ts)^{-1}$ through $\zs$ is included in $\Hsmr$. The corollary then follows at once from the commutativity of the diagram
\begin{equation}\nonumber
\xymatrix{
\ar[d]_{-\tenseur_{\Hsmr,\la}\Oiwa}\Ds^{-1}\ar[rr]^{\zs}&&\Hsmr\ar[d]^{\la}\\
\Delta(f)^{-1}_{\Iw}\ar[rr]^{\z(f)_{\Iw}}&&\Frac(\Oiwa).
}
\end{equation}
\end{proof}  
The following corollary exploits the compatibility of the Galois representations attached to two motivic points of $\Ts$ to provide examples of eigencuspforms $g$ with non-trivial $H^{2}_{\et}(\Z[1/p],T(g)_{\Iw})$. This property was noticed in \cite{EmertonPollackWeston} under the supplementary hypotheses that $f$ and $g$ are both $p$-ordinary and that one of them has trivial $\mu$-invariant. Attentive readers will notice that the full strength of theorem \ref{TheoCorps} is used only once in the proof: to ensure that the Iwasawa Main Conjecture is true for the second eigencuspform $g$. When this is known beforehand - for instance because the hypotheses of \cite[Theorem 1.1]{FouquetPMB} or the hypotheses of \cite[Theorem 1.7]{FouquetWan} are satisfied - then the assumptions of the corollary may be relaxed accordingly.
\begin{Cor}\label{CorEPW}
Under the hypotheses of theorem \ref{TheoCorps}, let
\begin{equation}\nonumber
\la:\Hsmr\fleche\Ocal,\ \la':\Hsmr\fleche\Ocal
\end{equation}
be two motivic points attached to two eigencuspforms $f\in S_{k}(\Gamma_{1}(N))$ and $g\in S_{k'}(\Gamma_{1}(N'))$. Assume that the following holds.
\begin{enumerate}
\item Kato's class $\z(f)_{\Iw}$ generates $\Hun_{\et}(\Z[1/p],T(f)_{\Iw})$.
%\item The $G_{\qp}$-representation $\rho_{f}|G_{\qp}$ has no $G_{\qp}$-stable quotient isomorphic to $\chi\chi_{\cyc}^{-1}$ for $\chi$ a finite order character of $\Gamma$.
\item There exists $\ell|N$ such that the Euler factor attached to $T(f)^{*}(1)$ is not a $p$-unit.
\item For all $\ell|N'$, the Euler factor attached to $T(g)^{*}(1)$ is a $p$-unit.
\end{enumerate}
Then $H^{2}_{\et}(\Z[1/p],T(g)_{\Iw})$ is not trivial and does not contain any non-zero pseudo-null submodules and $\z(g)_{\Iw}$ does not generate $\Hun_{\et}(\Z[1/p],T(g)_{\Iw})$.
\end{Cor}
\begin{proof}
%For all $\ell\nmid p$, the Euler polynomial
%\begin{equation}\nonumber
%\Eul_{\ell}(T^{\ord}(\aid),X)\eqdef \det\left(1-\Fr(\ell)X|T^{\ord}(\aid)^{I_{\ell}}\right)
%\end{equation} then defines a continuous function on $\Spec\Hsmrord/\aid$. In particular, our second hypothesis remains true for any motivic point $\psi$ which factors through $\Hsmrord/\aid$. There exists such points with weight $k>2$, in which case $\rho_{\psi}|G_{\qp}$ has no $G_{\qp}$-stable quotient isomorphic to $\chi\chi_{\cyc}^{-1}$ for $\chi$ a finite order character of $\Gamma$. 
Let us first denote by $\Ocal_{n}$ the ring of integers of $\Q_{n}$. Then
\begin{align}\nonumber
\Hun_{\et}(\Z[1/p],T(f)_{\Iw})&=\limproj{n}\ \Hun_{\et}(\Ocal_{n}[1/p],T(f))=\limproj{n}\ \Hun_{\et}(\Ocal_{n}[1/\Si],T(f))\\\nonumber
&=\Hun_{\et}(\Z[1/\Si],T(f)_{\Iw})
\end{align}
where the second equality comes from the fact that for all finite sub-extension $F$ of $\Q_{\infty}$, the image in $\Hun_{\et}(\Ocal_{F}[1/\Si],T(f))$ of a class in the inverse limit of the $\Hun_{\et}(\Ocal_{n}[1/\Si],T(f))$ is unramified at all primes not dividing $p$ (see for instance \cite[Lemma 8.5 (2)]{KatoEuler}).

In addition to our other assumptions, suppose first that $\rho_{f}|G_{\qp}$ has no $G_{\qp}$-stable quotient isomorphic to $\chi\chi_{\cyc}^{-1}$ for $\chi$ a finite order character of $\Gamma$. Then \cite[Theorem 12.5]{KatoEuler} yields a divisibility 
\begin{equation}\label{EqDivKato}
\carac_{\Oiwa}H^{2}_{\et}(\Z[1/p],T(f)_{\Iw})\mid\carac_{\Oiwa}H^{1}_{\et}(\Z[1/p],T(f)_{\Iw})/\Oiwa\cdot\z(f)_{\Iw}.
\end{equation}
As $\z(f)_{\Iw}$ is by assumption a basis of $\Hun_{\et}(\Z[1/p],T(f)_{\Iw})$, then the previous divisibility is an equality and so the zeta morphism yields an isomorphism
\begin{equation}\nonumber
\z(f)_{\Iw}:\Delta_{\Oiwa}(f)\isom\Oiwa.
\end{equation}

By construction, 
\begin{equation}\nonumber
\zs(f)_{\Iw}=\left(\produit{\ell\in\Si^{(p)}}{}\Eul_{\ell}(T(f)_{\Iw}^{*}(1))\right)\z(f)_{\Iw}
\end{equation}
The hypothesis on $\Eul_{\ell}(T(f)^{*}(1))$ for $\ell\in\Si^{(p)}$ thus implies that
\begin{equation}\nonumber
\Oiwa\cdot\zs(f)_{\Iw}\subsetneq\Oiwa\cdot\z(f)_{\Iw}=\Hun_{\et}(\Z[1/\Si],T(f)_{\Iw})
\end{equation}
or equivalently that
\begin{equation}\nonumber
\carac_{\Oiwa}\Hun_{\et}(\Z[1/\Si],T(f)_{\Iw})/\Oiwa\cdot\zs(f)_{\Iw}\neq\Oiwa.
\end{equation}
According to the Iwasawa Main Conjecture, then
\begin{equation}\nonumber
\carac_{\Oiwa}H^{2}_{\et}(\Z[1/\Si],T(f)_{\Iw})\neq\Oiwa.
\end{equation}
In particular, $H^{2}_{\et}(\Z[1/\Si],T(f)_{\Iw})$ is not the trivial torsion $\Oiwa$-module. The complexes $\RGamma_{\et}(\Z[1/\Si],\Ts)$ and $\RGamma_{\et}(\Z[1/\Si],T(g)_{\Iw})$ have cohomology concentrated in degree $[1,2]$ and commute with arbitrary base-change. Nakayama's lemma thus shows first that
\begin{equation}\nonumber
H^{2}_{\et}(\Z[1/\Si],\Ts)\neq0
\end{equation}
then that
\begin{equation}\nonumber
H^{2}_{\et}(\Z[1/\Si],T(g)_{\Iw})\neq0.
\end{equation}
The hypothesis on $\Eul_{\ell}(T(g)^{*}(1))$ for $\ell\in\Si^{(p)}$ implies by Tate's duality that 
\begin{equation}\nonumber
\sommedirecte{\ell\in\Si^{(p)}}{}H^{2}(G_{\ql},T(g))=0.
\end{equation}
The excision triangle relating étale cohomology with cohomology with compact support then yields the isomorphism
\begin{equation}\nonumber
H^{2}_{\et}(\Z[1/\Si],T(g)_{\Iw})=H^{2}_{c}(\Z[1/\Si],T(g)_{\Iw}).
\end{equation}
Finally, the $\Oiwa$-module $H^{2}_{\et}(\Z[1/\Si],T(g)_{\Iw})$ is torsion. Hence, all the hypotheses of \cite[Proposition 9.3.1]{SelmerComplexes} are satisfied. So $H^{2}_{\et}(\Z[1/\Si],T(g)_{\Iw})$ has no non-zero pseudo-null submodules. As it is non-trivial, it has non-trivial characteristic ideal. Hence
\begin{equation}\nonumber
\carac_{\Oiwa}H^{2}_{\et}(\Z[1/\Si],T(g)_{\Iw})\neq\Oiwa.
\end{equation}
On the other hand
\begin{equation}\nonumber
\z(g)_{\Iw}:\Delta_{\Oiwa}(f)\isom\Oiwa
\end{equation}
by theorem \ref{TheoCorps}. This entails that $\z(g)_{\Iw}$ is not a generator of $\Hun_{\et}(\Z[1/p],T(g)_{\Iw})$.

In the remaining part of this proof, we show for the sake of completeness that everything above carries over even in the absence of our supplementary assumption. A technical trick to bypass the fact that some complexes appearing naturally in the argument might not be perfect is required, making the proof slightly technical. As relaxing the assumption only provides a minimal improvement on the result, the reader is advised to skip the rest of this proof on first reading.

Now our supplementary assumption is relaxed (if our supplementary assumption does not hold, then $\rho_{f}|G_{\qp}$ is ordinary but does not have potential good reduction; this fact will however not be used as the results proved below are of independent interests even when our supplementary assumption holds). Let $\aid\in\Spec\Hsm$ be a minimal prime such that $\la$ factors through $\Hsaid\eqdef\Hsm/\aid$ and denote by $\pi:\Hsmr\fleche\Hsaid$ the corresponding \Iwagood specialization. By \cite[Proposition 3.14]{FouquetWan}, there exists a zeta morphism
\begin{equation}\label{EqZetaIrr}
\zaid:T_{\pi}(-1)^{+}\fleche\Hun_{\et}(\Z[1/p],T_{\pi})
\end{equation}
equal to 
\begin{equation}\label{EqZaidZs}
\zaid=\left(\produit{\ell\in\Sp}{}\Eul_{\ell}(T_{\pi}^{*}(1))\right)^{-1}\zs(\pi)
\end{equation}
and such that the following diagram commutes 
\begin{equation}\label{EqZetaComposanteIrr}
\xymatrix{
T_{\pi}(-1)^{+}\ar[d]\ar[rr]^{\zaid}&&\Hun_{\et}(\Z[1/p],T_{\pi})\ar[d]\\
T(f)_{\Iw}(-1)^{+}\ar[rr]^{\ziwa}&&\Hun_{\et}(\Z[1/p],T(f)_{\Iw})
}
\end{equation}
for all motivic point $\lambda({f}):\Hsm\fleche\Ocal$ factoring through $\pi$ (here $\z(f)_{\Iw}$ denotes the zeta morphism induced by Kato's class). In \cite[Definition 2.10]{FouquetX} is defined a free $\Hsaid$-module 
\begin{equation}\label{EqDefDeltaInt}
\Delta(T_{\pi})\eqdef\Ds(T_{\pi})\tenseur\produittenseur{\ell\in\Sigma^{(p)}}{}\left(\Det_{\Hsaid}\RGamma(G_{\ql},T_{\pi})\tenseur_{\Hsaid}\Det^{-1}_{\Hsaid}\frac{\Hsaid}{\Eul_{\ell}(T_{\pi})\Hsaid}\right)
\end{equation}
which is a free of rank 1, which is compatible with base change at motivic points in the sense that the natural map
\begin{equation}\label{EqControlIrr}
\Delta(T_{\pi})\tenseur_{\Hsaid,\la}\Oiwa\fleche\Delta(\Tla)_{\Iw}
\end{equation}
is an isomorphism for all motivic point $\la:\Hsmr\fleche\Ocal$ factoring through $\Hsaid$ (see \cite[Theorem 4.4]{FouquetX} though the particular case needed here is an elementary instance of this result). The free $\Hsaid$-module $\Delta(T_{\pi})$ sis canonically identified with
\begin{equation}\label{EqMachinDet}
\Det^{-1}_{\Hsaid}\RGamma\left(\Z[1/p],T_{\pi}\right)\tenseur_{\Hsaid}\Det^{-1}_{\Hsaid}T_{\pi}(-1)^{+}
\end{equation}
when the object \eqref{EqMachinDet} is defined (the potential problem being that the complex $\RGamma\left(\Z[1/p],T_{\pi}\right)$ might not be a perfect complex of $\Hsaid$-modules as $H^{0}(I_{\ell},-)$ does not send perfect complexes to prefect complexes in general).

Combining \eqref{EqZaidZs}, \eqref{EqDefDeltaInt} and Tate duality shows that the image of $\Ds(T_{\pi})$ inside $\Frac(\Hsaid)$ through $\zs(\pi)$ is equal to the image of $\Delta(T_{\pi})$ through $\z(\aid)$. We compute this image. By construction and our assumption \ref{HypMain}, the map
\begin{equation}\nonumber
\Hun_{\et}\left(\Z[1/p],T_{\pi}\right)\tenseur_{\Hsaid,\la}\Oiwa\fleche\Hun_{\et}\left(\Z[1/p],T(f)_{\Iw}\right)
\end{equation}
is injective. Its image contains the image of $\z(\aid)$ so contains $\Oiwa\cdot\z(f)_{\Iw}$ by \eqref{EqZetaComposanteIrr}. So it is surjective by our assumption that $\z(f)_{\Iw}$ generates $\Hun_{\et}\left(\Z[1/p],T(f)_{\Iw}\right)$. It is thus an isomorphism, which means that the zeta morphism \eqref{EqZetaIrr} is itself an isomorphism. In particular, $\Hun_{\et}(\Z[1/\Si],T_{\pi})$ is free of rank 1. As the complex $\RGamma_{\et}(\Z[1/\Si],T_{\pi})$ is a perfect complex of $\Hsaid$-modules, this means that $H^{2}_{\et}(\Z[1/\Si],T_{\pi})$ is itself a perfect complex of $\Hsaid$-modules and that the image of  $\Ds(T_{\pi})$ inside $\Frac(\Hsaid)$ through $\zs(\pi)$ is equal to
\begin{equation}\label{EqImageDeZs}
\left(\Det^{-1}_{\Hsaid}H^{2}_{\et}(\Z[1/\Si],T_{\pi})\right)\Hsaid.
\end{equation}
Let $\la':\Hsm\fleche\Ocal$ be a motivic point which factors through $\Hsaid$. Then \eqref{EqZetaComposanteIrr}, \eqref{EqControlIrr} and \eqref{EqImageDeZs} together show that the image of $\Delta(T_{\la'})$ through $\z(\la')$ is equal to 
\begin{equation}\nonumber
\left(\Det^{-1}_{\Hsaid}H^{2}_{\et}(\Z[1/\Si],T_{\pi})\right)\tenseur_{\Hsaid,\la'}\Oiwa.
\end{equation} 
As $H^{2}_{\et}(\Z[1/\Si],T_{\pi})$ is the last non-zero cohomology group of a complex which commutes with $-\Ltenseur_{\Hsaid,\la'}\Oiwa$, this image is thus
\begin{equation}\label{EqInclusionPasStricte}
\left(\Det^{-1}_{\Hsaid}H^{2}_{\et}(\Z[1/\Si],T_{\la'})\right)\Oiwa\subset\Oiwa.
\end{equation}
Hence, we have obtained that for all such motivic point, the image of $\Delta(T_{\la'})$ through $\z(\la')$ is included inside $\Oiwa$. The set of $\la'$ such that $T_{\la'}$ does not satisfy our supplementary hypothesis is the set of eigencuspforms $f$ of weight 2 with ordinary reduction but not potential good reduction whose system of eigenvalues factors through $\Hsaid$. This set is empty if $\rhobar|G_{\qp}$ is irreducible. Otherwise, Hida theory guarantees that there exists motivic points of $\Spec\Hsaid$ of weight $k>2$. Hence, there exists such a $\la'$ satisfying our supplementary hypothesis. In that case, the divisibility \eqref{EqDivKato} holds, which means that the image of $\Delta(T_{\la'})$ through $\z(\la')$ is an invertible module $\frac{1}{y}\Oiwa$ generated by a non-zero element whose inverse is in $\Oiwa$. Combining this two property, we see that $y$ is a unit. This means that the containment \eqref{EqInclusionPasStricte} cannot be strict, and hence that the image of $\Ds(T_{\pi})$ through $\zs(\pi)$ is $\Hsaid$; in other words, the Iwasawa Main Conjecture holds for $T_{\pi}$. By compatibility of the Iwasawa Main Conjecture at \Iwagood points, in particular at motivic points, we conclude that the Iwasawa Main Conjecture holds for all motivic points factoring through $\pi$.

\end{proof}
\subsection{Numerical examples}\label{SubExample}
This section contains two examples of numerical properties of special values of $L$-functions of eigencuspforms. In both cases, $p$-adic properties of Kato's class and of the second étale cohomology group of a modular motive are predicted by studying congruences with another eigencuspform. These phenomena are accounted for by congruences between various eigencuspforms if conjecture \ref{ConjUniv} is known to hold, but they elude the Iwasawa Main Conjecture for one or even both eigencuspforms involved. Further similar examples for various primes or eigencuspforms of various levels and weights are easily found by an inspection of level-raising congruences. All computation were conducted using the Pari/GP software (\cite{PARI2}).
\subsubsection{$p=3$}Let
\begin{equation}\nonumber
\rhobar:G_{\Q}\fleche\GL_{2}(\Fp_{3})
\end{equation}
be the only $G_{\Q}$-representation of Serre weight 2 and level $40$ satisfying 
\begin{equation}\nonumber
\tr(\rhobar(\Fr(\ell)))=\begin{cases}
-1&\textrm{if $\ell=7$,}\\
1&\textrm{if $\ell=11$,}\\
1&\textrm{if $\ell=13$,}\\
-1&\textrm{if $\ell=17$,}\\
1&\textrm{if $\ell=19$,}\\
1&\textrm{if $\ell=23$.}\\
\end{cases}
\end{equation}
Then $\rhobar$ is surjective and $\rhobar|G_{\Q_{3}}$ is irreducible. Hence, $\rhobar$ satisfies assumption \ref{HypMain}.

Put $N=1640=2^{3}\times 5\times 41$. There exists a motivic point of $\Spec\Hsm$ attached to the eigencuspform $f\in S_{2}(\Gamma_{0}(40))$ attached to the elliptic curve
\begin{equation}\nonumber
E_{1}:y^{2}=x^{3}-7x-6.
\end{equation}
Note that $E$ has good supersingular reduction at $p=3$ with $a_{p}(f)=0$ (in particular, the results of \cite{EmertonPollackWeston} do not apply to $f$). There exists exactly a motivic point of $\Spec\Hsm$ attached to an eigencuspform $g\in S_{4}(\Gamma_{0}(1640))$. More precisely, the eigencuspform $g$ has coefficients in an extension $K/\Q$ of degree 15. There is a single prime $\pid\in\Spec\Ocal_{K}$ above $3$ such that $f\equiv g\modulo\pid$. The eigencuspform $g$ has finite, non-zero slope at $\pid$ and $a_{41}(g)$ is equal to 41. Note that $\pi(g)_{41}$ is a Steinberg representation. As the ratio of eigenvalues of $\rhobar(\Fr(\ell))$ at $\ell=41$ is $-1$, $\rhobar$ and $\Hsm$ satisfy condition \eqref{ItemMoinsB} of assumption \ref{HypLocalDef}.

  The relevant Euler factor of $f$ at 41 evaluated at 1 is congruent to 0 modulo 3 whereas the Euler factor of $g$ at 41 evaluated at 1 is congruent to 2 modulo 3. All Euler factors at $\ell\in\Si$ for $f$ and $g$ are $3$-units. Finally, $E_{1}$ has analytic rank 0 and it is easily checked that the special value of its $L$-function is a 3-adic unit (for instance, the Kolyvagin class $\kg(7\times 67)$ is a $3$-adic unit). Hence, $T(f)_{\Iw}$ has trivial second étale cohomology group. Hence, $f$ and $g$ satisfy the hypotheses of corollary \ref{CorEPW}. As a consequence, the universal Iwasawa Main Conjecture is true for $\rhobar$ and the Iwasawa Main Conjecture is true for $T(g)_{\Iw}$. Moreover, the $\Oiwa$-module $H^{2}_{\et}(\Z[1/p],T(g)_{\Iw})$ is non-trivial and has non non-zero pseudonull submodules (a prediction that the Iwasawa Main Conjecture alone does not make).
\subsubsection{$p=5$}
Let us consider
\begin{equation}\nonumber
\rhobar_{E}:G_{\Q}\fleche\GL_{2}(\Fp_{5})
\end{equation}
the residual representation attached to the points of $5$-torsion of the elliptic curve
\begin{equation}\nonumber
E:y^{2}=x^{3}-1825x+306625
\end{equation}
which has conductor $N=2^{2}\times 5^{2}\times 23$. Then $\rhobar_{E}|G_{\qp}$ is absolutely irreducible and $\rhobar_{E}$ is surjective. Hence, $\rhobar_{E}$ satisfies assumption \ref{HypMain}.

The Artin conductor of $\rhobar_{E}$ is prime to $\ell=23$. As $E$ has bad additive reduction at $p=5$ and has no auxiliary prime $\ell'||N$ of bad reduction such that $\rhobar_{E}$ is ramified at $\ell'$, the techniques of \cite{SkinnerUrban,JetchevSkinnerWan,FouquetWan} are not applicable. As will be seen momentarily, the full Iwasawa Main Conjecture for $E$ at $p=5$ cannot be deduced from the results of \cite{KatoEuler} either, as $H^{2}_{\et}(\Z[1/p],T_{p}E\tenseur\La_{\Iw})$ is not trivial.

The only non-trivial Euler factor of $E$ at primes of bad reduction is at $\ell=23$. This Euler factor is equal to 22/23 and hence a $5$-adic unit. Because $\rhobar_{E}|I_{23}$ is trivial, we may level-lower $\rho_{E}$ at $\ell=23$. Indeed, if $E'$ is the elliptic curve
\begin{equation}\nonumber
E':y^{2}=x^{3}-x^{2}-33x+62
\end{equation}
of conductor $N'=2^{2}\times 5^{2}$, then $\rho_{E'}$ is a lift of $\rhobar_{E}$ unramified at $23$. The Euler factor of $E'$ at $\ell=23$ is equal to
\begin{equation}\nonumber
\frac{8115390}{6436343}
\end{equation}
and hence has 5-adic valuation 1. The elliptic curve $E'$ has analytic rank 0 and it is easy to check that the special value $L(E',1)/\Omega_{E}$ is a $5$-adic unit (for instance, the Kolyvagin class $\kg(61\times 311)$ is a $5$-adic unit). Finally, $\ell\not\equiv\pm1\modulo 5$. So $\rhobar$ and $\Hsm$ satisfy assumption \eqref{ItemDiffplusmoins} of \ref{HypLocalDef}.

It follows that $\rho_{E}$ and $\rho_{E'}$ satisfy the hypotheses of corollary \ref{CorEPW}. We conclude that the Iwasawa Main Conjecture is true for $E$ and that the $\La_{\Iw}$-module $H^{2}_{\et}(\Z[1/p],T_{p}E\tenseur\La_{\Iw})$ is not trivial and has no non-zero pseudo-null submodule (a prediction that the Iwasawa Main Conjecture alone does not make). More precisely, since the only non-trivial contribution to $H^{2}_{\et}(\Z[1/\Si],\Ts)$ can come from the Euler factor $\Eul_{\ell}\left(T(\aid)\right)$ on the irreducible component of of $\Hsmr$ containing the point attached to $E'$, $H^{2}_{\et}(\Z[1/\Si],\Ts)$ is cyclic over $\Hsmr$ by Nakayama's lemma. It follows that $H^{2}_{\et}(\Z[1/p],T_{p}E\tenseur\La_{\Iw})$ is $\La_{\Iw}$-cyclic with characteristic ideal generated by a distinguished polynomial of degree 1.

In particular, the leading term formula of \cite{PerrinRiouLpadique} yields that the $5$-adic valuation of 
\begin{equation}\label{EqBSDalg}
\frac{\cardinal{\Sha(E/\Q)[5^{\infty}]}\produit{\ell|N}{}\Tam_{\ell}(E/\Q)}{\cardinal{E(\Q)_{\tors}}^{2}}
\end{equation}
is equal to 1. Since $\cardinal{\Sha(E/\Q)[5^{\infty}]}$ is a square and $\cardinal{E(\Q)_{\tors}}^{2}$ is equal to $1$, this power of 5 comes from the Tamagawa numbers of $E$. Since $E$ is ramified at only one prime where $\rhobar$ is unramified, namely $\ell=23$, we conclude that the Tamagawa number of $E$ at $\ell=23$ has 5-adic valuation 1 (a fact that is of course also easily confirmed experimentally).

The curve $E$ has analytic rank 1. According to \cite{GrossZagier,KolyvaginEuler}, $\Sha(E/\Q)$ is thus a finite group. We have just deduced algebraically that its $5$-torsion part is trivial. Since moreover the Iwasawa Main Conjecture is true for $E$ at 5, verifying the Birch and Swinnerton-Dyer Conjecture is true for $E$ at $5$ amounts by the Gross-Zagier formula to verifying that there exists a Heegner point on $E$ generating a subgroup of index $n$ in the Mordell-Weil group with $n$ of 5-adic valuation 1. Indeed, $E(\Q)$ is a free abelian group of rank 1 generated by the point $P=[-15,575]$ and one can check experimentally that the Heegner point of conductor $-79$ is equal to $-30P$ and so has index $30$ in the Mordell-Weil group.

\subsection{Proof of the main theorem}\label{SubGeneral}
Recall that $R_{\ell}$ the universal framed deformation ring of $\rhobar|G_{\Q_{\ell}}$, that $B$ is the completed tensor product $\widehat{\underset{\ell\in\Si}{\bigotimes}}R_{\ell}$  and that a Taylor-Wiles-Kisin system
\begin{equation}\nonumber
\{B,R^{\square}_{\vide},\Ds^{\square},(R_{n}^{\square},\Delta_{n}^{\square})_{n\in\N}\}
\end{equation}
with coefficients in $B$ was constructed in proposition \ref{PropTWgeneral}. Let $(D_{n},L_{n})_{n\geq 1}$ be the patched system attached to the sytem $\{B,R^{\square}_{\vide},\Ds^{\square},(R_{n}^{\square},\Delta_{n}^{\square})_{n\in\N}\}$. 

The following lemma shows that if the zeta morphism does not identify $\Ds(\Ts)$ with $\Hsm$, then this failure of the universal Iwasawa Main Conjecture propagates by level-raising within the patched system $(D_{n},L_{n})_{n\geq 1}$.
\begin{LemEnglish}\label{LemRemonteBis}
Suppose that in the isomorphism
\begin{equation}\label{EqObjectifTer}
\Ds(\Ts)^{-1}\simeq\Det_{\Hsm}[y\Hsm\fleche x\Hsm]
\end{equation}
induced by $\zs$, $y$ cannot be chosen invertible in $\Hsm$. Then there exist a sequence $a\in\N^{\N}$ and choices of trivializations
\begin{equation}\nonumber
\z^{\square}_{a_{n}}:\left(\Delta_{a_{n}}^{\square}\right)^{-1}\simeq\Det_{R_{a_{n}}^{\square}}[y_{a_{n}}R_{a_{n}}^{\square}\fleche x_{a_{n}}R_{a_{n}}^{\square}]
\end{equation}
such that $x_{a_{n}}/y_{a_{n}}\notin R_{a_{n}}^{\square}$ for all $n\in\N$ and such that the diagram
\begin{equation}\label{DiagCommutativeDamPatched}
\xymatrix{
\left(\Delta_{a_{m}}^{\square}\right)^{-1}\tenseur_{R_{a_{m}}^{\square}}D_{a_{m}}\ar[d]\ar[r]&\Det_{D_{a_{m}}}[y_{a_{m}}D_{a_{m}}\fleche x_{a_{m}}D_{a_{m}}]\ar[d]\\
\left(\Delta_{a_{n}}^{\square}\right)^{-1}\tenseur_{R_{a_{n}}^{\square}}D_{a_{n}}\ar[r]&\Det_{D_{a_{n}}}[y_{a_{n}}D_{a_{n}}\fleche x_{a_{n}}D_{a_{n}}]
}
\end{equation}
is commutative.
\end{LemEnglish}
\begin{proof}
Because $R_{\vide}^{\square}$ is the power-series ring $\Hsm[[Y_{1},Y_{2},Y_{3}]]$, if in the isomorphism
\begin{equation}\label{EqObjectifSquare}
\zs:\Ds(\Ts)^{-1}\simeq\Det_{\Hs}[y\Hsm\fleche x\Hsm]
\end{equation}
induced by $\zs$, $y$ cannot be chosen invertible in $\Hsm$, then in the trivialization 
\begin{equation}\nonumber
\zs^{\square}:\left(\Delta^{\square}_{\Si}(\Ts)\right)^{-1}\simeq\Det_{R_{\vide}^{\square}}[y R_{\vide}^{\square}\fleche x R_{\vide}^{\square}]
\end{equation}
induced from trivialization \eqref{EqObjectifSquare} by tensor product over $R_{\vide}=\Hsm$ with $R_{\vide}^{\square}$, it is not possible to choose $y$ a unit in $R_{\vide}^{\square}$.

Suppose more generally that $\Sigma$ and $\Sigma'$ are admissible set of primes and let $\Sigma''$ be their union. Let $J$ and $J'$ be two ideals of $\Hsm$ and $\Hsmprime$ respectively such that $\Hsm/J$ and $\Hsmprime/J'$ are isomorphic to a given ring $R$. Fix a choice of $(x,y)\in\Hs$ and $(x',y')\in\Hsmprime$ such that $\zs$ and $\zsprime$ respectively induce isomorphisms
\begin{equation}\nonumber
\zs:\Ds(\Ts)^{-1}\simeq\Det_{\Hsm}[y\Hsm\fleche x\Hsm]
\end{equation}
and
\begin{equation}\nonumber
\zsprime:\Dsprime(\Tsprime)^{-1}\simeq\Det_{\Hsmprime}[y'\Hsmprime\fleche x'\Hsmprime].
\end{equation}
Assume moreover that $x,y$ do not belong to $J$ and that $x',y'$ do not belong to $J'$. For $*$ the symbol $\vide,'$ or $''$, define
\begin{equation}\nonumber
\phi_{\Sigma^{*}}=\frac{y^{*}}{x^{*}}\z_{\Sigma^{*}}.
\end{equation}
Equivalently, this corresponds to the trivialization of
\begin{equation}\nonumber
[y^{*}\Hecke_{\mgot_{\rhobar}}^{\Sigma^{*}}\fleche x^{*}\Hecke_{\mgot_{\rhobar}}^{\Sigma^{*}}]
\end{equation}
obtained by identifying the two terms of the complex. Consider the diagram 
\begin{equation}\label{EqDiagCompXY}
\xymatrix{
&\Dsprime\ar[r]^{\phi_{\Sigma'}}&\Hsmprime\ar[dr]^{\mod J'}\\
\Delta_{\Si''}\ar[ur]^{\pi_{\Sigma'',\Sigma'}}\ar[dr]_{\pi_{\Sigma'',\Sigma}}\ar[r]^{\phi_{\Sigma''}}&\Hecke^{\Sigma''}_{\mr}\ar[rd]\ar[ru]&&R\\
&\Ds\ar[r]^{\phi_{\Sigma}}&\Hsm\ar[ru]_{\mod J}
}
\end{equation}
in which rings are viewed as free modules of rank 1 over themselves. Each arrow then sends a generator of its source to a generator of its target. Consequently, if there is an isomorphism
\begin{equation}\label{EqIsoPrimePrime}
\Delta_{{\Sigma''}}(T_{\Sigma''})^{-1}\simeq\Det_{\Hsmprimeprime}[y''\Hsmprimeprime\fleche x''\Hsmprimeprime]
\end{equation}
in which $y''$ is invertible in $\Hecke^{\Si''}_{\mr}$, then there exists an isomorphism
\begin{equation}\nonumber
\Ds(\Ts)^{-1}\simeq\Det_{\Hsm}[y\Hsm\fleche x\Hsm]
\end{equation}
with $y$ invertible in $\Hsm$. By contraposition, if there is no isomorphism 
\begin{equation}\label{EqIsoEtoile}
\Ds(\Ts)^{-1}\simeq\Det_{\Hsm}[y\Hsm\fleche x\Hsm]
\end{equation}
with $y$ invertible in $\Hsm$, then there is no isomorphism \eqref{EqIsoPrimePrime} with $y''$ invertible.

We now return to the claims of the lemma. The commutativity of diagram \eqref{DiagCommutativeDamPatched} can always be arranged by construction of the set $\{D_{m}\}_{m\in\N}$ of patched rings. If, as in the lemma, we assume in addition that in
\begin{equation}\label{EqObjectifBis}
\zs:\Ds(\Ts)^{-1}\simeq\Det_{\Hs}[y\Hs\fleche x\Hs],
\end{equation}
the element $y$ cannot be chosen invertible in $\Hsm$, then for $N\in\N$ large enough, $x$ and $y$ do not belong to $\mgot_{\Hs}^{N}$. Hence, for $m$ large enough, $x$ and $y$ do not belong to the ideal $\cid_{m}R_{Q(m)}+\mgot_{R_{Q(m)}}^{(r_{m})}$. If two integers $n_{1},n_{2}$ are large enough, the reasoning of the first part of the proof thus applies to the rings $R^{\square}_{n_{1}}$ and $R^{\square}_{n_{2}}$ and to the ideals $\left(\cid_{m_{i}}R_{n_{i}}+\mgot_{R_{n_{i}}}^{(r_{m_{i}})}\right)$ for $i\in\{1,2\}$ and $m_{i}$ as in the definition of the patched system. This yields a sequence $(a_{n})_{n\in\N}$ such that the $\Delta^{\square}_{a_{n}}$ verify the desired properties.
%
%
%The proof of lemma \ref{LemRemonte} then relies exclusively on the compatibility of $\Ds$ with arbitrary change of ring of coefficients so can be reproduced \textit{verbatim}.
\end{proof}
Replacing the Taylor-Wiles-Kisin system $\{B,R^{\square}_{\vide},\Ds^{\square},(R_{n}^{\square},\Delta_{n}^{\square})_{n\in\N}\}$ with the Taylor-Wiles-Kisin system $\{B,R^{\square}_{\vide},\Ds^{\square},(R_{a_{n}}^{\square},\Delta_{a_{n}}^{\square})_{n\in\N}\}$ if necessary, we obtain the existence of an isomorphism
\begin{equation}\nonumber
\left(\Delta_{\infty}^{\square}\right)^{-1}\simeq\Det_{R^{\square}_{\infty}}[y_{\infty}R^{\square}_{\infty}\fleche x_{\infty}R^{\square}_{\infty}]
\end{equation}
with $y_{\infty}$ in the maximal ideal of $R^{\square}_{\infty}$ and $x_{\infty}/y_{\infty}\notin R^{\square}_{\infty}$. 
\begin{LemEnglish}\label{LemDescendBis}
If there is an isomorphism
\begin{equation}\label{EqPointDeDepart}
\left(\Delta_{\infty}^{\square}\right)^{-1}\simeq\Det_{R^{\square}_{\infty}}[y_{\infty}R^{\square}_{\infty}\fleche x_{\infty}R^{\square}_{\infty}]
\end{equation}
with $y_{\infty}$ in the maximal ideal of $R^{\square}_{\infty}$ and $x_{\infty}/y_{\infty}\notin R^{\square}_{\infty}$, then there exists an admissible set of primes $\Sigma'\subset\Si$, a discrete valuation ring finite and flat over $\zp$ and a specialization
$$\psi:\Hsmprime\fleche S_{\Iw}$$
satisfying the following properties.
\begin{enumerate}
\item\label{ItLemDescend1} The specialization factors through a single irreducible component of $\Hsmprime$.
\item\label{ItLemDescend2} The class $\zs(\psi)$ is not zero.
\item\label{ItLemDescend3} The group $H^{2}(G_{\qp},\Tpsi)$ is finite.
\item\label{ItLemDescend4} The morphism $\zs(\psi)$ does not induce an inclusion $\Ds(\Tpsi)^{-1}\plonge S_{\Iw}$.
\end{enumerate}
\end{LemEnglish}
\begin{proof}
Because $R^{\square}_{\infty}$ is reduced, if there exists a trivialization \ref{EqPointDeDepart} with $x_{\infty}/y_{\infty}\notin R^{\square}_{\infty}$, there exist a minimal prime ideal $\aid_{\infty}\in\Spec R_{\infty}^{\square}$ and a trivialization
\begin{equation}\label{EqEtape}
\left(\Delta_{\infty}^{\square}\tenseur_{R_{\infty}^{\square}} R_{\infty}^{\square}/\aid_{\infty}\right)^{-1}\simeq\Det_{R_{\infty}^{\square}/\aid_{\infty}}[\tilde{y}_{\infty}R_{\infty}^{\square}/\aid_{\infty}\fleche \tilde{x}_{\infty}R_{\infty}^{\square}/\aid_{\infty}]
\end{equation}
such that the images $\tilde{x}_{\infty}$ and $\tilde{y}_{\infty}$ (which are both non-zero) do not satisfy that $\tilde{x}_{\infty}/\tilde{y}_{\infty}$ belongs to $R_{\infty}^{\square}/\aid_{\infty}$. Because every irreducible component of $R_{\infty}$ is a regular local ring by proposition \ref{PropTWgeneral}, $R_{\infty}^{\square}/\aid_{\infty}$ is a regular local ring by lemma \ref{LemTW}.
 The ring $R^{\square}_{\infty}/\aid_{\infty}$ is a regular local ring of Krull dimension at least 10 equal to an irreducible component of an inverse limit of a projective system of quotients $R^{\square}_{n}$ with surjective transition maps so it admits a local morphism $\psi_{0}:R^{\square}_{\infty}\fleche S_{0}$ satisfying the following properties.
\begin{enumerate}
\item The ring $S_{0}$ is a principal artinian quotient of a $\zp$-flat discrete valuation ring of residual characteristic $p$.
\item Neither $\psi_{0}(y_{\infty})$ nor $\psi_{0}(x_{\infty})$ are zero and 
\begin{equation}\nonumber
\max\{n\in\N|\psi_{0}(x_{\infty})\in\mgot^{n}_{S_{0}}\}<\max\{n\in\N|\psi_{0}(y_{\infty})\in\mgot^{n}_{S_{0}}\}.
\end{equation}
\item There exists a surjection $R^{\square}_{n}\surjection S_{0}$ for some $n\in\N$.
\suspend{enumerate}
Forgetting the choice of basis induces a surjection $R_{n}\surjection S_{0}$.

Denote by $X$ the formal variable corresponding to cyclotomic deformations of $\rhobar$. Because $R_{n}$ is the suitable universal deformation of $\rhobar$ attached to the supplementary deformation data encoded in the integer $n$ and is a complete intersection ring, there exist formal variables $Y_{1},\cdots,Y_{m}$ and a regular sequence $(f_{1},\cdots,f_{m-2})$ in $\Ocal[[Y_{1},\cdots,Y_{m}]]$ such that
\begin{equation}\nonumber
R_{n}\simeq\Ocal[[X,Y_{1},\cdots,Y_{m}]]/(f_{1},\cdots,f_{m-2}).
\end{equation}
By the smoothness of $\Ocal[[Y_{1},\cdots,Y_{m}]]$, there exists a complete intersection ring $S_{1}$ of relative dimension zero over $\Ocal$ and a morphism $\psi_{1}:R_{n}\fleche S_{1}[[X]]$ such that the diagram
\begin{equation}\nonumber
\xymatrix{
&S_{1}[[X]]\ar[d]^{\pi}\\
R_{n}\ar[r]^{\psi_{0}}\ar[ru]^{\psi_{1}}&S_{0}
}
\end{equation}
is commutative. Denote by $S$ the normalization of $S_{1}$ and denote by $\psi$ the natural morphism $\psi:R_{n}\fleche S_{\Iw}$. Suppose that the image of $y_{n}$ in $S_{\Iw}$ divides the image of $x_{n}$ in $S_{\Iw}$. Then $\psi_{0}(y_{\infty})=\pi\circ\psi(y_{n})$ divides $\psi_{0}(x_{\infty})=\pi\circ\psi(x_{n})$; in contradiction with the definition of $S_{0}$. In particular, for any choice of $(x_{\psi},y_{\psi})\in S_{\Iw}^{2}$ such that there is an isomorphism \begin{equation}\nonumber
\zs(\psi)\Ds(\Tpsi)^{-1}\simeq\Det_{S_{\Iw}}[y_{\psi}S\fleche x_{\psi}S],
\end{equation}
the map $\zs(\psi)$ does not induce an inclusion $\Delta_{S_{\Iw}}(\Tpsi)^{-1}\plonge S_{\Iw}$.

It remains to show that there exists such a $\psi$ satisfying in addition the requirements \ref{ItLemDescend1}, \ref{ItLemDescend2} and \ref{ItLemDescend3} of the lemma. The union of $S_{\Iw}$-valued points $\psi$ of $\Spec\Hsmprime$ such that either $\zs(\psi)$ is a divisor of zero, or $H^{2}(G_{\qp},\Tpsi)$ is infinite or $\psi$ factors through several distinct irreducible components is contained in a subscheme of codimension greater than 1 whereas the requirement \ref{ItLemDescend4} is open. Hence, there exists $\psi$ as in the statement of the lemma.
\end{proof}
\begin{proof}[Proof of theorem 4.1.1]
Assume by way of contradiction that $\zs$ does not induce an inclusion $\Ds(\Ts)^{-1}\plonge\Hsm$. According to lemma \ref{LemRemonteBis} and \ref{LemDescendBis}, there then exists an admissible set $\Sigma'$, a discrete valuation ring $S$ finite and flat over $\zp$ and a specialization $\psi:\Hsmprime\fleche S_{\Iw}$ such that $\zs({\psi})$ is not zero, the group $H^{2}(G_{\qp},\Tpsi)$ is finite and $\zs(\psi)$ does not induce an inclusion $\Delta_{S_{\Iw}}(\Tpsi)^{-1}\plonge S_{\Iw}$. 

By assumption, the image of $\rhobar$ acts irreducibly on $\Fpbar_{p}^{2}$ and is of order divisible by $p$ so contains a subgroup conjugate to $\SL_{2}(\Fp_{q})$ for some $q=p^{n}$ and hence an element $\bar{\s}\neq\Id$ which is unipotent. Let $\s$ be a lift of $\bar{\s}$ to $\rho_{\psi}(G_{\Q})$. Then the kernel of $\s-1$ is strictly included in $\Tpsi$ and its cokernel is dimension 1 after tensor product with $S/\mgot_{S}$. Hence, the cokernel of $\s-1$ is not torsion and is generated by a single element, so it is free of rank 1. Denote by $X$ the cyclotomic variable, so that $S_{\Iw}=S[[X]]$. The representation $\Tpsi/X\Tpsi$ satisfies hypotheses $(\operatorname{i_{{str}}})$, $\operatorname{(ii_{{str}}})$ and $(\operatorname{iv}_{\pid})$ of \cite[Theorem 0.8]{KatoEulerOriginal}. As $\Tpsi/X\Tpsi$ is absolutely irreducible and not abelian, the conclusion of \cite[Proposition 8.7]{KatoEulerOriginal} also holds. As establishing this proposition is the sole function of hypothesis $(\operatorname{iii})$ of \cite[Theorem 0.8]{KatoEulerOriginal}, the conclusion of this theorem holds for $\Tpsi/X\Tpsi$. As $H^{2}(G_{\qp},\Tpsi)$ is finite, its localization any $\pid\in\Spec S_{\Iw}$ of height 1 vanishes and so there is a canonical isomorphism between its determinant and $S_{\Iw}$. Consequently, the isomorphism
\begin{equation}\nonumber
\zs(\psi):\Ds(T_{\psi^{}})\tenseur\Frac(S_{\Iw})\isom\Frac(S_{\Iw})
\end{equation}
induces and inclusion 
\begin{equation}\nonumber
\zs(\psi^{}):\Ds(T_{\psi^{}})^{-1}\plonge S_{\Iw},
\end{equation}
in contradiction with the defining properties of $\psi$.

Hence, there exists $x\in\Hsmr$ such that the zeta morphism 
\begin{equation}\nonumber
\zs:\Ds(\Ts)\tenseur Q(\Hsmr)\isom Q(\Hsmr)
\end{equation}
induces an isomorphism
\begin{equation}\nonumber
\zs:\Ds(\Ts)^{-1}\isom x\Hsmr.
\end{equation} 
All the equivalences of the theorem then follows easily from the commutativity of the diagram
\begin{equation}\nonumber
\xymatrix{
\ar[d]_{-\tenseur_{\Hsmr,\la}\Oiwa}\Ds(\Ts)^{-1}\ar[r]^(0.6){\zs}&x\Hsmr\ar[d]^{\la}\\
\Delta_{\Oiwa}(f)\ar[r]_{\z(f)_{\Iw}}&\la(x)\Oiwa
}
\end{equation}
for all motivic point $\la:\Hsmr\fleche\Oiwa$.
\end{proof}
\paragraph{Acknowledgments}The author acknowledges financial support from ANER project MethodParithé and ANR project PadLEfAn. It is a pleasure to thank Chris Skinner and Xin Wan for useful remarks on this manuscript, as well as the anonymous referee(s) for their careful reading.

\def\Dbar{\leavevmode\lower.6ex\hbox to 0pt{\hskip-.23ex \accent"16\hss}D}
  \def\cfac#1{\ifmmode\setbox7\hbox{$\accent"5E#1$}\else
  \setbox7\hbox{\accent"5E#1}\penalty 10000\relax\fi\raise 1\ht7
  \hbox{\lower1.15ex\hbox to 1\wd7{\hss\accent"13\hss}}\penalty 10000
  \hskip-1\wd7\penalty 10000\box7}
  \def\cftil#1{\ifmmode\setbox7\hbox{$\accent"5E#1$}\else
  \setbox7\hbox{\accent"5E#1}\penalty 10000\relax\fi\raise 1\ht7
  \hbox{\lower1.15ex\hbox to 1\wd7{\hss\accent"7E\hss}}\penalty 10000
  \hskip-1\wd7\penalty 10000\box7} \def\Dbar{\leavevmode\lower.6ex\hbox to
  0pt{\hskip-.23ex \accent"16\hss}D}
  \def\cfac#1{\ifmmode\setbox7\hbox{$\accent"5E#1$}\else
  \setbox7\hbox{\accent"5E#1}\penalty 10000\relax\fi\raise 1\ht7
  \hbox{\lower1.15ex\hbox to 1\wd7{\hss\accent"13\hss}}\penalty 10000
  \hskip-1\wd7\penalty 10000\box7}
  \def\cftil#1{\ifmmode\setbox7\hbox{$\accent"5E#1$}\else
  \setbox7\hbox{\accent"5E#1}\penalty 10000\relax\fi\raise 1\ht7
  \hbox{\lower1.15ex\hbox to 1\wd7{\hss\accent"7E\hss}}\penalty 10000
  \hskip-1\wd7\penalty 10000\box7} \def\Dbar{\leavevmode\lower.6ex\hbox to
  0pt{\hskip-.23ex \accent"16\hss}D}
  \def\cfac#1{\ifmmode\setbox7\hbox{$\accent"5E#1$}\else
  \setbox7\hbox{\accent"5E#1}\penalty 10000\relax\fi\raise 1\ht7
  \hbox{\lower1.15ex\hbox to 1\wd7{\hss\accent"13\hss}}\penalty 10000
  \hskip-1\wd7\penalty 10000\box7}
  \def\cftil#1{\ifmmode\setbox7\hbox{$\accent"5E#1$}\else
  \setbox7\hbox{\accent"5E#1}\penalty 10000\relax\fi\raise 1\ht7
  \hbox{\lower1.15ex\hbox to 1\wd7{\hss\accent"7E\hss}}\penalty 10000
  \hskip-1\wd7\penalty 10000\box7} \def\Dbar{\leavevmode\lower.6ex\hbox to
  0pt{\hskip-.23ex \accent"16\hss}D}
  \def\cfac#1{\ifmmode\setbox7\hbox{$\accent"5E#1$}\else
  \setbox7\hbox{\accent"5E#1}\penalty 10000\relax\fi\raise 1\ht7
  \hbox{\lower1.15ex\hbox to 1\wd7{\hss\accent"13\hss}}\penalty 10000
  \hskip-1\wd7\penalty 10000\box7}
  \def\cftil#1{\ifmmode\setbox7\hbox{$\accent"5E#1$}\else
  \setbox7\hbox{\accent"5E#1}\penalty 10000\relax\fi\raise 1\ht7
  \hbox{\lower1.15ex\hbox to 1\wd7{\hss\accent"7E\hss}}\penalty 10000
  \hskip-1\wd7\penalty 10000\box7} \def\Dbar{\leavevmode\lower.6ex\hbox to
  0pt{\hskip-.23ex \accent"16\hss}D}
  \def\cfac#1{\ifmmode\setbox7\hbox{$\accent"5E#1$}\else
  \setbox7\hbox{\accent"5E#1}\penalty 10000\relax\fi\raise 1\ht7
  \hbox{\lower1.15ex\hbox to 1\wd7{\hss\accent"13\hss}}\penalty 10000
  \hskip-1\wd7\penalty 10000\box7}
  \def\cftil#1{\ifmmode\setbox7\hbox{$\accent"5E#1$}\else
  \setbox7\hbox{\accent"5E#1}\penalty 10000\relax\fi\raise 1\ht7
  \hbox{\lower1.15ex\hbox to 1\wd7{\hss\accent"7E\hss}}\penalty 10000
  \hskip-1\wd7\penalty 10000\box7} \def\cprime{$'$}
  \def\cftil#1{\ifmmode\setbox7\hbox{$\accent"5E#1$}\else
  \setbox7\hbox{\accent"5E#1}\penalty 10000\relax\fi\raise 1\ht7
  \hbox{\lower1.15ex\hbox to 1\wd7{\hss\accent"7E\hss}}\penalty 10000
  \hskip-1\wd7\penalty 10000\box7}

%\newpage
\setlength{\parindent}{0pt}%
\bigskip{\footnotesize%
  \textrm{Olivier Fouquet} \par
  \textsc{Laboratoire de mathématiques de Besan\c con}
  \par
  \textsc{16, route de Gray 25000 Besan\c con}
  \par
  \textsc{France} \par  
  \textit{E-mail address}: \texttt{olivier.fouquet@univ-fcomte.fr} \par


\begin{thebibliography}{10}

\bibitem{BellaicheEigen}
Jo\"{e}l Bella\"{\i}che.
\newblock Critical {$p$}-adic {$L$}-functions.
\newblock {\em Invent. Math.}, 189(1):1--60, 2012.

\bibitem{BellaicheRank}
Jo\"{e}l Bella\"{\i}che.
\newblock Ranks of {S}elmer groups in an analytic family.
\newblock {\em Trans. Amer. Math. Soc.}, 364(9):4735--4761, 2012.

\bibitem{BellaicheChenevier}
Jo{\"e}l Bella{\"{\i}}che and Ga{\"e}tan Chenevier.
\newblock Families of {G}alois representations and {S}elmer groups.
\newblock {\em Ast\'erisque}, (324):xii+314, 2009.

\bibitem{BlochKato}
Spencer Bloch and Kazuya Kato.
\newblock {$L$}-functions and {T}amagawa numbers of motives.
\newblock In {\em The {G}rothendieck {F}estschrift, Vol.\ I}, volume~86 of {\em
  Progr. Math.}, pages 333--400. Birkh\"auser Boston, Boston, MA, 1990.

\bibitem{BockleDemuskin}
Gebhard B{\"o}ckle.
\newblock Demu\v skin groups with group actions and applications to
  deformations of {G}alois representations.
\newblock {\em Compositio Math.}, 121(2):109--154, 2000.

\bibitem{BurnsFlachMotivic}
David Burns and Matthias Flach.
\newblock Motivic ${L}$-functions and {G}alois module structures.
\newblock {\em Math. Annalen}, 305:65--102, 1996.

\bibitem{ColmezBSD}
Pierre Colmez.
\newblock La conjecture de {B}irch et {S}winnerton-{D}yer {$p$}-adique.
\newblock {\em Ast\'erisque}, (294):ix, 251--319, 2004.

\bibitem{ColmezFoncteur}
Pierre Colmez.
\newblock Repr\'{e}sentations de {${\rm GL}_2(\bold Q_p)$} et
  {$(\phi,\Gamma)$}-modules.
\newblock {\em Ast\'{e}risque}, (330):281--509, 2010.

\bibitem{DeligneFonctionsL}
Pierre Deligne.
\newblock Valeurs de fonctions {$L$} et p\'eriodes d'int\'egrales.
\newblock In {\em Automorphic forms, representations and $L$-functions (Proc.
  Sympos. Pure Math., Oregon State Univ., Corvallis, Ore., 1977), Part 2},
  Proc. Sympos. Pure Math., XXXIII, pages 313--346. Amer. Math. Soc.,
  Providence, R.I., 1979.
\newblock With an appendix by N. Koblitz and A. Ogus.

\bibitem{DiamondFlachGuo}
Fred Diamond, Matthias Flach, and Li~Guo.
\newblock The {T}amagawa number conjecture of adjoint motives of modular forms.
\newblock {\em Ann. Sci. \'Ecole Norm. Sup. (4)}, 37(5):663--727, 2004.

\bibitem{EmertonPollackWeston}
Matthew Emerton, Robert Pollack, and Tom Weston.
\newblock Variation of {I}wasawa invariants in {H}ida families.
\newblock {\em Invent. Math.}, 163(3):523--580, 2006.

\bibitem{FontaineValeursSpeciales}
Jean-Marc Fontaine.
\newblock Valeurs sp\'eciales des fonctions {$L$} des motifs.
\newblock {\em Ast\'erisque}, (206):Exp.\ No.\ 751, 4, 205--249, 1992.
\newblock S{\'e}minaire Bourbaki, Vol. 1991/92.

\bibitem{FouquetX}
Olivier Fouquet.
\newblock {$p$}-adic properties of motivic fundamental lines.
\newblock {\em J. \'Ec. polytech. Math.}, 4:37--86, 2017.

\bibitem{FouquetPMB}
Olivier Fouquet.
\newblock Congruences and the {I}wasawa {M}ain {C}onjecture for modular forms.
\newblock In {\em Alg\`ebre et th\'eorie des nombres.}, Publ. Math. Univ.
  Franche-Comt\'e Besan\c con Alg\`ebr. Theor. Nr., page~17. Lab. Math. Besan\c
  con, Besan\c con, 2024.
\newblock to appear.

\bibitem{FouquetOchiai}
Olivier Fouquet and Tadashi Ochiai.
\newblock Control theorems for {S}elmer groups of nearly ordinary deformations.
\newblock {\em J. reine angew. Math.}, 666:163--187, 2012.

\bibitem{FouquetWan}
Olivier Fouquet and Xin Wan.
\newblock The {I}wasawa {M}ain {C}onjecture for modular motives.
\newblock soumis, disponible sur arxiv 2107.13726, 2022.

\bibitem{FujiwaraDeformation}
Kazuhiro Fujiwara.
\newblock {D}eformation rings and {H}ecke algebras in the totally real case,
  1999.
\newblock {P}reprint, 99pp.

\bibitem{FukayaKato}
Takako Fukaya and Kazuya Kato.
\newblock A formulation of conjectures on {$p$}-adic zeta functions in
  noncommutative {I}wasawa theory.
\newblock In {\em Proceedings of the {S}t. {P}etersburg {M}athematical
  {S}ociety. {V}ol. {XII}}, volume 219 of {\em Amer. Math. Soc. Transl. Ser.
  2}, pages 1--85, Providence, RI, 2006. Amer. Math. Soc.

\bibitem{GouveaMazur}
Fernando~Q. Gouv{\^e}a and Barry Mazur.
\newblock On the density of modular representations.
\newblock In {\em Computational perspectives on number theory ({C}hicago, {IL},
  1995)}, volume~7 of {\em AMS/IP Stud. Adv. Math.}, pages 127--142. Amer.
  Math. Soc., Providence, RI, 1998.

\bibitem{GreenbergControl}
Ralph Greenberg.
\newblock Galois theory for the {S}elmer group of an abelian variety.
\newblock {\em Compositio Math.}, 136(3):255--297, 2003.

\bibitem{GrossZagier}
Benedict~H. Gross and Don~B. Zagier.
\newblock Heegner points and derivatives of {$L$}-series.
\newblock {\em Invent. Math.}, 84(2):225--320, 1986.

\bibitem{HidaInventionesOrdinary}
Haruzo Hida.
\newblock Galois representations into {${\rm GL}\sb 2({\bf Z}\sb p[[X]])$}
  attached to ordinary cusp forms.
\newblock {\em Invent. Math.}, 85(3):545--613, 1986.

\bibitem{HidaNearlyOrdinary}
Haruzo Hida.
\newblock On nearly ordinary {H}ecke algebras for {${\rm GL}(2)$} over totally
  real fields.
\newblock In {\em Algebraic number theory}, volume~17 of {\em Adv. Stud. Pure
  Math.}, pages 139--169. Academic Press, Boston, MA, 1989.

\bibitem{IwasawaMainConjecture}
Kenkichi Iwasawa.
\newblock Analogies between number fields and function fields.
\newblock In {\em Some {R}ecent {A}dvances in the {B}asic {S}ciences, {V}ol. 2
  ({P}roc. {A}nnual {S}ci. {C}onf., {B}elfer {G}rad. {S}chool {S}ci., {Y}eshiva
  {U}niv., {N}ew {Y}ork, 1965-1966)}, pages 203--208. Belfer Graduate School of
  Science, Yeshiva Univ., New York, 1969.

\bibitem{JetchevSkinnerWan}
Dimitar Jetchev, Christopher Skinner, and Xin Wan.
\newblock The {B}irch and {S}winnerton-{D}yer formula for elliptic curves of
  analytic rank one.
\newblock {\em Camb. J. Math.}, 5(3):369--434, 2017.

\bibitem{KatoHodgeIwasawa}
Kazuya Kato.
\newblock Iwasawa theory and {$p$}-adic {H}odge theory.
\newblock {\em Kodai Math. J.}, 16(1):1--31, 1993.

\bibitem{KatoViaBdR}
Kazuya Kato.
\newblock Lectures on the approach to {I}wasawa theory for {H}asse-{W}eil
  {$L$}-functions via {$B\sb {\rm dR}$}. {I}.
\newblock In {\em Arithmetic algebraic geometry (Trento, 1991)}, volume 1553 of
  {\em Lecture Notes in Math.}, pages 50--163. Springer, Berlin, 1993.

\bibitem{KatoEulerOriginal}
Kazuya Kato.
\newblock Euler systems, {I}wasawa theory, and {S}elmer groups.
\newblock {\em Kodai Math. J.}, 22(3):313--372, 1999.

\bibitem{KatoEuler}
Kazuya Kato.
\newblock {$p$}-adic {H}odge theory and values of zeta functions of modular
  forms.
\newblock {\em Ast\'erisque}, (295):ix, 117--290, 2004.
\newblock Cohomologies $p$-adiques et applications arithm\'etiques. III.

\bibitem{KatoICM}
Kazuya Kato.
\newblock Iwasawa theory and generalizations.
\newblock In {\em International {C}ongress of {M}athematicians. {V}ol. {I}},
  pages 335--357. Eur. Math. Soc., Z\"urich, 2007.

\bibitem{KisinFlat}
Mark Kisin.
\newblock Moduli of finite flat group schemes, and modularity.
\newblock {\em Ann. of Math. (2)}, 170(3):1085--1180, 2009.

\bibitem{KobayashiIMC}
Shin-ichi Kobayashi.
\newblock Iwasawa theory for elliptic curves at supersingular primes.
\newblock {\em Invent. Math.}, 152(1):1--36, 2003.

\bibitem{KolyvaginEuler}
Viktor Kolyvagin.
\newblock Euler systems.
\newblock In {\em The {G}rothendieck {F}estschrift, Vol.\ II}, volume~87 of
  {\em Progr. Math.}, pages 435--483. Birkh\"auser Boston, Boston, MA, 1990.

\bibitem{KubotaLeopoldt}
Tomio Kubota and Heinrich-Wolfgang Leopoldt.
\newblock Eine {$p$}-adische {T}heorie der {Z}etawerte. {I}. {E}inf\"uhrung der
  {$p$}-adischen {D}irichletschen {$L$}-{F}unktionen.
\newblock {\em J. reine angew. Math.}, 214/215:328--339, 1964.

\bibitem{ManinPadic}
Ju.~I. Manin.
\newblock Non-{A}rchimedean integration and {$p$}-adic {J}acquet-{L}anglands
  {$L$}-functions.
\newblock {\em Uspehi Mat. Nauk}, 31(1(187)):5--54, 1976.

\bibitem{MazurValues}
Barry Mazur.
\newblock On the arithmetic of special values of {$L$} functions.
\newblock {\em Invent. Math.}, 55(3):207--240, 1979.

\bibitem{MazurSwinnertonDyer}
Barry Mazur and Peter Swinnerton-Dyer.
\newblock Arithmetic of {W}eil curves.
\newblock {\em Invent. Math.}, 25:1--61, 1974.

\bibitem{MazurTateTeitelbaum}
Barry Mazur, J.~Tate, and J.~Teitelbaum.
\newblock On {$p$}-adic analogues of the conjectures of {B}irch and
  {S}winnerton-{D}yer.
\newblock {\em Invent. Math.}, 84(1):1--48, 1986.

\bibitem{MazurWiles}
Barry Mazur and Andrew Wiles.
\newblock Class fields of abelian extensions of {${\bf Q}$}.
\newblock {\em Invent. Math.}, 76(2):179--330, 1984.

\bibitem{NakamuraUniversal}
Kentaro Nakamura.
\newblock Zeta morphisms for rank two universal deformations.
\newblock {\em Invent. Math.}, 234(1):171--290, 2023.

\bibitem{SelmerComplexes}
Jan Nekov{\'a}{\v{r}}.
\newblock Selmer complexes.
\newblock {\em Ast\'erisque}, (310):559, 2006.

\bibitem{OchiaiMainConjecture}
Tadashi Ochiai.
\newblock On the two-variable {I}wasawa main conjecture.
\newblock {\em Compos. Math.}, 142(5):1157--1200, 2006.

\bibitem{PaskunasMontreal}
Vytautas Pa\v{s}k\={u}nas.
\newblock The image of {C}olmez's {M}ontreal functor.
\newblock {\em Publ. Math. Inst. Hautes \'{E}tudes Sci.}, 118:1--191, 2013.

\bibitem{PerrinRiouLpadique}
Bernadette Perrin-Riou.
\newblock Fonctions {$L$} {$p$}-adiques des repr\'esentations {$p$}-adiques.
\newblock {\em Ast\'erisque}, (229):198, 1995.

\bibitem{PollackSupersingular}
Robert Pollack.
\newblock On the {$p$}-adic {$L$}-function of a modular form at a supersingular
  prime.
\newblock {\em Duke Math. J.}, 118(3):523--558, 2003.

\bibitem{RamakrishnaFlat}
Ravi Ramakrishna.
\newblock On a variation of {M}azur's deformation functor.
\newblock {\em Compositio Math.}, 87(3):269--286, 1993.

\bibitem{Rohrlich}
David~E. Rohrlich.
\newblock Nonvanishing of {$L$}-functions for {${\rm GL}(2)$}.
\newblock {\em Invent. Math.}, 97(2):381--403, 1989.

\bibitem{SchollMotivesModular}
Anthony~J. Scholl.
\newblock Motives for modular forms.
\newblock {\em Invent. Math.}, 100(2):419--430, 1990.

\bibitem{ShottonDeformation}
Jack Shotton.
\newblock Local deformation rings for {${\rm GL}_2$} and a {B}reuil-{M}\'ezard
  conjecture when {$\ell\ne p$}.
\newblock {\em Algebra Number Theory}, 10(7):1437--1475, 2016.

\bibitem{SkinnerWilesDur}
C.~M. Skinner and A.~J. Wiles.
\newblock Residually reducible representations and modular forms.
\newblock {\em Inst. Hautes \'Etudes Sci. Publ. Math.}, (89):5--126 (2000),
  1999.

\bibitem{SkinnerUrban}
Christopher Skinner and Eric Urban.
\newblock The {I}wasawa main conjectures for {$\rm GL\sb 2$}.
\newblock {\em Invent. Math.}, 195(1):1--277, 2014.

\bibitem{SprungIMC}
Florian Sprung.
\newblock Iwasawa theory for elliptic curves at supersingular primes: a pair of
  main conjectures.
\newblock {\em J. Number Theory}, 132(7):1483--1506, 2012.

\bibitem{TaylorWiles}
Richard Taylor and Andrew Wiles.
\newblock Ring-theoretic properties of certain {H}ecke algebras.
\newblock {\em Ann. of Math. (2)}, 141(3):553--572, 1995.

\bibitem{PARI2}
{The PARI~Group}, Univ. Bordeaux.
\newblock {\em {PARI/GP version \texttt{2.17.0}}}, 2024.

\bibitem{VenjakobETNC}
Otmar Venjakob.
\newblock From the {B}irch and {S}winnerton-{D}yer {C}onjecture to
  non-commutative {I}wasa theory via the {E}quivariant {T}amagawa {N}umber
  {C}onjecture-a survey.
\newblock In {\em $L$-functions and Galois representations (Durham, July
  2004)}, volume 320 of {\em London Math. Soc. Lecture Note Ser.}, pages
  333--380. Cambridge Univ. Press, Cambridge, 2004.

\bibitem{VisikPadic}
M.M. {Visik}.
\newblock {Non-Archimedean measures connected with Dirichlet series.}
\newblock {\em {Math. USSR, Sb.}}, 28:216--228, 1976.

\end{thebibliography}
\end{document}